%% file: GeneralQ.tex
\documentclass[10pt]{articleHJ}

\usepackage[centertags]{amsmath}
\usepackage{amsfonts,amssymb,amsthm}
\usepackage{caption,graphicx}
\usepackage{subcaption}
\usepackage{mathrsfs}
\usepackage{tikz}
\usepackage{url}
\usepackage[text={6.0in,8.6in},centering,letterpaper]{geometry}
\usepackage[numbers,comma,square,sort&compress]{natbib}

\usepackage{mathtools}

\newlength{\defbaselineskip}
\newcommand{\setlinespacing}[1]%
           {\setlength{\baselineskip}{#1 \defbaselineskip}}

\newcommand{\singlespacing}{\setlength{\baselineskip}{\defbaselineskip}}

\newcommand{\R}{\ensuremath{\mathbb{R}}}

\newcommand{\e}{\ensuremath{\varepsilon}}
\newcommand{\p}{\ensuremath{\partial}}

\newcommand{\g}{\ensuremath{\gamma}}
\newcommand{\G}{\ensuremath{\Gamma}}
\newcommand{\s}{\ensuremath{\sigma}}

\renewcommand{\l}{\ensuremath{\lambda}}
\renewcommand{\a}{\ensuremath{\alpha}}

\newcommand{\psit}{\ensuremath{\psi_{\mathrm{tw}}}}

\renewcommand{\L}{\ensuremath{\mathcal{L}}}
\renewcommand{\O}{\ensuremath{\mathcal{O}}}

\DeclarePairedDelimiter{\ip}\langle\rangle
\DeclarePairedDelimiter{\bip}{\big\langle}{\big\rangle}
\DeclarePairedDelimiter{\nrm}\lVert\rVert

\def\Re{\mathop\mathrm{Re}\nolimits}			
\newcommand{\Real}{\mathbb{R}}							
\newcommand{\abs}[1]{\left\vert#1\right\vert}			
\newcommand{\norm}[1]{\left\Vert#1\right\Vert}		
\newcommand{\sref}[1]{(\ref{#1})}                       

\newtheorem{thm}{Theorem}[section]
\newtheorem{cor}[thm]{Corollary}
\newtheorem{lem}[thm]{Lemma}
\newtheorem{prop}[thm]{Proposition}

\newtheorem*{defn*}{Definition}

 {\begin{trivlist}\item[]\textbf{Acknowledgments }}{\end{trivlist}}

\begin{document}

\begin{frontmatter}
\title{Travelling Waves \\ for Reaction-Diffusion Equations \\
  Forced by Translation Invariant Noise}
\journal{Physica D}

\author[LD1]{C. H. S. Hamster\corauthref{coraut}},
\corauth[coraut]{Corresponding author. }
\author[LD2]{H. J. Hupkes},
\address[LD1]{
  Mathematisch Instituut - Universiteit Leiden \\
  P.O. Box 9512; 2300 RA Leiden; The Netherlands \\
 Email:  {\normalfont{\texttt{c.h.s.hamster@math.leidenuniv.nl}}}
}
\address[LD2]{
  Mathematisch Instituut - Universiteit Leiden \\
  P.O. Box 9512; 2300 RA Leiden; The Netherlands \\ Email:  {\normalfont{\texttt{hhupkes@math.leidenuniv.nl}}}
}

\date{\today}

\begin{abstract}
\singlespacing
Inspired by applications, we consider reaction-diffusion equations on $\R$
that are stochastically forced by a small multiplicative noise term that is white in time, coloured in space and invariant
under translations. We show how these equations can be understood as a stochastic partial differential equation (SPDE) forced by a cylindrical Q-Wiener process and subsequently explain how to study stochastic travelling waves in this setting. In particular, we generalize the phase tracking framework that was developed in \cite{Hamster2017,Hamster2018}
for noise processes driven by a single Brownian motion.
The main focus lies on explaining how this framework naturally leads to long term approximations for the stochastic wave profile and speed.
We illustrate our approach by two fully worked-out examples, which
highlight the predictive power of our expansions.
\end{abstract}



\end{frontmatter}

\numberwithin{equation}{section}
\numberwithin{figure}{section}
\renewcommand{\theequation}{\thesection.\arabic{equation}}

\section{Introduction}
\label{sec:int}

In this paper we set out to
study the propagation of wave solutions to
stochastic equations of the form
\begin{align}
\label{eq:NagumoPhysics}
\begin{split}
u_t&=\rho \p_{xx} u +f(u) +\s g(u)\xi(x,t),\\
\end{split}
\end{align}
in which $\xi$ is a Gaussian process\footnote{
Actually `generalised Gaussian' would be a more
accurate term, since we will see that $\xi$
does not have the right properties
to be a Gaussian random variable
on $L^2(\R)$.}
that is white in time and coloured in space.
In particular, we assume formally that
\begin{align}
\label{eq:int:exp:xi}
\begin{split}
    E[\xi(x,t)]&=0,\\
    E[\xi(x,t)\xi(x',t')]&= \delta(t-t')q(x-x'),
    \end{split}
\end{align}
for some smooth covariance function $q$
that describes the correlation in space.
Such equations have been used in a wide range of applications, for example
to model the appearance of travelling waves in light-sensitive Belousov-Zhabotinsky chemical reactions \cite{Kadar1998}
or to study the excitability of activator-inhibitor systems
such as nerve fibres \cite{Garcia2001}.
We refer to \cite[\S 6.1]{GarciaSpatiallyExtended}
for an extended list of examples.

We will assume here that \sref{eq:NagumoPhysics}
with $\sigma = 0$ admits a spectrally stable deterministic travelling front or pulse and examine
the impact of the multiplicative noise term for small $\sigma$. The nonlinearity
$g$ will be chosen to vanish at the
endpoints of the deterministic wave.
In particular, the full stochastic system
is at rest whenever the deterministic
portion is at rest.
Such an assumption is typically
used to examine the distortions
on a system caused by
\textit{external} random effects,
such as fluctuations in the
intensity of the light driving a Belousov-Zhabotinsky reaction.
In a controlled setting these effects
can often be minimized or switched off completely, leading to the notion of
the deterministic limit.

On the other hand,
\textit{internal} fluctuations
arise from the microscopic
properties of the system itself
and cannot be readily eliminated.
For example, the vibrations of the individual atoms are essential
ingredients in the derivation
of the ideal gas equations. It is more natural to use an additive noise term to model effects of this type, but we do not focus on this case here.

The main goal of this paper is to uncover the corrections to the deterministic wave that are caused by the (small)
multiplicative noise term. In particular, we develop a framework that allows the corrections
to the speed and shape of the wave to be computed to any desired order in $\sigma$. We explicitly
compute the second and third order correction terms and show numerically that these expansions are valid
for long time scales. We also outline in which sense these predictions can be made rigorous,
which involves casting the translationally invariant Stochastic Partial Differential Equation (SPDE) \sref{eq:NagumoPhysics}
into a mathematically precise form.

\paragraph{Example I: The Nagumo equation}
In order to set the stage, let us consider the stochastic Nagumo equation
\begin{align}\label{eq:ExNag}
\begin{split}
u_t&= \rho \p_{xx} u + f_{\mathrm{cub}}(u; \a ) +\sigma u(1-u)\xi(x,t),\\
\end{split}
\end{align}
with the bistable cubic nonlinearity
\begin{equation}
f_{\mathrm{cub}}(u; \a) = u(1-u)(u-\a), \qquad \qquad 0 < \a < 1
\end{equation}
and the
 Gaussian covariance kernel
\begin{equation}
\label{eq:int:q:gaussian}
q(x) =\frac{1}{2}e^{\frac{-\pi x^2}{4}}.
\end{equation}
For $0< a < 1$, the deterministic system is known to have a spectrally stable wave
solution $u(x,t) = \Phi_0(x  - c_0 t)$
that connects the stable rest states zero and one \cite{Kapitula}. In fact, the travelling wave ODE
\begin{align}
    \rho\Phi_0''+c_0\Phi_0'+f_{\mathrm{cub}}(\Phi_0;\a)=0
\end{align}
can be solved by using the explicit expressions
\begin{equation}
\label{eq:int:formulas:det:nag:wave}
c_0 = \sqrt{2\rho_0}\left(\frac{1}{2}-\a_0\right), \qquad
\Phi_0(x) = \frac{1}{2}\left[1-\tanh\left(\frac{1}{2\sqrt{2\rho_0}}x\right)\right]
\end{equation}
with $(\rho_0, \a_0) = (\rho, \a)$.

Note that the form of the nonlinear terms in \sref{eq:ExNag}
allows us to recast the system as
\begin{align}\label{eq:ExNag:par}
\begin{split}
u_t&= \rho_0 \p_{xx} u + f_{\mathrm{cub}}\big(u; \a_0- \s\xi(x,t)   \big),\\
\end{split}
\end{align}
showing that we are stochastically forcing the (external) parameter $\a$.
As a consequence, it is natural to ask whether effective $\sigma$-dependent parameters
$(\rho_\sigma, \a_\sigma)$ can be derived that
are able to capture the stochastic effects on the waves
by replacing \sref{eq:int:formulas:det:nag:wave}
with
\begin{equation}
\label{eq:int:formulas:det:nag:wave:sigma}
c_\s = \sqrt{2\rho_\s}\left(\frac{1}{2}-\a_\s \right), \qquad
\Phi_\s(x) = \frac{1}{2}\left[1-\tanh\left(\frac{1}{2\sqrt{2\rho_\s}}x\right)\right].
\end{equation}
In the case where $\xi(x,t)$ is replaced by the derivative of a single $x$-independent Brownian motion in time,
this point of view can be made fully explicit and precise. Indeed,
in this case the wave \sref{eq:int:formulas:det:nag:wave:sigma}
with 
\begin{equation}
\label{eq:int:wave:steepening}
    (\rho_\sigma, \a_\sigma) = (\frac{\rho_0}{1+\s^2\rho_0},\a_0)
\end{equation}
is an exact solution
to the underlying SPDE \cite{Hamster2017,Cartwright2018},
with a phase that follows a scaled Brownian motion.

The early results in \cite{GarciaSpatiallyExtended} for general $\xi$ can
also be seen in this light. Indeed,
the authors use a formal (partial) expansion to suggest the choices
\begin{equation}
\rho_{\sigma} =\frac{1}{1-\s^2q(0)},
\qquad \a_{\sigma} = \frac{2\a_0-\s^2q(0)}{2-2\s^2q(0)}
\end{equation}
where $\rho_0=1$.
However, the waves found in this way are only approximate solutions\footnote{Note that these scalings hold for the Stratonovich interpretation, while the results from \cite{Hamster2017,Cartwright2018} hold for the It\^o interpretation.}
to \sref{eq:ExNag}. We show in \S\ref{sec:resN} how our techniques
can be used to significantly improve the quality
of this approximation.

In order to discuss the stability of these waves,
we introduce the linearised operator
\begin{align}
    \L_\mathrm{tw}v=\rho v''+c_0v'+f_{\mathrm{cub}}'(\Phi_0;\a)v,
\end{align}
together with its formal adjoint
\begin{align}
    \L^*_\mathrm{tw}w=\rho w'' -c_0w'+f_{\mathrm{cub}}'(\Phi_0,\a)w.
\end{align}
By direct substitution, it can be verified that $\psi_{\mathrm{tw}}(\xi) = \kappa e^{c_0 \xi/\rho } \Phi_0'(\xi)$ is an element of the kernel of $\L^*_\mathrm{tw}$; see also \cite[ex. 2.3.1 and 4.1.4]{Kapitula}
for more intuition. The constant $\kappa$ can be chosen in such a way that
\begin{equation}
\label{eq:int:AdjointEf}
\L_\mathrm{tw} \Phi_0' = \L^*_\mathrm{tw} \psi_{\mathrm{tw}} = 0,
\qquad \langle \psi_{\mathrm{tw}} , \Phi_0' \rangle = 1,
\end{equation}
which allows us to write 
\begin{equation}
\label{eq:int:def:sp:proj}
P v = \langle \psi_{\mathrm{tw}}, v \rangle_{L^2}\Phi_0'
\end{equation}
for the spectral projection onto the simple zero
eigenspace of $\L_{\mathrm{tw}}$.

Since the remainder of the spectrum of $\L_{\mathrm{tw}}$ lies strictly to the left of the imaginary axis,
general considerations \cite{lorenzi2004analytic} can be used to show that the associated semigroup satisfies the bound
\begin{equation}
\label{eq:int:sem:bound:with:m}
    \norm{e^{\mathcal{L}_{\mathrm{tw}} t} (I-P) v}_{L^2(\R)}
      \le M e^{- \beta t} \norm{(I-P) v}_{L^2(\R)}
\end{equation}
for some $\beta > 0$ and $M \ge 1$. The approach in this paper shows how this bound
can be exploited to show that the stochastic waves discussed above are robust
against small perturbations.

With additional ad hoc work  \cite{StannatNag} it is even
possible to show that $M=1$ holds in \sref{eq:int:sem:bound:with:m}.
Based on this latter property we say that
the semigroup is immediately contractive.
Indeed, perturbations are not able to grow even on short timescales, but always
decay exponentially fast back to the wave.
We do not use this
property here, but it has played an essential role in many previous studies on stochastic waves
\cite{Inglis,StannatKruger}.

\paragraph{Example II: The FitzHugh-Nagumo system}
Our second main example is the two-component FitzHugh-Nagumo system
\begin{align}
\label{eq:int:fhn:spde}
\begin{split}
u_t&=\p_{xx} u + f_{\mathrm{cub}}(u; \a)-w+\sigma u
    \xi(x,t),\\
w_t&=\e\p_{xx}w+\varrho(u-\gamma w),\\
\end{split}
\end{align}
in which $\e > 0$ and $\varrho > 0$ are small parameters and $\gamma > 0$ is not too large.
For convenience, we reuse the covariance kernel $q$ 
given in \sref{eq:int:q:gaussian}.
In the deterministic case $\sigma =0$, this system can be used to describe signal propagation through nerve fibres.
It is famous for its  fast and slow travelling pulses that make an excursion from the stable 0 state.
Indeed, the construction of these pulses sparked many developments in the area of singular perturbation theory \cite{hastings1976travelling,carpenter1977geometric,jones1984stability,jones1991construction,jones1995geometric}.
Unlike the previous example, explicit expressions are not available for the profiles and wavespeeds.
Nevertheless, it is known that the fast pulses are spectrally stable \cite{alexander1990topological}.
This allows the framework developed in this paper to be applied to \sref{eq:int:fhn:spde},
leading to the numerical and theoretical discussion   in \S\ref{sec:resFHN}. Let us emphasize
that the specific structure of the noise term in \sref{eq:int:fhn:spde} is just for illustrative purposes.
Indeed, our conditions in \S\ref{sec:res} are rather general and allow cross-talk between the noise
on the $u$ and $w$ components.

For systems such as \sref{eq:int:fhn:spde}, we typically expect $M > 1$ to hold for the general bound \sref{eq:int:sem:bound:with:m}.
This means that larger excursions from the wave are possible before the exponential decay  of the linear semigroup steps in.
In particular, \sref{eq:int:fhn:spde} does not fit into the framework of any previous results in this area. 
In fact, besides our earlier work \cite{Hamster2017,Hamster2018}, there do not seem to be many rigorous studies of travelling waves
for multi-component SPDEs in the literature.

\paragraph{Translational invariance}
Notice that \sref{eq:NagumoPhysics} is translationally invariant,
in the sense that the deterministic terms
are autonomous while the correlation function
depends only on the distance between two points.
This is a natural assumption, as any explicit dependence on $x$ and $y$ individually
would imply some a-priori knowledge about the noise that is often not available.
Indeed, in many applications \cite{Armero1996,Garcia2001,GarciaSpatiallyExtended,Bernitt2017}
translationally invariant noise is considered to be the preferred modelling tool.

However, this choice does present certain mathematical issues that
do not arise when replacing
$q(x-y)$ by $\tilde{q}(x,y)$ in \sref{eq:NagumoPhysics}
and assuming that $\tilde{q}$ is square integrable
with respect to $(x,y)$. This breaks
the translational symmetry, but
does allow the noise-term to be expanded as a countable sum of Brownian motions.
This approach is taken in several other works such as \cite{Cartwright2018}.

In the following paragraph we will explain how to set up a framework to study
translationally invariant noise,
but we emphasize that this is only relatively straightforward in the case of multiplicative noise \cite[\S2.2.1]{LangThesis}
Indeed, 
additive noise of this type cannot be treated
directly, but needs a far more abstract machinery that is still under development \cite{Hairer2015}.

We recall that the goal of our approach is to understand
the long-term behaviour of the travelling waves under consideration, which move
freely throughout the entire spatial domain.
We therefore believe that the elegance
of the translationally invariant point of view
in combination with the direct
relevance for 
applications
outweighs the additional mathematical complications.

\paragraph{Cylindrical Wiener process}
At present,  \sref{eq:NagumoPhysics}
should be interpreted as a pre-equation rather than
an actual SPDE.
Our first task is to give a mathematical interpretation to the stochastic term involving the process $\xi$.
To this end, we
assume that the correlation function $q$
is integrable,
which allows us to
define a bounded\footnote{The boundedness of $Q$ follows from Young's convolution inequality: $\nrm{q*v}_{L^2}\leq \nrm{q}_{L^1}\nrm{v}_{L^2}$.} linear convolution operator
$Q: L^2(\R) \to L^2(\R)$
that acts as
\begin{equation}
[Q v](x) = [q* v](x) = \int_\Real q(x-y) v(y)
\, d y.
\end{equation}
Assuming furthermore that the Fourier transform
$\widehat{q}$ is a non-negative function,
one can show that $Q$
is a non-negative symmetric operator. More concretely,
we have $\ip{Qv,v}_{L^2(\R)}\geq 0$ for all $v\in L^2(\R)$
and it is possible to define a square-root
$Q^{1/2}: L^2(\R) \to L^2(\R)$.

However, we caution the reader
that typically $Q$ has infinite trace and is not even compact.
In particular, it is not generally possible to construct
a countable orthonormal basis of $L^2(\R)$ that consists
of the eigenfunctions of $Q$. This prevents us from using the Brownian-motion expansion discussed above. Stated in technical terms,
we cannot interpret $\xi$ as the derivative of a `regular' $Q$-Wiener process.

These difficulties can be resolved through the use of cylindrical
Wiener processes. Historically, such processes were
developed to handle noise that is white (i.e. completely decorrelated) both in space and time. In our notation, this means that $q$ is replaced by the delta-function to yield $Q = I$, which is clearly not of finite trace.
This approach only requires that $Q$ is bounded, self adjoint and nonnegative and hence applies to the class of operators $Q$ introduced above \cite[\S4.3]{DaPratoZab}.

To set the stage, we 
define the subspace
\begin{equation}
    L_Q^2 = Q^{1/2}\big(L^2(\R)\big) \subset L^2(\R)
\end{equation}
equipped with the inner product
\begin{equation}
    \langle v, w \rangle_{L^2_Q} = \ip{Q^{-\frac{1}{2}}v,Q^{-\frac{1}{2}}w}_{L^2(\R)}.
\end{equation}
In addition, we follow the construction in \cite[{\S}2.5]{Concise}
to define a Hilbert space $L^2_\mathrm{ext}$ that contains $L^2(\R)$
and has the special property that the inclusion
$L^2_Q \subset L^2_\mathrm{ext}$ is Hilbert-Schmidt.

One can now follow the procedure in \cite[{\S}2.5]{Concise} or \cite{Karczewska2005},
to construct the so-called cylindrical $Q$-Wiener process $W^Q_t$.
This process arises as a limit of processes on $L^2_Q$ that converges as a process on $L^2_\mathrm{ext}$, where it can be understood as a `regular' $\bar Q$-Wiener process for some compact $\bar Q: L^2_\mathrm{ext} \to L^2_\mathrm{ext}$. This means that $W^Q_t$ does not necessarily attain values in $L^2(\R)$, but fortunately, the exact choice for $L^2_\mathrm{ext}$
is immaterial\footnote{For translation invariant processes it is possible to explicitly characterize a choice for $L^2_\mathrm{ext}$ in terms of the dual of a Schwartz space, see \cite{Peszat1997}. } for two important reasons.

First,
it turns out \cite[Prop. 2]{Karczewska2005}
that the expression $\langle W^Q_t , v \rangle_{L^2(\R)}$ is well-defined
for
any $v \in L^2(\R)$.
In fact, it can be interpreted as a scaled Brownian motion
that satisfies the correlations
\begin{equation}
    E \Big[ \langle W^Q_t , v \rangle_{L^2(\R)}
       \langle W^Q_s , w \rangle_{L^2(\R)} \Big]
        = (t \wedge s) \langle Q v, w \rangle_{L^2(\R)}.
\end{equation}
In particular, formally replacing $v$ and $w$ by delta functions $\delta_x(\cdot)$ and $\delta_y(\cdot)$ and taking the derivative with respect to $t$ and $s$, we see that
\begin{equation}
    E \Big[ \langle dW^Q_t , \delta_x \rangle_{L^2(\R)}
       \langle dW^Q_s , \delta_y \rangle_{L^2(\R)} \Big]
        =  \delta(t-s)q(x-y).
\end{equation}
Comparing this with \sref{eq:int:exp:xi},
we see that $\frac{d}{dt}{W}^Q_t(x)$ and $\xi(x,t)$
are natural counterparts.

The second reason is that all the essential stochastic estimates
we will need only rely on the space $L^2_Q$. For example,
the full noise term in \sref{eq:NagumoPhysics} is well-defined
if 
the pointwise multiplication
\begin{equation}
v(\xi) \mapsto g\big(U(\xi)\big) v(\xi)
\end{equation}
can be interpreted as a Hilbert-Schmidt operator
from $L^2_Q$ into $L^2(\R)$ for any relevant
function $U$. We will show in Appendix \ref{sec:app:est} that this
can be achieved by imposing simple bounds
on the scalar function $g: \Real \to \Real$
and its derivative.

\paragraph{Interpretation}
In many applications involving external noise, it is natural
to interpret the stochastic terms in the Stratonovich
sense \cite{vanKampen}. Indeed, this interpretation yields
the correct limit when approximating a Wiener process by regularized
versions that can be fitted into the standard deterministic framework (a so-called Wong-Zakai Theorem).
Upon using the process $W^Q_t$ to
recast \sref{eq:NagumoPhysics} as a SPDE
in Stratonovich form,
we  arrive at
\begin{align}
\label{eq:NagumoMathS}
dU&=[\rho\p_{xx} U +f(U)]dt+\sigma g(U)\circ dW^Q_t.
\end{align}

The equivalent It\^o
formulation is given by
\begin{align}
\label{eq:NagumoMathI}
\begin{split}
dU&=\Big[\rho\p_{xx} U +f(U)+\mu\frac{\s^2}{2}q(0)g'(U)g(U)\Big]dt+\sigma g(U)dW^Q_t\\
\end{split}
\end{align}
with $\mu = 1$. From a mathematical point of view it is
more convenient to work in this formulation, since most
of the technical machinery for SPDEs is based on It\^o calculus.
The choice $\mu =0$ allows us to interpret
the noise in \sref{eq:NagumoPhysics} in the It\^o sense directly.
Our results in this paper cover both cases,
in order to ease the comparison with previous work and to
illustrate how the two types of noise
impact the deterministic waves in different ways.

\paragraph{Previous results}
Rigorous results concerning the well-posedness
of SPDEs of type \sref{eq:NagumoMathI} are widely available by now, see e.g. \cite{Concise}. However, the dynamics of this type of equation is less well studied in the math community. Several authors have considered
the dynamics of stochastic waves driven by
$Q$-Wiener processes, 
which means that the noise is necessarily localized in space.
For example, the shape and speed of stochastic waves
for Nagumo-type SPDEs 
were computed numerically in \cite{Lord2012}
and derived formally in \cite{Cartwright2018}
using a collective coordinate approach.
In addition, short-time stability results
for immediately contractive systems can be found in
\cite{StannatKruger,Inglis}.
The results in  \cite{Mueller1995}
do use cylindrical $Q$-Wiener processes for waves in the Fisher-KPP equation, but there the  smooth covariance function $q$ is replaced by a delta-function in order
to model noise that is white in space and time. A more detailed overview of results on stochastic travelling waves can be found in the review by Kuehn \cite{KuehnReview}.

Turning to non-rigorous results for
\sref{eq:NagumoPhysics} from other fields, we refer
to \cite{GarciaSpatiallyExtended} for an interesting
overview of studies that have appeared in the physics
and chemistry literature. For the Nagumo SPDE
\sref{eq:ExNag}, the dynamics up to first order in $\s$ of
have been formally computed \cite[eq. (6.11)]{GarciaSpatiallyExtended}.
At this order, the shape of the wave is equal to the deterministic shape and the phase of the wave follows a Brownian motion with a variance
that can be expressed in closed form. We will see
in {\S}\ref{sec:resN} how these conclusions can be recovered
as special cases from our expansions.

\paragraph{Phase tracking}
Our work here builds on the framework
developed in \cite{Hamster2017,Hamster2018}
to study travelling waves in stochastic reaction-diffusion equations  forced by a single Brownian motion.
The main idea is to use a phase-tracking
approach that is based purely on technical considerations rather than
ad hoc geometric intuition. Inspired by the overview in \cite{Zumbrun2009},
this allows us to adapt modern tools developed for deterministic
stability results for use in a stochastic setting.

In order to explain the key concepts, we turn to
the deterministic Nagumo PDE that arises by taking
$\sigma = 0$ in \sref{eq:ExNag}. The translational invariance
of the travelling wave $u(x,t) = \Phi_0(x - c_0 t)$ can be captured
by introducing an Ansatz of the form
\begin{align*}
    u\big(\cdot+\gamma(t),t\big)=\Phi_0(\cdot )+v(\cdot,t),
\end{align*}
in which $\gamma(t)$ can be interpreted as the phase of $u$.
We now demand that the evolution of the phase is governed by
\begin{align}\label{int:defdetphase}
    \dot{\g}(t)=c_0 +a\big(v(\cdot,t)\big),
\end{align}
for some (nonlinear) functional $a:L^2(\R)\to\R$ that we are still free to choose.
The resulting equation for $v$ is then given by
\begin{align}\label{eq:res:va}
    \p_t v(t)=\L_\mathrm{tw} v(t)
       +N\big(v(t) \big)+a\big(v(t)\big)\p_\xi \big[\Phi_0+v(t) \big],
\end{align}
in which
 $N(v)=f_{\mathrm{cub}}(\Phi_0+v)-f_{\mathrm{cub}}(\Phi_0)
-f_{\mathrm{cub}}'(\Phi_0)v$.
This can be recast into the mild form
\begin{align}
    v(t)=e^{\L_{\mathrm{tw}} t} v_0
      +\int_0^te^{\L_{\mathrm{tw}} (t-s)}\Big[N\big(v(s)\big)+a\big(v(s)\big)\p_\xi\big[\Phi_0+v(s)\big]\Big]ds,
\end{align}
inviting us to apply the bound \sref{eq:int:sem:bound:with:m}.

In order to apply exponential bounds such as \sref{eq:int:sem:bound:with:m} to the semigroup $e^{\L_\mathrm{tw}}$, we must avoid the neutral non-decaying part of the
semigroup. In order to force the integrand 
to be orthogonal to the zero eigenspace, 
we recall the spectral projection \sref{eq:int:def:sp:proj} and choose
\begin{align}\label{eq:res:defa}
    a(v)=-\frac{\ip{N\big(v\big),\psi_{\mathrm{tw}}}_{L^2(\R)}}
      {\ip{\p_\xi(\Phi_0+v), \psi_{\mathrm{tw}}}_{L^2(\R)}}.
\end{align}
In fact, one arrives at the same choice
if one directly imposes the orthogonality condition\\
$\langle v(t), \psi_{\mathrm{tw}} \rangle_{L^2(\R)} = 0$.
By a standard bootstrapping procedure one can now establish the limits
$\nrm{v(t)}_{L^2}\to 0$ and $t^{-1}\g(t)\to c_0$
for $t \to \infty$, provided that $v(0)$ is sufficiently small.  This allows us to conclude
that the travelling wave is orbitally stable.

In our stochastic setting, the pair $(v,\gamma)$
is replaced by its stochastic counterpart $(V,\Gamma)$, which we always write in capitals. The resulting equations for 
this pair are naturally much more complicated. They both consist
of a deterministic and a stochastic part, resulting in \textit{two}
free functionals that can be tuned to ensure $\langle V, \psit \rangle_{L^2(\R)} = 0$, see \S\ref{sec:derphase}.

Besides the fact that we use the adjoint eigenfunction $\psit$
instead of $\Phi_0'$, the main difference with the phase tracking
approaches developed in \cite{Inglis,StannatKruger}
is that our perturbation $V$ is taken relative
to a novel pair $(\Phi_{\s}, c_{\s})$ that we refer
to as the instantaneous\footnote{See {\S}\ref{sec:mr:pred}
for a justification for this terminology.} stochastic wave. This pair
is chosen in such a way that the deterministic part
of the equation for $V$ vanishes at $V = 0$.
However, this does not hold for the
stochastic part, leading to persistent fluctuations
that must be controlled.

\paragraph{Stability}
Our first contribution is that we establish that the wave $(\Phi_{\sigma},c_\s)$ is stable, in the sense that
the perturbation $V(t)$
remains small over time scales of $\O(\s^{-2})$.
In particular, we show that
the semigroup techniques developed in our earlier work \cite{Hamster2017,Hamster2018} are general enough
to remain applicable in the present more convoluted setting.
The main effort is to verify that certain technical
estimates remain valid, which is possible by the powerful
theory that has been developed for cylindrical $Q$-Wiener processes.

The procedure in \cite{Hamster2017,Hamster2018} is rather delicate
in order to compensate for the lack of immediate contractivity.
Indeed, the $H^1$-norm of $V(t)$ must be kept under control,
resulting in apparent singularities in the stochastic
integrals that must be handled with care. The time scale mentioned
above arises as a consequence of the mild Burkholder-Davis-Gundy
inequality that we used to obtain supremum bounds
on stochastic integrals. However, these bounds are known to be suboptimal.
Indeed, we believe that our phase tracking approach can be maintained
for time scales that are exponential in $\sigma$. This
is confirmed by the numerical results at the end of \S\ref{sec:resN}.

\paragraph{Expansions in $\sigma$}
The second -- and main -- contribution in this paper is that we explicitly show how to expand the fluctuations around the stochastic
wave $(\Phi_\sigma, c_\sigma)$ in powers of the noise strength $\sigma$.
In particular, we show that our framework yields
a natural procedure to compute Taylor expansions for the pair
$\big(V(t),\Gamma(t) \big)$. These results extend
the pioneering work in 
\cite{Lang,StannatKruger}, where related multi-scale expansions
were achieved in a variety of settings on short time-scales.

An important advantage
of our semigroup approach is that the resulting terms
have expectations that are well-defined in the limit $t \to \infty$.
In particular, we are able to uncover the long-term stochastic
corrections to the speed and shape of the travelling waves.

We provide explicit formula's for the first and second
order corrections, which all crucially involve
the semigroup $e^{\L_{\mathrm{tw}}t}$.
In addition, we show
how to compute the third order corrections
in the phase $\Gamma$
from these second order
corrections. In principle, the expansions can be computed to any desired
order in $\sigma$, but the process quickly becomes unwieldy.

At first order in $\sigma$,
our predictions 
concur with many earlier results
\cite{GarciaSpatiallyExtended,StannatKruger,Cartwright2018,Bressloff},
which show
that the phase $\Gamma(t)$ of the wave
behaves as a Brownian motion centered around the deterministic position $c_0t$.
In addition, the shape of the wave fluctuates at first order like an infinite dimensional Ornstein-Uhlenbeck process around the deterministic wave $\Phi_0$.

At second order in $\sigma$, two distinct effects start to play a role. The first is that differences start to appear between
the instantaneous stochastic wave $(\Phi_{\sigma},c_{\sigma})$
and the deterministic wave $(\Phi_0,c_0)$. 
This generalizes
the wave steepening effect \sref{eq:int:wave:steepening} discussed in 
\cite{Hamster2017,Cartwright2018,Lord2012}.
On top of this,
there is an additional contribution to the average speed and shape
that is caused by the feedback of the first order
fluctuations
of $V(t)$. These effects -- which we refer to as orbital drift --
become visible after an initial transient
period. Besides a brief discussion in our earlier
work \cite{Hamster2017,Hamster2018}, we do not believe
that this  long-term behaviour
has been systematically explored before.

Taken together, we now have
a quantitative procedure
that is able to accurately describe the
numerical results  in \cite{Lord2012} and the formal computations in \cite{Cartwright2018} for the Nagumo SPDE \sref{eq:ExNag}, both for the Stratonovich and It\^o formulation. This allows us to understand the differences in speed and shape between both interpretations analytically.
These predictions are confirmed by our numerical results,
which compare the solutions of the full SPDE with our
explicit formula's and exhibit the rate of convergence
with respect to $\sigma$.

\paragraph{Outlook}
In this paper we will not treat space-time white noise, i.e. $q(x-y)=\delta(x-y)$, as our mathematical framework does not yet allow distributions to be used as kernels. This can already
be seen from the fact that \sref{eq:NagumoMathS} depends
explicitly on $q(0)$, which of course is not well-defined for distributions.
In fact, it is still a subject of active research \cite{Hairer2015}
to give a clear interpretation of \sref{eq:NagumoPhysics}
in this case.
In the It\^o interpretation however, many of our computations concerning the shape and speed of the stochastic wave have a well-defined limit
if we let $q$ converge to $\delta$. In addition, many of our
expressions still make sense for $g=1$, which suggests
that our expansions could also be used to make predictions
for additive noise.

In addition, we  expect that our methods can be extended to other types of equations. For example, stochastic neural field models have attracted a lot of attention in recent years, but they lack the smoothening effect of the
diffusion operator. Finally, we are exploring
techniques that would allow us to extend the $\O(\s^{-2})$
time scales in our results to the exponentially long  scales
observed in the numerical computations.

\paragraph{Organization}
In \S\ref{sec:res} we state our assumptions and
give a step-by-step overview of the steps that lead to our
expansions. In addition, we provide the explicit
formula's that describe the first and second order terms
in the expansions for $(V,\Gamma)$.
In  \S\ref{sec:resN} and \S\ref{sec:resFHN} we illustrate
how our results can be applied to the Nagumo and FitzHugh-Nagumo
SPDEs and verify the predictions with numerical computations.

The remaining sections contain the technical heart of this paper
and provide the link between our setting here
and the bootstrapping procedure developed
in \cite{Hamster2017,Hamster2018}.
In particular, 
we show in \S\ref{sec:sps} how 
the stochastic evolution equation for $V$ 
can be computed using It\^o calculus,
which represents the core computation of this work. In {\S}\ref{sec:stb}
we explain how the technical machinery available
for cylindrical $Q$-Wiener processes 
can be used to follow
the steps of \cite{Hamster2017,Hamster2018}.
Appendix \ref{sec:app:est} provides
the main link between these papers,
showing how
the estimates in \cite{Hamster2017} can be generalized to the nonlinearities appearing here.

\paragraph*{Acknowledgements}
HJH acknowledges support from the Netherlands Organization for Scientific Research (NWO) (grant 639.032.612).

\newpage
\section{Main results}
\label{sec:res}
In this paper we study the properties of travelling
wave solutions to stochastic reaction-diffusion systems
with translationally invariant noise. In {\S}\ref{sec:mr:st}
we introduce the class of systems that we are interested in.
The main steps of our approach are outlined
in {\S}\ref{sec:derphase}, which allows us to expand the stochastic
corrections to the shape and speed of the waves in powers of the noise
strength. The precise form of these expansions
is described in {\S}\ref{sec:explexp}. Finally,
we discuss several consequences of our results
in {\S}\ref{sec:mr:pred} and compare them to earlier work
in this area.

\subsection{Setup}
\label{sec:mr:st}
In this section we formulate the conditions that we need to impose
on our stochastic reaction-diffusion system. Taken together,
these conditions ensure that the noise-term is well-defined,
that the SPDE is well-posed and that the deterministic part
admits a spectrally stable travelling wave.

\paragraph{Noise Process}
We start by
discussing the  covariance function $q$ that underpins the noise process,
which we assume to have $m$ components. In particular, we impose
the following condition on the $m \times m$ components of the function $q$
and its Fourier transform $\widehat{q}$.

\begin{itemize}
\item[(Hq)] {
We have
$q \in H^1(\R,\R^{m \times m})\cap L^1(\R,\R^{m \times m})$,
with $q(-\xi) = q(\xi)$ and $q^T(\xi) = q(\xi)$ for all $\xi \in \Real$.
In addition, for each $k \in \Real$
the $m\times m$ matrix $\widehat{q}(k)$
is non-negative definite.
}
\end{itemize}

Since the Fourier transform maps Gaussians onto Gaussians, this condition
can easily be verified for $q(x) = \mathrm{exp}(-x^2)$ in the scalar case
$m=1$. Other examples include the exponential $q(x) = \mathrm{exp}\big(-\abs{x}\big)$
and the tent function $q(x) = 1 - \abs{x}$ supported on $[-1,1]$.

The integrability of $q$ allows us to introduce a bounded linear operator $Q: L^2(\R,\R^m) \to L^2(\R,\R^m)$
that acts as
\begin{equation}
    [Q v](x) = [q* v](x) = \int_{\Real} q( x - y) v(y) \, dy .
\end{equation}
The remaining conditions in (Hq) show that $Q$ is
symmetric and
that
$\langle Q v , v \rangle_{L^2(\R,\R^m)} \ge 0$
holds for all $v \in L^2(\R,\R^m)$.
As  explained in {\S}\ref{sec:int}
and {\S}\ref{sec:sps:bck},
this allows us to follow \cite[{\S}2.5]{Concise}
and \cite{Karczewska2005} to define a cylindrical $Q$-Wiener process $W^Q_t$
over $L^2(\R,\R^m)$.
In particular, for any $v,w \in L^2(\R,\R^m)$ we have
\begin{align}
    E[\ip{W^Q_t,v}_{L^2(\R,\R^m)}\ip{W^Q_s,w}_{L^2(\R,\R^m)}]=s\wedge t\, \ip{q*v,w}_{L^2(\R,\R^m)}.
\end{align}
Upon writing $\{ e_i \}$ for the standard unit vectors
together with $v(x) = \delta(x - x_0)e_i$ and
$w(x) = \delta(x - x_1) e_j$,
this reduces formally to the familiar expression
\begin{align}
    E[dW^Q_t(x_0)dW^Q_s(x_1)]=\delta(t-s) q_{ij}(x_0-x_1),
\end{align}
after taking the time derivative with respect to $t$ and $s$.
This highlights the role that the correlation function $q$
plays in our setup.

\paragraph{Stochastic reaction-diffusion equation}
The main SPDE that we will study
can now be formulated as
\begin{align}\begin{split}
\label{eq:mr:main:spde}
dU &=
  \big[ \rho \partial_{xx} U + f(U)+\s^2 h(U) \big] dt
  + \sigma g(U) d W^Q_t.\\
\end{split}
\end{align}
Here we take $U = U(x,t) \in \Real^n$ with $x \in \Real$ and $t \ge 0$.
The nonlinearities $f: \Real^n \to \Real^n$,
$h: \Real^n \to \Real^n$ and $g: \Real^n \to \Real^{n \times m}$
are considered to act in a pointwise fashion.
In order to proceed, we need to assume that these nonlinearities
have a common pair of equilibria.
\begin{itemize}
    \item[(HEq)]{
      There exist $u_\pm \in \Real^n$ so that
      \begin{equation}
        f(u_\pm) = g(u_\pm) = h(u_\pm) = 0.
      \end{equation}
    }
\end{itemize}
If $u_- \neq u_+$, then the relevant solutions $U$ to \sref{eq:mr:main:spde}
cannot be captured in the Hilbert space $L^2(\R,\R^n)$. In order
to remedy this, we pick a smooth reference
function $\Phi_\mathrm{ref}$ that has the limits
$\Phi_\mathrm{ref}(\pm \infty) = u_\pm$
and introduce the affine spaces
\begin{equation}
    \mathcal{U}_{H^1} = \Phi_\mathrm{ref} + H^1(\R,\R^n),
    \hspace{1cm}
    \mathcal{U}_{H^2} = \Phi_\mathrm{ref} + H^2(\R,\R^n).
\end{equation}
Naturally, we can simply take $\Phi_{\mathrm{ref}} = 0$
if $u_- = u_+$.
We will see that \sref{eq:mr:main:spde} is well-posed
as a stochastic evolution equation on $\mathcal{U}_{H^1}$. In fact,
it is advantageous to study $X(t)=U(t)-\Phi_\mathrm{ref}$, which solves the SPDE\footnote{
We emphasize that the reference function $\Phi_{\mathrm{ref}}$
does not depend on time,
which is why \sref{eq:mr:main:spdeX}
contains no additional time derivatives.
}
\begin{align}\begin{split}
\label{eq:mr:main:spdeX}
dX &=
  \big[ \rho \partial_{xx}(X+\Phi_\mathrm{ref}) + f(X+\Phi_\mathrm{ref})+\s^2 h(X+\Phi_\mathrm{ref}) \big] dt
  + \sigma g(X+\Phi_\mathrm{ref}) d W^Q_t\\
\end{split}
\end{align}
and hence attains values in $H^1(\R,\R^n)$. 
This decomposition will be used in
\S\ref{sec:sps}-\S\ref{sec:stb}.

\paragraph{Stochastic terms}
In order to ensure that the stochastic term
in \sref{eq:mr:main:spde} is well-defined,
we impose the following growth bound on $g$.
\begin{itemize}
\item[(HSt)]{
  We have $g\in C^2(\Real^n, \Real^{n \times m})$.
  In addition, the derivative $Dg$ is bounded and globally Lipschitz continuous.
}
\end{itemize}
Indeed, in Appendix \ref{sec:app:est}  this assumption
is used
to establish that the pointwise map
\begin{equation}
    [g(U) v](x) = g\big(U(x)\big) v(x)
\end{equation}
is a Hilbert-Schmidt operator
from
\begin{equation}
    L^2_Q = Q^{1/2}\big(L^2(\R,\R^m) \big)
\end{equation}
into $L^2(\R,\R^n)$ for any $U \in \mathcal{U}_{H^1}$.
We note that the existence of the square-root $Q^{1/2}$
follows from the fact that $Q$ is a nonnegative operator.
This square-root has a convolution kernel $p$
that is also translationally invariant;
see Appendix \ref{sec:est:prlm} for the details.

For clarity, we take the noise intensity $\s$ to be a scalar
factor in front of $g$. In principle however,
each of the $n\times m$ components of $g$ could have its own
scaling.
This can also be fitted into our framework, but it would unnecessarily complicate the expansions we are after.

\paragraph{Deterministic terms}
Turning to the deterministic part of \sref{eq:mr:main:spde},
we first note that
$h$ is a function that can be used to represent the
appropriate It\^o-Stratonovich correction terms.
For example, in the scalar case $n=m=1$ we saw in {\S}\ref{sec:int}
that the choice $h(U)=\frac{1}{2}q(0) g'(U)g(U)$
allows us to interpret \sref{eq:mr:main:spde}
as the Stratonovich SPDE
\begin{align}\begin{split}
\label{eq:mr:main:spde:strat}
dU &=
  \big[ \rho \partial_{xx} U + f(U) \big] dt
  + \sigma g(U) \circ d W^Q_t.\\
\end{split}
\end{align}
We refer to \cite{Twardowska1996,evans2012introduction}
for further information concerning the construction of similar
correction terms for multi-component systems.

We now impose the following conditions
on the nonlinearities  $f$ and $h$.
\begin{itemize}
\item[(HDt)]{
  The matrix $\rho \in \Real^{n\times n}$ is a diagonal matrix with
  strictly positive diagonal elements $\{\rho_i\}_{i=1}^n$.
  In addition, we have $f,h \in C^3(\Real^n, \Real^n)$.
 Finally,  $D^3 f$ and $D^3 h$ are bounded
 and  there exists
 a constant $K_{\mathrm{var}} > 0$ so that the one-sided inequality
 \begin{align}
\begin{split}
    \langle f( u_A ) - f(u_B) , u_A - u_B \rangle_{\Real^n}
    + \sigma^2 \langle h( u_A ) - h(u_B) , u_A - u_B \rangle_{\Real^n}
    &\le K_{\mathrm{var}} \abs{ u_A - u_B}^2
\end{split}
 \end{align}
 holds for all pairs $(u_A,u_B) \in \Real^n \times \Real^n$
 and all $0\le \sigma \le 1$. }
\end{itemize}
The precise form of these assumptions is strongly motivated
by the setup in \cite{LiuRockner}. Indeed,
the four conditions (Hq), (HEq), (HSt) and (HDt)
together allow us to apply \cite[Thm 1.1]{LiuRockner}.
This implies  that our system
\sref{eq:mr:main:spde}
has a unique solution in $\mathcal{U}_{H^1}$
that is defined
for all $t \ge 0$. A precise statement on the properties of
these solutions can be found in Proposition \ref{prp:phs:main:ex}.

\paragraph{Travelling wave}

The following assumption states that the deterministic part
of \sref{eq:mr:main:spde}
has a spectrally stable travelling wave solution
that  connects the two equilibria $u_\pm$. We remark
again that these two limiting values are allowed to be equal.

\begin{itemize}
\item[(HTw)]{
  There exists a wavespeed $c_0 \in \Real$ and a waveprofile
  $\Phi_0 \in C^2(\Real, \Real^n)$ that
  satisfies the travelling wave ODE
  \begin{equation}\label{eq:MR:TWODE}
    \rho \Phi_0''+c_0 \Phi_0' + f(\Phi_0)=0
  \end{equation}
  and approaches its limiting values
  $\Phi_0(\pm \infty) = u_\pm$ at an exponential rate.
  In addition, the associated linear operator
  $\mathcal{L}_{\mathrm{tw}}: H^2(\R,\R^n) \to L^2(\R,\R^n)$ that acts as
  \begin{equation}
   \label{eq:mr:def:l:tw}
   [\mathcal{L}_{\mathrm{tw}} v](\xi) =
      \rho v''(\xi)+c_0 v'(\xi)  + Df\big(\Phi_0(\xi) \big) v(\xi)
  \end{equation}
  has a simple eigenvalue at $\lambda = 0$
  and has no other spectrum
  in the half-plane $\{\Re \lambda \ge -2\beta\} \subset \mathbb{C}$
  for some $ \beta > 0$.
}
\end{itemize}
The formal adjoint
\begin{equation}
\mathcal{L}_{\mathrm{tw}}^*: H^2(\R,\R^n)
  \to L^2(\R,\R^n)
\end{equation}
of the operator \sref{eq:mr:def:l:tw}  acts as
\begin{equation}
[\mathcal{L}_{\mathrm{tw}}^* w](\xi) =
\rho w''(\xi) -c_0 w'(\xi) +  Df\big(\Phi_0(\xi) \big)^T w(\xi) .
\end{equation}
Indeed, one easily verifies that
\begin{equation}
\langle \mathcal{L}_{\mathrm{tw}} v , w \rangle_{L^2(\R,\R^n)}
= \langle v, \mathcal{L}_{\mathrm{tw}}^* w \rangle_{L^2(\R,\R^n)}
\end{equation}
holds whenever $(v,w) \in H^2(\R,\R^n) \times H^2(\R,\R^n)$.
The assumption that zero is a simple eigenvalue for $\mathcal{L}_{\mathrm{tw}}$
implies that $\mathcal{L}_{\mathrm{tw}}^* \psi_{\mathrm{tw}} = 0$
for some $\psi_{\mathrm{tw}} \in H^2(\R,\R^n)$ that can be normalized to have
  \begin{equation}
    \label{eq:mr:hs:norm:cnd:psitw}
     \langle \Phi_0' , \psi_{\mathrm{tw}} \rangle_{L^2(\R,\R^n)} = 1.
  \end{equation}

These assumptions imply \cite[{\S}4]{Kapitula}
that the family of travelling wave solutions
\begin{equation}
U(x,t) = \Phi_0(x + c_0 t + \vartheta),
\qquad \qquad \vartheta \in \Real,
\end{equation}
is nonlinearly stable
under the dynamics of $\sref{eq:mr:main:spde}$ at $\s=0$. In particular,
any small perturbation from such a wave converges exponentially
fast to a nearby translate.

\subsection{Overview}
\label{sec:derphase}

Guided by the short sketch in {\S}\ref{sec:int} of the ideas
behind the deterministic stability proof,
we now give a step-by-step description of
the stochastic framework that we use to
generalize this result and compute our expansions.
At this stage we only give an overview of the key concepts,
leaving the details to {\S}\ref{sec:explexp} and later sections.

\paragraph{Step 1: Stochastic phase.}
We introduce a  wavespeed $c_{\sigma} \in \Real$,
together with a nonlinear functional
$\overline{a}_{\sigma}: \mathcal{U}_{H^1} \times \Real \to \Real$.
In addition, for any $U \in \mathcal{U}_{H^1}$
and $\Gamma \in \Real$
we define a Hilbert-Schmidt operator $\overline{b}(U,\Gamma)$
that maps $L^2_Q$ into $\Real$. We emphasize
that all three objects are unknown at present; see \S\ref{sec:sps:prlm}
for their precise definitions.
However, they do allow us to
define a stochastic phase $\Gamma(t)$
by coupling the SDE
\begin{align}\label{eq:res:Gamma}
    d \Gamma=  \big[ c_{\sigma} + \overline{a}_{\sigma}(U,\G) \big] \, dt
     + \sigma \overline{b}(U,\G) dW^Q_t
\end{align}
to the SPDE \sref{eq:mr:main:spde} that governs $U$.
This generalizes the deterministic phase
that was introduced in \sref{int:defdetphase}. For convenience, we also write the phase in the integrated form 
\begin{align}\label{eq:res:Gamma:int}
    \Gamma(t)=\G_0 +c_{\sigma}t + \int_0^t\overline{a}_{\sigma}
      \big(U(s),\Gamma(s) \big)\, ds
     + \sigma\int_0^t\overline{b}
       \big(U(s), \Gamma(s) \big) dW^Q_s.
\end{align}

\paragraph{Step 2: Decomposition of $U$.}
We now introduce a, yet unknown, waveprofile
$\Phi_{\sigma} \in \mathcal{U}_{H^2}$. We can use this
together with the phase $\Gamma$ defined in Step 1
to define the perturbation
\begin{align}
\label{eq:res:defV}
    V(t)=U(\cdot +\Gamma(t),t)-\Phi_\s,
\end{align}
which measures the deviation from $\Phi_\s$
of $U$ after shifting it to the left by $\Gamma(t)$.
We note that $V$ takes values in $H^1(\R,\R^n)$
in a sense that is made precise
in Proposition \ref{prp:sps:props:v}. In addition, we will from now on write
\begin{equation}
    a_{\sigma}(V) = \overline{a}_\sigma(\Phi_\sigma + V,0),
    \qquad
    \qquad
    b_\s(V) = \overline{b}(\Phi_\sigma + V,0).
\end{equation}

Using It\^o calculus, we will show in \S\ref{sec:sps}
that
$V(t)$ solves an equation of the form
\begin{align}\label{eq:res:V-first}
dV= \big[F_\s(\Phi_\s,c_\s;b_\s)+\L_{\mathrm{tw}} V+N_\s(V;b_\s)
+a_{\sigma} (V)\p_\xi(\Phi_\s+V)\big]dt
  +\sigma \mathcal{S}_\s(V;b_\s)dW^Q_t.
\end{align}
Due to the second order terms in the It\^o formula,
the specific shapes of $F_\s$, $N_\s$ and
$\mathcal{S}_\s$ all depend on the functional $b_\s$.
It is therefore helpful to make a choice for $b_\s$, which we will
do in the next step.

However, at this point an important warning
is in order. We are using
the same symbol for the noise
processes driving
\sref{eq:mr:main:spde} and
\sref{eq:res:V-first}, 
because
- by translation invariance -
they are are indistinguishable from one another.

On the other hand, when one wants to compare $U(t)$ and $V(t)$ numerically for a specific 
\textit{realization} of $W_t^Q$,
then one must take care to spatially translate
this realization by $\Gamma(t)$ 
when passing between \sref{eq:mr:main:spde}
and \sref{eq:res:V-first}. 
Indeed, these two equations
are defined in separate coordinate systems.
This distinction will be explained in detail in the proof
of Proposition \ref{prp:sps:props:v}.
For now, we remark that all the averages that we 
compute in this section are invariant
under translations in the noise.

\paragraph{Step 3: Choice of $b$. }
For any $V \in H^1(\Real,\Real^n)$, the computations in {\S}\ref{sec:sps}
show that
\begin{align}
    \mathcal{S}_\s(V; b_\s)[v]=
     g(\Phi_\s+V)v+\p_\xi(\Phi_\s+V)b_\s(V)[v]
\end{align}
for all $v \in U_0$.
As in the deterministic case, the goal is
to achieve the identity
\begin{equation}
    \big\langle \psi_{\mathrm{tw}} , \mathcal{S}_{\s}(V;b_\s)\big\rangle_{L^2(\R,\R^n)} = 0
\end{equation}
in order to circumvent the neutral mode of the semigroup.
Whenever $\norm{V}_{L^2(\R,\R^n)}$ is sufficiently small,
this can be achieved by writing
\begin{align}\label{eq:derphase:b}
    b_\s(V)[v]=-\frac{\big\langle g(\Phi_\s+V)v,\psit\big\rangle_{L^2(\R,\R^n)}}
      {\big\langle \p_\xi(\Phi_\s+V),\psit \big\rangle_{L^2(\R,\R^n)}}.
\end{align}
Having made this choice,
we now drop the
dependence on $b_\s$ in $F_{\sigma}(V)$, $N_{\sigma}(V)$
and $\mathcal{S}_{\sigma}(V)$.

\paragraph{Step 4: Construction of $(\Phi_\s,c_\s)$.}
Ideally, we would like $V(t) = 0$ to be a solution
to \sref{eq:res:V-first}, 
since then
$U(x,t) = \Phi_\s (x + \Gamma(t), t)$
would be an exact solution to 
\sref{eq:mr:main:spde}.
However, since the deterministic and stochastic
terms both need to vanish simultaneously, this can only be achieved
in very special situations\footnote{In the case of 1d Brownian motion,
we explain in \cite{Hamster2017} how $g$ can be chosen
to make this possible.}.
In Prop. \ref{prp:var:swv:ex} we show that for small $\sigma$,
it is possible to construct a pair $(\Phi_\s, c_\s)$
for which $F_\s(\Phi_\s,c_\s)=0$. Since we will see
below that $N_\s(0) = 0$ and $a_\s(0) =0 $,
this ensures that
the state $V =0$ only experiences (instantaneous) stochastic forcing.

For any pair $(\Phi, c) \in \mathcal{U}_{H^2} \times \R$, the nonlinearity $F_\s(\Phi,c)$ can be decomposed as
\begin{equation}
\label{eq:mr:def:F:i}
    F_{\sigma}(\Phi, c) = F_0(\Phi, c)
    + \sigma^2 F_{0;2}(\Phi).
\end{equation}
The leading order term $F_0(\Phi,c)$ is related to the deterministic wave
in the sense that
\begin{align}
\label{eq:mr:def:F:ii}
    F_0(\Phi,c)=\rho\Phi''+c\Phi'+f(\Phi),
\end{align}
while the correction term $F_{0;2}(\Phi)$ is found to be
\begin{align}
\label{eq:mr:def:F:iii}
 F_{0;2}(\Phi)=\frac{1}{2}\frac{\ip{g(\Phi)
 Q g^T(\Phi) \psi_{\mathrm{tw}},\psit}_{L^2(\R,\R^n)}}{\ip{\Phi',\psi_\mathrm{tw}}_{L^2(\R,\R^n)}^2}\Phi''
 -\frac{(g(\Phi)Q g^T(\Phi)\psi_\mathrm{tw})'}{\ip{\Phi',\psit}_{L^2(\R,\R^n)}}+h(\Phi)
\end{align}
whenever $\nrm{\Phi-\Phi_0}_{L^2}$ is sufficiently small. We emphasize here that the transpose is taken in a pointwise fashion.
Note here that the correction term $F_{0;2}(\Phi)$ depends on the operator $g(\Phi)Qg^T(\Phi)$, which is the covariance operator of the stochastic process
\begin{align}
    \int_0^t g(\Phi)dW^Q_s
\end{align}
for $t\to\infty$ \cite{Hairernotes}.
Therefore, as we will see, the lowest order corrections of $\Phi_\s$ to $\Phi_0$ can be understood in terms of the covariance of the stochastic process $\int_0^tg(\Phi_0)dW^Q_s$.

\paragraph{Step 5: Choice of $a_\s$.}
As in the deterministic case,
we now define $a_\s$ in 
such a way that the deterministic part 
of 
\sref{eq:res:V-first} (i.e., the
terms multiplying $dt$)
becomes
orthogonal to $\psit$. In particular,
for $V$ small we write
\begin{align}\label{eq:derphase:a}
    a_{\sigma}(V)=-\frac{\bip{N_\s(V),\psit}_{{L^2(\R,\R^n)}}}{\bip{\p_\xi(\Phi_\s+V),\psit}_{{L^2(\R,\R^n)}}}.
\end{align}
For reference, we note that the non-linearity $N_\s$
introduced in \sref{eq:res:V-first} is given by
\begin{align}
\begin{split}
    N_{\sigma}(V)  = & F_\s(\Phi_\s+V,c_\s)-F_\s(\Phi_\s,c_\s)-\L_\mathrm{tw}V.
	\end{split}
\end{align}
As we claimed in Step 4,
we indeed see that $a_{\s}(0) = 0$ and $N_\s(0) = 0$.
Upon introducing a nonlinearity $\mathcal{R}_\s$
that acts as
\begin{align}
\label{eq:mr:defn:R}
    \mathcal{R}_\s\big(V\big)=
      F_\s(\Phi_\s +V, c_\s)+a_{\s}(V)\p_\xi(\Phi_\s+V)
\end{align}
for small $V$,
we conclude that $V$ solves the equation
\begin{align}\label{eq:res:V-final}
    dV=\mathcal{R}_\s(V) dt+\s\mathcal{S}_\s(V)dW^Q_s.
\end{align}
The nonlinearities on the right hand side of this equation are now both 
orthogonal to $\psit$
for small $V$.

\paragraph{Step 6: Stability.}
We write $S(t)$ for the semigroup generated
by the linear operator $\mathcal{L}_{\mathrm{tw}}$ and consider
the mild formulation of \sref{eq:res:V-final},
which is given by the integral equation
\begin{align}\label{eq:res:mildV}
    V(t)=S(t)V_0+\int_0^t S(t-s)\widetilde{\mathcal{R}}_\s\big(V(s)\big)ds
    + \sigma\int_0^tS(t-s)\mathcal{S}_\s\big(V(s)\big)dW^Q_s ,
\end{align}
in which we have
\begin{equation}
    \widetilde{\mathcal{R}}_\s(V) = \mathcal{R}_\s(V) - \L_{\mathrm{tw}} V
     =  N_\s(V)+a_{\s}(V)\p_\xi(\Phi_\s+V) .
\end{equation}
By construction, we have achieved
$\ip{\psit,\widetilde{\mathcal{R}}_\s(V)}_{L^2(\R,\R^n)} = \ip{\psit,\mathcal{S}_{\s}(V)}_{L^2(\R,\R^n)} =0$
for small $V$. Whenever $U(0)$ is sufficiently close to $\Phi_0$,
we can also ensure $\ip{\psit,V_0}_{L^2(\R,\R^n)} = 0$
by picking the initial phase $\Gamma(0)$ appropriately.

We caution the reader that it is hard
to obtain estimates on $V$
directly from \sref{eq:res:mildV},
because the term $\widetilde{\mathcal{R}}_\s(V)$
still contains second order derivatives.
Tackling this problem is the key part of \cite{Hamster2018}, as we discuss
in {\S}\ref{sec:stb}.
Nevertheless, it is possible
to show that the instantaneous wave
$(\Phi_\s, c_\s)$ is stable
in the sense that the size of $V(t)$
can  be kept under control. An exact statement to this effect
can be found in {\S}\ref{sec:stb}, 
but we here provide an informal summary.
\begin{thm}[see \S\ref{sec:stb}]
\label{thm:mr:stb}
Assume that (Hq), (HEq), (HDt), (HSt) and (HTw) all  hold and that $\s$
and $V_0$
are sufficiently small. Then for time scales up to $\O(\s^{-2})$,
the perturbation $V(t)$ remains small and
the phase $\G(t)$ accurately represents the position of $U(t)$
relative to the wave $(\Phi_\s,c_\s)$.
\end{thm}
As we discussed in {\S}\ref{sec:int}, we expect this result
to remain valid up to exponentially long time scales.
This is supported by the numerical evidence in \S\ref{sec:resN}.

\paragraph{Step 7: Expansion in $\sigma$.}
In order to investigate the fluctuations around
the instantaneous stochastic wave $(\Phi_\s, c_\s)$,
we choose $(V_0, \Gamma_0)= (0,0)$ and expand our equations
for $(V, \Gamma)$ in powers of $\sigma$.
In particular, we look for expansions of the form
\begin{equation}\label{eq:steps:expV}
V(t) = \sigma V^{(1)}_{\sigma}(t) + \s^2 V^{(2)}_\s(t)+ V_\mathrm{\mathrm{res}}(t)
\end{equation}
and
\begin{equation}\label{eq:steps:expG}
\Gamma(t) = c_{\sigma} t + \sigma \Gamma^{(1)}_{\sigma}(t)
+ \sigma^2 \Gamma^{(2)}_{\sigma}(t)+\s^3 \Gamma^{(3)}_\s(t) + \O(\sigma^4).
\end{equation}
For example, using \sref{eq:res:mildV} we may write
\begin{equation}
\label{eq:mr:expr:v:1:i}
V_\s^{(1)}(t) = \int_0^tS(t-s)\mathcal{S}_\s(0)dW^Q_s ,
\end{equation}
which can be substituted back into \sref{eq:res:mildV}
to find an expression for $V^{(2)}_\s(t)$ and so on.
In addition, using  \sref{eq:res:Gamma:int}
it is natural to write
\begin{equation}
\label{eq:mr:expr:gm:1:i}
    \Gamma_\s^{(1)}(t) = \int_0^t b_\sigma(0) dW^Q_s.
\end{equation}
Knowledge of $V^{(1)}_\sigma$ can subsequently be used to define $\Gamma^{(2)}_\sigma$, while $V^{(2)}_\sigma$ can be used
to compute $\Gamma^{(3)}_\sigma$.

We provide explicit formula's for these expansion terms in
{\S}\ref{sec:explexp} below. We mention here that we are including a $\sigma$-dependence in these terms as it often increases the readability to use $(\Phi_\s,c_\s)$ instead of $(\Phi_0,c_0)$.
For example, $\mathcal{S}_\s(0)$ can be expanded in terms
of $\sigma$ to yield $V^{(1)}_{\sigma}(t) = V^{(1)}_0(t) + \O(\sigma^2)$,
hence the difference between $\sigma V^{(1)}_{\sigma}(t)$ and $\s V^{(1)}_0(t)$ is only seen at third order.

\begin{cor}[see \S\ref{sec:stb}]
\label{cor:mr:res:bnd}
Assume that (Hq), (HEq), (HSt), (HDt) and (HTw) all hold.   Then $\sigma^{-2} V_{\mathrm{res}}$ remains small for
time scales up to $\O(\s^{-2})$.
\end{cor}

\paragraph{Step 8: Formal limits.}
We are now in a position to address
our main question concerning the average long-term
behaviour of the speed and shape of $U(t)$.
In particular, we are interested to see
if - and in what sense - it is possible
to define limiting quantities
\begin{equation}
\label{eq:mr:def:lim:wave}
    \big(\Phi_{\s;\mathrm{lim}},c_{\s;\mathrm{lim}} \big)
    = `\lim_{t\to\infty}\textrm{'} \, \,  E 
    \Big(
        U(\cdot + \Gamma(t), t),t^{-1} \Gamma(t)
    \Big).
\end{equation}

Any rigorous definition of such a limit most likely requires the use of 
carefully constructed stopping times, since our wave-tracking mechanism
almost surely fails at some finite time (see also {\S}\ref{sec:mr:pred}).
This delicate theoretical question is outside of the scope of the present paper unfortunately. 
From a practical point of view however, the numerical results in {\S}\ref{sec:resN}-\ref{sec:resFHN} displayed
in the stochastic reference frame $\Gamma(t)$ clearly indicate
that some type of fast convergence is taking place on long time scales.
In fact, by evaluating the averages in \sref{eq:mr:def:lim:wave}
for sufficiently large values of $t$, we construct
observed quantities 
$\big(\Phi^{\mathrm{obs}}_{\s;\mathrm{lim}},c^{\mathrm{obs}}_{\s;\mathrm{lim}} \big)$ that we feel are useful proxies for the limits \sref{eq:mr:def:lim:wave}.

We emphasize that we expect these quantities to differ from
the instantaneous stochastic wave $(\Phi_\s,c_\s)$. Indeed,
the stochastic forcing leads to an effect that we refer
to as `orbital drift'. Upon (formally) writing
\begin{equation}
\big( V^{\mathrm{od}}_{\s},c^{\mathrm{od}}_{\s}  \big)
=
    \big(\Phi_{\s;\mathrm{lim}} - \Phi_{\s},c_{\s;\mathrm{lim}} - c_{\s} \big)
\end{equation}
to quantify this difference, we note that
\begin{equation}
\label{eq:mr:def:c:v:od}
\begin{array}{lcl}
    c^{\mathrm{od}}_\s & = &
      `\lim_{t \to \infty}\textrm{'} \, \, E \, t^{-1} [\Gamma(t) - c_{\sigma} t ], \\[0.2cm]
    V^{\mathrm{od}}_\s & = & `\lim_{t \to \infty}\textrm{'} \, \,  E \, V(t) .
    \end{array}
\end{equation}
Of course, the same theoretical issues discussed above apply to these limits.

The key point however,
is that such limits \textit{do} exist naturally
for the individual terms in the expansions
\sref{eq:steps:expV}-\sref{eq:steps:expG}.
In particular, it \textit{is} possible to compute the expansions
\begin{equation}
\begin{array}{lcl}
    c^{\mathrm{od}}_{\s;i} & = &
      \lim_{t \to \infty} E \, t^{-1} \Gamma^{(i)}_{\s}(t), \\[0.2cm]
    V^{\mathrm{od}}_{\s;i} & = &
       \lim_{t \to \infty} E \, V^{(i)}_{\s}(t),
    \end{array}
\end{equation}
which allows us to compute approximations for
\sref{eq:mr:def:lim:wave} that can be explicitly
evaluated. In {\S}\ref{sec:resN}-\ref{sec:resFHN} we show that they
agree remarkably well with the observed numerical proxies
$\big(\Phi^{\mathrm{obs}}_{\s;\mathrm{lim}},c^{\mathrm{obs}}_{\s;\mathrm{lim}} \big)$ for \sref{eq:mr:def:lim:wave}.

Giving an interpretation to the pair $(\Phi_{\s;\mathrm{lim}},c_{\s;\mathrm{lim}})$ however is difficult. We do not have any ODE that it solves, but we think of $(\Phi_{\s;\mathrm{lim}},c_{\s;\mathrm{lim}})$ as the ceasefire line between the stochastic term that pushes the solution away from $(\Phi_\s,c_\s)$ and the exponential decay of the deterministic part that pushes it back to $(\Phi_\s,c_\s)$. We remark that it might be possible to embed $(\Phi_{\s;\mathrm{lim}},c_{\s;\mathrm{lim}})$ in some type of an invariant measure for the SPDE. There is a rich literature on the existence of invariant measures to stochastic Reaction-Diffusion equations, see e.g. \cite{Cerrai2005stabilization} and we intend to study this in the future.

\subsection{Explicit expansions}
\label{sec:explexp}
We now set out to explain in detail how the expansions
discussed in {\S}\ref{sec:derphase} can be derived. We give
general results here, but also show how they can be applied
to two explicit examples in \S\ref{sec:resN} and \S\ref{sec:resFHN}.

\paragraph{Expansions for $(\Phi_{\sigma}, c_{\sigma})$}
First, we examine the correction terms that are required
to obtain the instantaneous stochastic wave
from the deterministic wave $(\Phi_0, c_0)$.
In particular, we recall the defining identity
\begin{align}
\label{eq:mr:def:id:for:F:sigma}
F_0(\Phi_\s,c_\s)+\s^2F_{0;2}(\Phi_\s,c_\s)=0
\end{align}
and write
\begin{align}\label{eq:res:phi-c-clim}
    \begin{split}
        \Phi_\s&=\Phi_0+\s^2\Phi_{0;2}+\O(\s^4),\\
        c_\s&=c_0+\s^2c_{0;2}+\O(\s^4)\\
    \end{split}
\end{align}
for the solutions that are constructed in Proposition \ref{prp:var:swv:ex}.
We note that the $\O(1)$-terms in \sref{eq:mr:def:id:for:F:sigma}
indeed vanish because $F_0(\Phi_0, c_0) = 0$.
Balancing the $\O(\s^2)$-terms, we find
\begin{align}\label{eq:res:Phi02}
\begin{split}
\L_\mathrm{tw}\Phi_{0;2}&=-\frac{1}{2}\Phi_0''\ip{g(\Phi_0) Q g^T(\Phi_0)\psi_{\mathrm{tw}},\psit}_{L^2(\R,\R^n)}^2-c_{0;2}\Phi_0'+(g(\Phi_0)Qg^T(\Phi_0)\psi_{\mathrm{tw}})'-h(\Phi_0)\\
&=-F_{0;2}(\Phi_0,c_0)-c_{0;2}\Phi_0' .
\end{split}
\end{align}
By the Fredholm alternative, we know that we can solve for $\Phi_{0;2}$ when
the right hand side of this equation is orthogonal to $\psit$.
In view of the normalization $\ip{\Phi_0',\psit}_{L^2(\R,\R^n)}=1$, we
hence find
\begin{align}\label{eq:res:c02}
\begin{split}
c_{0;2}=&-\ip{F_{0;2}(\Phi_0,c_0),\psit}_{L^2(\R,\R^n)}.
\end{split}
\end{align}
The function $\Phi_{0;2}$ can now be computed
by numerically (or analytically when possible)
inverting $\L_\mathrm{tw}$ and solving \sref{eq:res:Phi02}.

\paragraph{First order: $(\Gamma_\s^{(1)} , V_\s^{(1)})$ }
We now turn our attention to the first order
terms in the expansions
\sref{eq:steps:expV}-\sref{eq:steps:expG}.
Expanding the expressions
\sref{eq:mr:expr:v:1:i}-\sref{eq:mr:expr:gm:1:i},
we may write
\begin{align}\label{eq:res:V1}
\begin{split}
    V_\s^{(1)}(t)&=\int_0^tS(t-s) g(\Phi_\sigma) dW^Q_s
    \\
    & \qquad -
          \int_0^t S(t-s) \Phi_\s'
             \frac{\ip{\psit,g(\Phi_\s) d W^Q_s}_{L^2(\R,\R^n)}}{\ip{\Phi_\s',\psit}_{L^2(\R,\R^n)}},
    \end{split}
\end{align}
together with
\begin{align}\label{eq:res:Gam1}
\begin{split}
    \G^{(1)}_\s(t)
    &=-\int_0^t\frac{\ip{\psit,g(\Phi_\s)dW^Q_s}_{L^2(\R,\R^n)}}{\ip{\Phi_\s',\psit}_{L^2(\R,\R^n)}}.
\end{split}
\end{align}

We note that $E[V_\s^{(1)}(t)]=E[\G_\s^{(1)}(t)]=0$.
On account of the decay of the semigroup, $V_\s^{(1)}$ can be regarded
as a process of Ornstein-Uhlenbeck type. On the other hand,
$\G^{(1)}_\s$ behaves as a scaled Brownian motion
with variance
  \begin{align}\label{eq:res:varG}
\begin{split}
    \mathrm{Var}(\G_\s^{(1)}(t))
    &=\ip{g(\Phi_\s)Qg^T(\Phi_\s)\psit,\psit}_{L^2(\R,\R^n)}^2t.
\end{split}
\end{align}

\paragraph{Second order: $(\Gamma_\s^{(2)} , V_\s^{(2)})$ }
Substituting the first order term $V^{(1)}_{\s}$
into the right-hand-side of \sref{eq:res:mildV},
we find that $V_\s^{(2)}(t)$
picks up a deterministic contribution coming
from the quadratic terms in $\mathcal{R}_{\s}$,
together with a stochastic contribution arising
from the linear terms in $\mathcal{S}_{\s}$.
In particular, we obtain
\begin{align}\label{eq:res:V2}
\begin{split}
    V_\s^{(2)}(t)=&\int_0^t S(t-s)\mathcal{R}_\s^{(2)}[V_\s^{(1)}(s),V_\s^{(1)}(s)]\,ds
    +\int_0^tS(t-s)\mathcal{S}_\s^{(1)}\big(V_\s^{(1)}(s)\big) \,dW^Q_s,
    \end{split}
    \end{align}
in which we have
    \begin{align}\begin{split}
 \mathcal{R}^{(2)}_\s[V,V]=&\frac{1}{2}D^2f(\Phi_\s)[V,V]
    -\frac{1}{2}\Phi_\s' \frac{\ip{D^2f(\Phi_\s)[V,V],\psi_{\mathrm{tw}}}_{L^2(\R,\R^n)}}{\ip{\Phi_\s',\psit}_{L^2(\R,\R^n)}},\\
\end{split}
\end{align}
together with
\begin{align}\begin{split}
\mathcal{S}_\s^{(1)}(V)[w]=&Dg(\Phi_\s)[V]w
    -\p_\xi V\frac{\ip{\psi_{\mathrm{tw}},g(\Phi_\s)w}_{L^2(\R,\R^n)}}{\ip{\Phi_\s',\psit}_{L^2(\R,\R^n)}}\\
    &-\Phi_\s'\frac{\ip{\psi_{\mathrm{tw}},Dg(\Phi_0)[V]w}_{L^2(\R,\R^n)}}{\ip{\Phi_\s',\psit}_{L^2(\R,\R^n)}}\\
    &+\Phi_\s'\ip{\p_\xi V,\psi_{\mathrm{tw}}}_{L^2(\R,\R^n)}\frac{\ip{\psi_{\mathrm{tw}},g(\Phi_\s)w}_{L^2(\R,\R^n)}}{\ip{\Phi_\s',\psit}_{L^2(\R,\R^n)}^2}
    \end{split}
\end{align}
for any $w \in L^2_Q$.
In a similar fashion, we find
\begin{align}\label{eq:res:Gam2}\begin{split}
    \Gamma_\s^{(2)}(t)=&\int_0^ta_\s^{(2)}(\Phi_\s)[V_\s^{(1)}(s),V_\s^{(1)}(s)] \, ds
    +\int_0^tb_\s^{(1)}(\Phi_\s)[V_\s^{(1)}(s)]\, dW^Q_t,
\end{split}
\end{align}
in which we have
\begin{align}\label{eq:res:Gam2partsa}\begin{split}
    a_\s^{(2)}[V,V]=&-\frac{1}{2}\ip{D^2f(\Phi_\s)[V,V],\psi_{\mathrm{tw}}}_{L^2(\R,\R^n)},\\
\end{split}
\end{align}
together with
\begin{align}\label{eq:res:Gam2partsb}\begin{split}
    b_\s^{(1)}(\Phi_\s)[V][w]=&-\frac{\ip{\psit,Dg(\Phi_\s)[V]v}_{L^2(\R,\R^n)}}{\ip{\Phi_\s',\psit}_{L^2(\R,\R^n)}}\\
    &  -\ip{\p_\xi V,\psit}_{L^2(\R,\R^n)}\frac{\ip{\psit,g(\Phi_\s)w}_{L^2(\R,\R^n)}}{\ip{\Phi_\s',\psit}_{L^2(\R,\R^n)}^2}.
\end{split}
\end{align}

Note that the expressions for $V_\s^{(2)}(t)$ and $\G_\s^{(2)}(t)$ depend on $f$, but not on $h$. This is due to the fact that the $\O(\s^2)$ part of the It\^o-Stratonovich correction term is already absorbed in $(\Phi_\s,c_\s)$.
If we would have started our computations around $(\Phi_0,c_0)$, the dependence of $h$ would show up via $(\Phi_{0;2},c_{0;2})$.
However the extra second order terms together form  \sref{eq:res:Phi02} and therefore vanish.

We remark that both of these second-order terms have a nonzero expectation,
which can be explicitly computed using
the It\^o lemma.
To this end,
we introduce the notation
\begin{align}\label{eq:res:DefK}
K_\s(s)[w_1,w_2]=\frac{1}{2}D^2f(\Phi_\s)[S(s)\mathcal{S}_\s(0)w_1,S(s)\mathcal{S}_\s(0)w_2]
\end{align}
for any $v,w\in L_Q^2$. Upon choosing a basis $(e_k)$ of $L^2(\R,\R^m)$ and applying Lemma \ref{lem:ito},
we find
\begin{align}\label{eq:res:EGam2}
E[\G_\s^{(2)}(t)]=-\int_0^t\int_0^s \sum_{k=0}^\infty\ip{K_\s(s')[\sqrt Qe_k,\sqrt Qe_k],\psi_{\mathrm{tw}}}_{L^2(\R,\R^n)} \, ds'ds ,
\end{align}
together with
\begin{align}\label{eq:res:EV2}
\begin{split}
E[V_\s^{(2)}(t)]=\int_0^tS(t-s)\int_0^s\sum_{k=0}^\infty &
\Big[ K_\s(s')[\sqrt Qe_k,\sqrt Qe_k]\\
   &-\Phi'_\s\frac{\ip{K_\s(s')
   [\sqrt Qe_k,\sqrt Qe_k],\psi_{\mathrm{tw}}}_{L^2(\R,\R^n)}}
   {\ip{\Phi_\s',\psit}_{L^2(\R,\R^n)}}
\Big]
ds'ds.
    \end{split}
\end{align}
Sending $t\to\infty$, we can explicitly compute
\begin{align}\label{eq:res:od}
c_{\s;2}^\mathrm{od}=\lim_{t\to\infty}t^{-1}E[\G_\s^{(2)}(t)]=-\int_0^\infty \sum_{k=0}^\infty\ip{K_\s(s')[\sqrt Qe_k,\sqrt Qe_k],\psi_{\mathrm{tw}}}_{L^2(\R,\R^n)} \, ds.
\end{align}
Note that this integral converges because
$\mathcal{S}_\s(0)$ is orthogonal to $\psit$,
which circumvents the nondecaying mode of the semigroup.

In a similar fashion, we can obtain
\begin{equation}
    V_{\s;2}^\mathrm{od}
      = \lim_{t \to \infty} 
         E \, \big[ V_{\s}^{(2)}(t) \big].
\end{equation} 
Switching the integrals in \sref{eq:res:EV2} and applying the operator identity \cite[Prop. 1.3.6]{lorenzi2004analytic}
\begin{align}
    \L_\mathrm{tw}\int_0^tS(s)ds=S(t)-I,
\end{align}
we arive at
\begin{align}\label{eq:res:odV}
V_{\s;2}^\mathrm{od}=-\L_\mathrm{tw}^{-1}\int_0^\infty\sum_{k=0}^\infty
\Big[ K_\s(s)[\sqrt Qe_k,\sqrt Qe_k]
   &-\Phi'_\s\frac{\ip{K_\s(s)
   [\sqrt Qe_k,\sqrt Qe_k],\psi_{\mathrm{tw}}}_{L^2(\R,\R^n)}}
   {\ip{\Phi_\s',\psit}_{L^2(\R,\R^n)}}
\Big]ds.
\end{align}

\paragraph{Third order: $\Gamma_\s^{(3)}$ }
Provided that the nonlinearities are sufficiently smooth,
the methods in the previous paragraphs
can in principle be extended to any desired order in $\sigma$,
but the computations get more involved. However, it is important to note that in order to compute the $n$-th order approximation of $\G(t)$, we only need information from $V(t)$ up to order $n-1$.
In particular, upon inspecting equation \sref{eq:res:Gamma:int} we find that
\begin{equation}
\begin{split}
\sigma^3 \G^{(3)}_\s(t)
=&  \int_0^t a_\s\big(\s V_\s^{(1)}(s)+\s^2V_\s^{(2)}(s)\big)\,ds
+\s\int_0^t b_\s\big(\s V_\s^{(1)}(s)+\s^2V_\s^{(2)}(s)\big) \,dW^Q_s\\
&-\s\G_\s^{(1)}(t)-\s^2\G_\s^{(2)}(t)+ O (\sigma^4) .
\end{split}
\end{equation}
This gives us a convenient numerical procedure
to compute $c^{\mathrm{od}}_{0;3}$ without having
to explicitly compute $a^{(3)}_{\sigma}$
and $b^{(2)}_{\sigma}$. Indeed, we may write
\begin{equation}
\label{eq:res:cubic}
 c^{\mathrm{od}}_{0;3}
 = \lim_{\sigma \to 0, \, t \to \infty}
   \sigma^{-3} E\Big[\int_0^t a_\s\big(\s V_\s^{(1)}(s)+\s^2V_\s^{(2)}(s)\big) \, ds-\s^2\G_\s^{(2)}(t)\Big].
\end{equation}

\subsection{Predictions}
\label{sec:mr:pred}
Based upon the perturbation analysis in the previous section, we can make the following predictions on the behaviour of the wave.

\paragraph{Diffusive phase wandering}

At leading order in $\sigma$, we see that the phase wanders diffusively
around the deterministic position $c_0 t$. Indeed,
based upon \sref{eq:res:varG} we predict that
\begin{align}\label{eq:res:varG:full}
\begin{split}
    \mathrm{Var}\big(\G(t)\big)
    =\sigma^2 \ip{g(\Phi_\s)Qg^T(\Phi_\s)\psit,\psit}_{L^2(\R,\R^n)}^2t
      +\O(\s^{3}).
\end{split}
\end{align}
Note that this expression 
coincides with 
the mean square deviation from the deterministic phase $E[(\G(t)-c_0t)^2]$ up to $\O(\s^3)$. In the specific case of the
stochastic Nagumo equation, this expression
has been known for two decades already \cite[eq.
(6.25)]{GarciaSpatiallyExtended}.
Similar identities (with $g(u) = 1$)
were found for almost translationally invariant additive noise
\cite[\S 3.4]{StannatKruger} and in the context
of neural field equations \cite{Bressloff,Lang}.
Remark that the difference between the
It\^o and Stratonovich interpretation cannot yet be observed
at this level.

\paragraph{Short term behaviour}
Based on \sref{eq:res:EGam2} we see that on short timescales
we have
\begin{equation}
    E [\Gamma_{\sigma}^{(2)}(t) ] \sim t^2,
\end{equation}
which does not contribute meaningfully to the speed for small $t$.
Similarly, we have
\begin{equation}
  \mathrm{Var} [ V_\s^{(1)}(t) ] \sim t,
  \qquad
    E [V_{\sigma}^{(2)}(t) ] \sim t^2,
\end{equation}
which shows that also the shape of the wave is relatively unaffected
by these correction terms. In particular, we see that on short timescales
the pair $(\Phi_\s, c_\s)$ indeed accurately describes the shape
and speed of the wave. We feel that this justifies
the use of our  `instantaneous stochastic wave' terminology.

\paragraph{Long term behaviour}
On longer timescales the Ornstein-Uhlenbeck-like process
$V^{(1)}_\s$ starts to play an important role,
causing fluctuations around $(\Phi_\s, c_\s)$
that lead to the orbital drift corrections.
Using the It\^o isometry,
we predict that the size of the perturbations
behaves as
\begin{align}\label{eq:res:EnrmV}
\begin{split}
    E[\nrm{V(t)}_{L^2(\R,\R^n)}^{2}]&=\sigma^2E[\nrm{V_0^{(1)}(t)}_{L^2(\R,\R^n)}^2]+\O(\s^3)\\
    &=\s^2\int_0^t\nrm{S(s)\mathcal{S}_0(0)}^2_{HS\big(L_Q^2,L^2(\R,\R^n)\big)} \, ds+\O(\s^3),
    \end{split}
\end{align}
see \cite[Ex. 4]{DaPratomildArx}.
We point out that the integral actually
converges (at an exponential rate) as $t \to \infty$. Naturally,
the residual is predicted to behave as
\begin{align}\label{eq:res:Vres}
    E[\nrm{V_\mathrm{res}(t)}_{L^2(\R,\R^n)}^{2}]&=E[\nrm{V(t)-V_\s^{(1)}(t)-V_\s^{(2)}(t)}_{L^2(\R,\R^n)}^{2}]\sim\O(\s^6).
\end{align}

Although the average of $V_0^{(1)}(t)$ can be kept under control,
our phase tracking breaks down as soon as
$\norm{V(t)}_{L^2(\R,\R^n)}$ exceeds a certain $\sigma$-independent threshold. Based on the hitting-time estimates for the scalar Ornstein-Uhlenbeck process that were obtained in \cite[Eq. (6a)]{nobile1985exponential},
we conjecture that the expected break-down time increases
exponentially with respect to $\sigma^{-2}$.
In a similar vein, the quantity of interest
for our stability analysis
is the expectation of the \textit{supremum} of
$\nrm{V(t)}^2$ over the interval $[0,T]$. 
Based on detailed
and very delicate computations for the standard scalar Ornstein-Uhlenbeck
process \cite{ricciardi1988first,alili2005representations,pickands1969}, we conjecture that 
$E \sup_{0\le t \le T} \, \nrm{V_0^{(1)}(t)}^2_{L^2(\R,\R^n)}
\sim \ln T$,
while 
the corresponding expression
for $V_0^{(2)}$ grows
as $\ln^2 T$;
see also Fig. \ref{fig:resN:supa}.

Both these conjectures
suggest that our framework remains valid for $T$ up to $\O(e^{\eta \s^{-2}})$ for some small $\eta > 0$.
However, it is not clear to us how the bounds above can be obtained
in the infinite-dimensional semigroup setting. In our rigorous
stability proof we are therefore forced to work with a weaker $\O(\sigma^2 T)$ bound for the supremum expectation,
which understates the timescales over which we can keep track
of the stochastic wave. 

Once we have lost track of the wave, it could potentially be possible to restart the tracking mechanism by allowing for an instantaneous jump in the phase. 
In \cite[\S 7]{Inglis} some first promising results in this direction were obtained by defining the phase as the - possibly discontinuous - solution to a global minimization problem.

Turning to the limiting speed and shape of the wave,
we arrive at the prediction
\begin{align}\label{eq:res:expclim}
\begin{split}
    c_{\s;\mathrm{lim}}&=c_\s+\s^2c_{\s;2}^{\mathrm{od}}+\s^3c_{\s;3}^{\mathrm{od}}+\O(\s^4)\\
    &=c_0+\s^2[c_{0;2}+c_{0;2}^{\mathrm{od}}]+\s^3c_{0;3}^\mathrm{od}+\O(\s^4),
    \end{split}
\end{align}
where we used
$(\Phi_\s,c_\s)=(\Phi_0,c_0)+\O(\s^2)$ to conclude
that the difference between $c_{\s;2}^\mathrm{od}$
and $c_{0;2}^\mathrm{od}$ is also of order $\s^2$.
In a similar fashion, we obtain
\begin{align}\label{eq:res:EU}
\begin{split}
    \Phi_{\s;\mathrm{lim}}
     &=\Phi_\s+\s^2 V^{\mathrm{od}}_{0;2}+\O(\s^3)\\
    &=\Phi_0+\s^2\big[\Phi_{0;2}+ V^{\mathrm{od}}_{0;2} \big]+\O(\s^3).
\end{split}
\end{align}
The leading order terms in the expressions 
\sref{eq:res:expclim}-\sref{eq:res:EU}
can all be explicitly computed,
which will allow us to test our predictions
against numerical simulations
in \S\ref{sec:resN} and \S\ref{sec:resFHN}.

\section{Example I: The Nagumo Equation}
\label{sec:resN}
In this section, we study the 
explicit example
\begin{align}\label{eq:resN:nageq}
\begin{split}
    dU&=[\p_{xx}U+f_\mathrm{cub}(U)+\frac{\mu\s^2}{2}q(0)g'(U)g(U)]dt+\s g(U)dW^Q_t,\\
\end{split}
\end{align}
in which $\mu$ is either zero (It\^o) or one (Stratonovich),
while the nonlinearities are given by
\begin{equation}
 f_\mathrm{cub}(U)=U(1-U)(U-\a),
     \qquad \qquad
     g(U)=U(1-U)
\end{equation}
for some $\a \in (0,1)$.
We do remark that $g$ does not have a bounded second derivative as demanded
by our assumption (HSt). This technical problem can be remedied by applying a cut-off function to $g(U)$ to ensure that this value levels off for $U>>1$.

Following \cite{Lord2012},
we use the normalized kernel
\begin{align}
    q(x)=\frac{1}{2\zeta}e^{\frac{-\pi x^2}{4\zeta^2}}
\end{align}
to generate
the cylindrical
$Q$-Wiener process $W^Q_t$ over $L^2( \Real)$.
Here the parameter $\zeta>0$ is a measure for the spatial correlation length,
which is defined \cite{GarciaSpatiallyExtended}
as the second moment of $q$, i.e.  $\frac{2\zeta}{\pi}$.
The kernel $p$ of $\sqrt Q$ can be computed 
by taking the inverse Fourier transform
of $\sqrt{\widehat{q}}$;
see Appendix \ref{sec:est:prlm}.
This yields
\begin{align}
    \begin{split}
        p(x)&=
        \sqrt[4]{\frac{\pi}{2}}\frac{e^{-\frac{\pi x^2}{2\zeta^2}}}{\zeta}.
    \end{split}
\end{align}
Notice that the dimensions of the problem
are $n=m=1$, which means that $g=g^T$.

We now set out to carefully perform the computations in \S\ref{sec:explexp} and compare the results with
our numerical simulations.
These simulations are based on the algorithms from \cite[Chap. 10]{Lord2014Book}. In particular, we use a semi-implicit scheme in time and a straight-forward central-difference
discretization in space. In addition, we use circulant embedding \cite[Alg. 6.8]{Lord2014Book} to generate a stochastic Wiener process with the prescribed spatial correlation function.

\paragraph{Computing $(\Phi_\sigma, c_\sigma)$}
As explained in {\S}\ref{sec:int}, the wave $(\Phi_0,c_0)$
satisfies the ODE
\begin{align}
    \Phi_0''+c_0\Phi_0'+f_\mathrm{cub}(\Phi_0)=0
\end{align}
and is given by
\begin{align}\label{eq:resN:wave}
    \begin{split}
        \Phi_0&=\frac{1}{2}\left[1-\tanh\left(\frac{1}{2\sqrt 2}x\right)\right],
        \qquad
        c_0 
        =\sqrt{2}\left(\frac{1}{2}-\a\right).
    \end{split}
\end{align}
The linear operators $\L_{\mathrm{tw}}$ and $\L_\mathrm{tw}^*$ act as
\begin{align}\label{eq:resN:linop}
    \begin{split}
        \L_\mathrm{tw}v&=v''+c_0v'+f_\mathrm{cub}'(\Phi_0)v,
        \qquad
        \L^*_\mathrm{tw}w
         =w''-c_0w'+f_\mathrm{cub}'(\Phi_0)w
    \end{split}
\end{align}
and we write $\Phi_0'$ respectively $\psit(x)=e^{c_0x}\Phi_0'/\ip{\Phi'_0,e^{c_0 \cdot}\Phi'_0}_{L^2(\R)}$ for their normalized simple eigenfunctions at zero.

\begin{figure}
\centering
\begin{subfigure}{.49\textwidth}
  \centering
  \includegraphics[width=1\columnwidth]{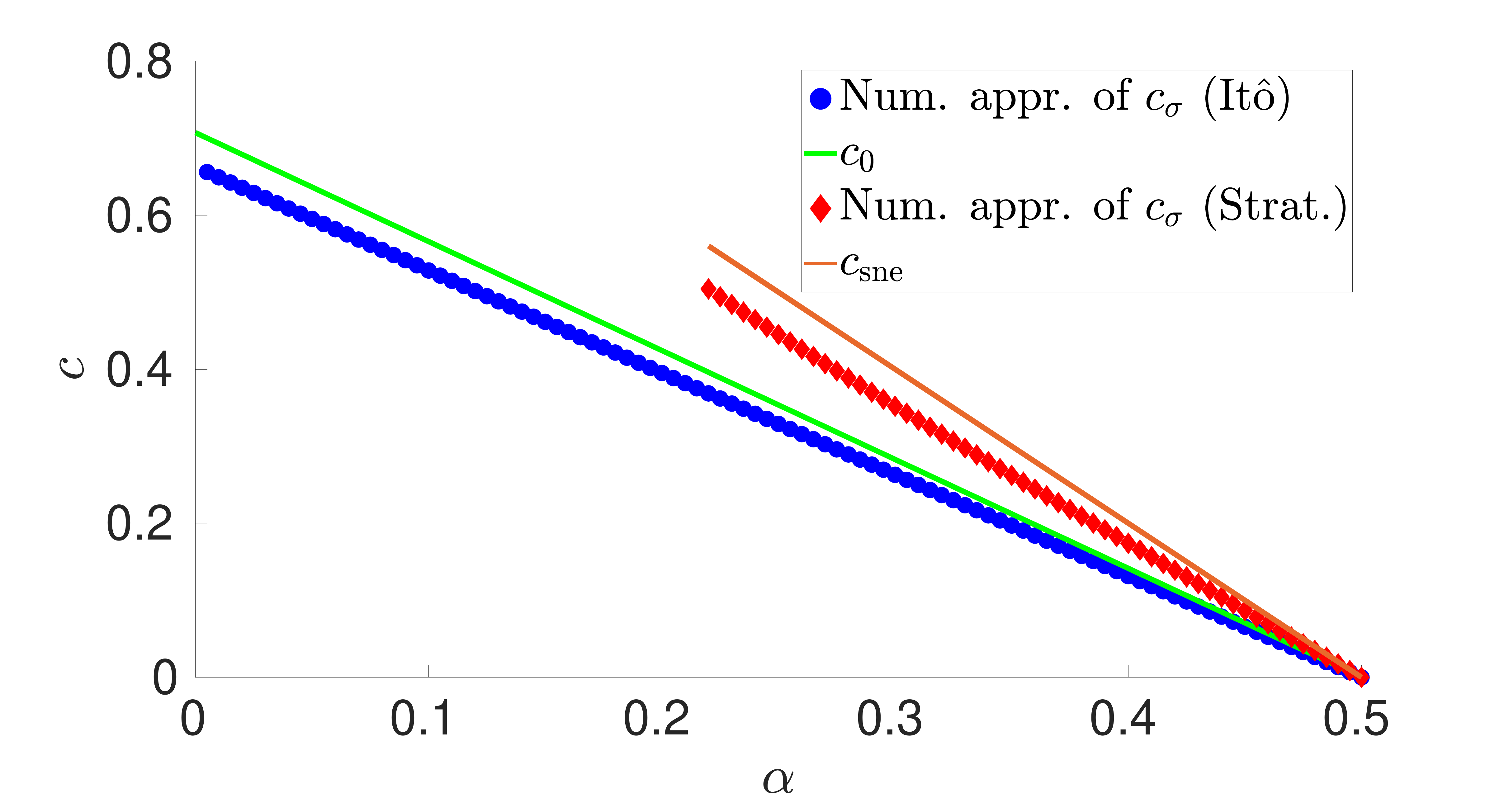}
  \caption{}
      \label{fig:resN:cvsa}
\end{subfigure}%
\begin{subfigure}{.49\textwidth}
  \centering
\includegraphics[width=1\columnwidth]{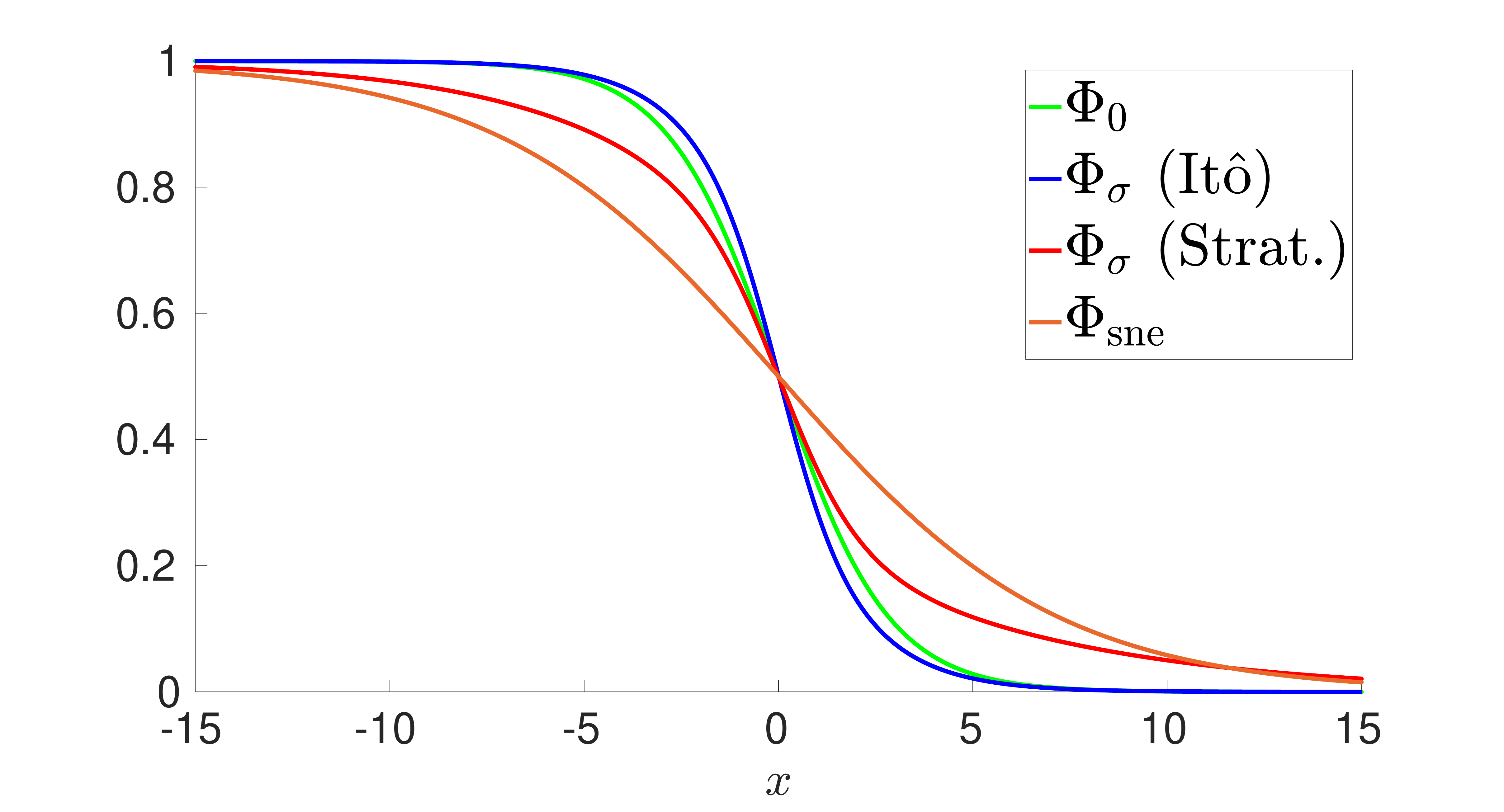}
    \caption{}
        \label{fig:resN:phiItoStrat}
\end{subfigure}
\caption{Panel (a) 
compares the deterministic wave speed
\sref{eq:resN:wave} (green),
the instantaneous stochastic wave speed
$c_{\s}$ for the Ito (blue) and
Stratonovich (red) interpretations
and the speed $c_{\mathrm{sne}}$
derived in
\cite[eq. (6.33)]{GarciaSpatiallyExtended}
(orange), all for $\sigma = 1$
and $\zeta = 1$. The red and orange lines
are only plotted for $\a_{\mathrm{eff}} \in (0,1/2)$.
Panel (b)
compares the associated wave profiles
$\a=0.45$ and $\s=1.3$.
Notice the steepening and flattening
of the waves for the It\^o
respectively Stratonovich interpretations.
The profiles are computed on the interval $[-40,40]$, but here zoomed in to $[-15,15]$ to highlight the differences.}
\end{figure}

In this scalar setting, the full equation
$F_\s(\Phi_\s,c_\s)=0$ can be written as
\begin{align}\label{eq:resN:odePhi}
\begin{split}
    \Phi_\s''+c_\s\Phi_\s'+f_\mathrm{cub}(\Phi_\s)=&-\frac{\s^2}{2}\frac{\ip{q*(g(\Phi_\s)\psit),g(\Phi_\s)\psit}_{L^2(\R)}}{\ip{\Phi_\s',\psit}^2_{L^2(\R)}}\Phi_\s''\\&+\s^2\frac{\big(g(\Phi_\s)q*(g(\Phi_\s)\psi_\mathrm{tw})\big)'}{\ip{\Phi'_\s,\psit}_{L^2(\R)}}-\frac{\mu\s^2}{2}q(0)g'(\Phi_\s)g(\Phi_\s).
\end{split}
\end{align}
It is interesting to compare this equation
with the system
\begin{align}\label{eq:resN:odePhisne}
\begin{split}
\Phi_\mathrm{sne}''+c_\mathrm{sne}\Phi_\mathrm{sne}'+f_\mathrm{cub}(\Phi_\mathrm{sne})= -\frac{\mu\s^2}{2}q(0)g'(\Phi_\mathrm{sne})g(\Phi_\mathrm{sne})
\end{split}
\end{align}
used to construct the waves $(\Phi_\mathrm{sne},c_{\mathrm{sne}})$
in \cite{GarciaSpatiallyExtended} using their so-called small noise expansion technique. As the authors remark, this equation is not the result of a systematic perturbative expansion in $\s$, but rather a partial resummation of such an expansion. For example,
the additional two $O(\sigma^2)$ terms in \sref{eq:resN:odePhi}
arise from the second order terms in the It\^o formula
which were neglected in \cite{GarciaSpatiallyExtended}.

In any case,
for $\mu = 1$ we can rewrite \sref{eq:resN:odePhisne}
in the explicit form
\begin{align}
\Phi_\mathrm{sne}''+c_\mathrm{sne}\Phi_\mathrm{sne}'
+   \big(1-\s^2q(0)\big)u(1-u)\Big(u - \a_{\mathrm{eff}}\Big) =0,
\end{align}
with a new effective detuning parameter
\begin{equation}
    \a_{\mathrm{eff}} = \frac{2\a - \s^2q(0)}{2-2\s^2q(0)}.
\end{equation}
This equation is just a scaled version of the original ODE, which
can be solved by rescaling \sref{eq:resN:wave} as  
\begin{align}\label{eq:resN:waveSNE}
    \begin{split}
        \Phi_\mathrm{sne}&=\frac{1}{2}\left[1-\tanh\left(\frac{\sqrt{1-\s^2q(0)}}{2\sqrt 2}x\right)\right],
        \qquad
        c_\mathrm{sne}
        =\sqrt{2\big(1-\s^2q(0)\big)}(\frac{1}{2}-\a_{\mathrm{eff}}).
    \end{split}
\end{align}

Our full system \sref{eq:resN:odePhi}
cannot be solved explicitly, but in the bistable
regime $\a_{\mathrm{eff}} \in (0,1)$ we were able
to use a straightforward
numerical boundary value problem solver to
approximate the solutions; see Fig. \ref{fig:resN:cvsa}.
These results show  that $c_\mathrm{sne}$ is a reasonable
approximation for $c_\s$, but in Fig. \ref{fig:resN:CubicSpeed} we shall see that $c_\mathrm{sne}$ compares less favourably with the full limiting wave speed.
Note that our solutions are in agreement
with the numerical observations from \cite{Lord2012}: for the Stratonovich interpretation the wave moves faster and is less steep, but for the It\^o interpretation the wave slows down and becomes steeper.

We now turn to expanding $(\Phi_\s,c_\s)$ in powers of
$\sigma$. Following \sref{eq:res:c02}, the lowest order correction to $c_\s$  becomes
\begin{align}\label{eq:resN:c02}
\begin{split}
c_{0;2}=&-\frac{1}{2}\ip{\Phi_0'',\psi_{\mathrm{tw}}}_{L^2(\R)}\ip{q*(g(\Phi_0)\psi_{\mathrm{tw}}),g(\Phi_0)\psi_{\mathrm{tw}}}_{L^2(\R)}-\ip{g(\Phi_0)q*(g(\Phi_0)\psi_{\mathrm{tw}}),\psi_{\mathrm{tw}}'}_{L^2(\R)}\\
& \qquad -\frac{\mu q(0)}{2}\ip{g'(\Phi_0)g(\Phi_0),\psit}_{L^2(\R)}.
\end{split}
\end{align}
We can subsequently find $\Phi_{0;2}$
by numerically inverting the linear operator $\L_\mathrm{tw}$
to solve
\begin{align}\label{eq:resN:Phi02}
\begin{split}
\L_\mathrm{tw}\Phi_{0;2}&=-\frac{1}{2}\Phi_0''\ip{q*(g(\Phi_0)\psi_{\mathrm{tw}}),\psit}_{L^2(\R)}^2-c_{0;2}\Phi_0'
+\Big(g(\Phi_0)q*(g(\Phi_0)\psi_{\mathrm{tw}})\Big)'
\\
& \qquad \qquad -\frac{q(0)}{2}g'(\Phi_0)g(\Phi_0).
\end{split}
\end{align}
We remark that these approximations can also be evaluated for additive noise ($g=1$), or, in the It\^o interpretation, for $q(x-y)=\delta(x-y)$.
In Fig. \ref{fig:resN:csPhis} we compare $(\Phi_\s-\Phi_0,c_\s-c_0)$ with our quadratic approximations for a range of different values of $\s$. There appears to be a good agreement, both for the It\^o and Stratonovich interpretation.

\paragraph{Limiting wave speed}

\begin{figure}
\centering
\begin{subfigure}{.33\textwidth}
\includegraphics[width=1\columnwidth]{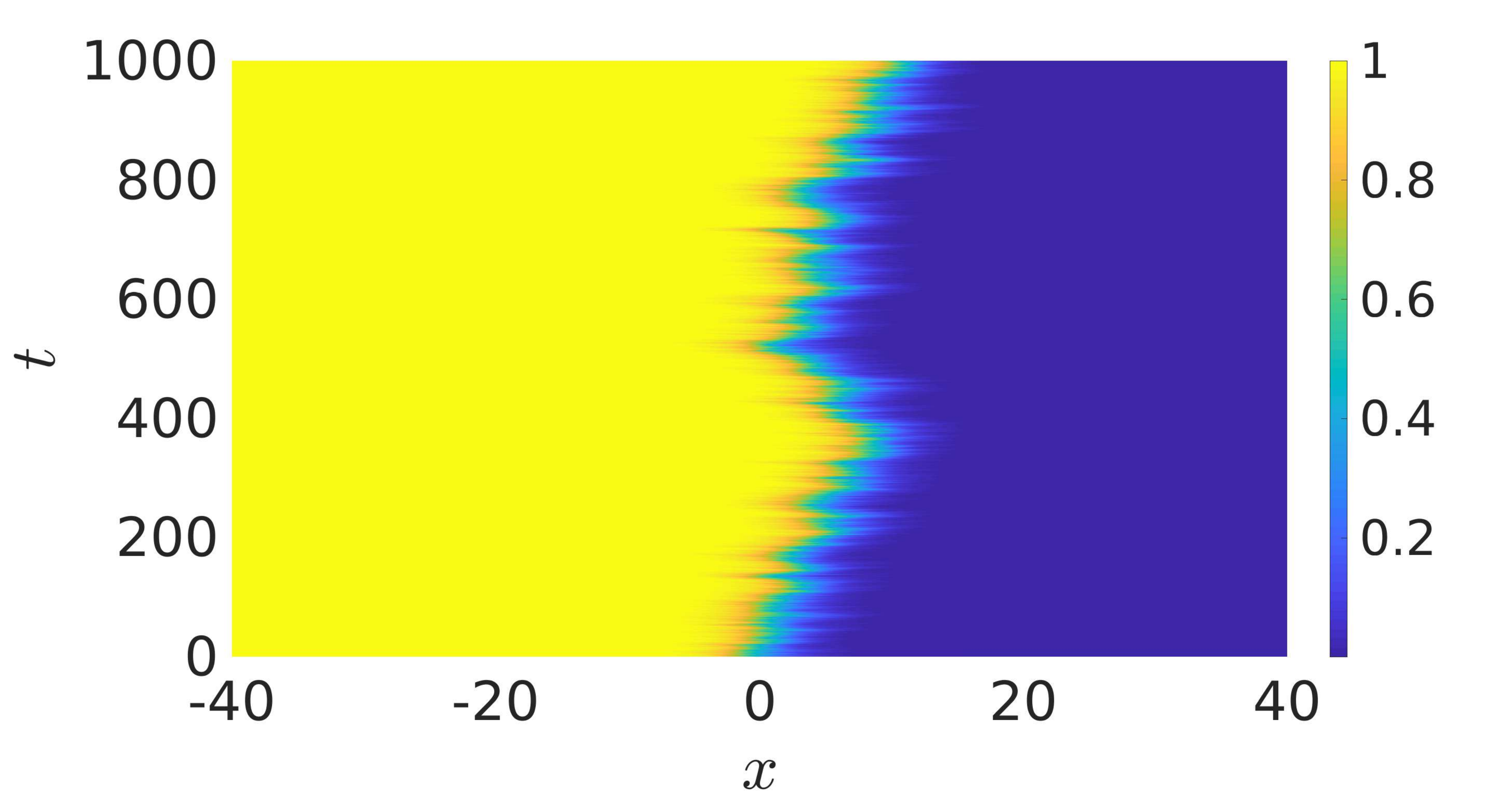}
  \caption{$U(x +c_0t,t)$}
    \label{fig:UDifRefa}
\end{subfigure}%
\begin{subfigure}{.33\textwidth}
  \centering
\includegraphics[width=1\columnwidth]{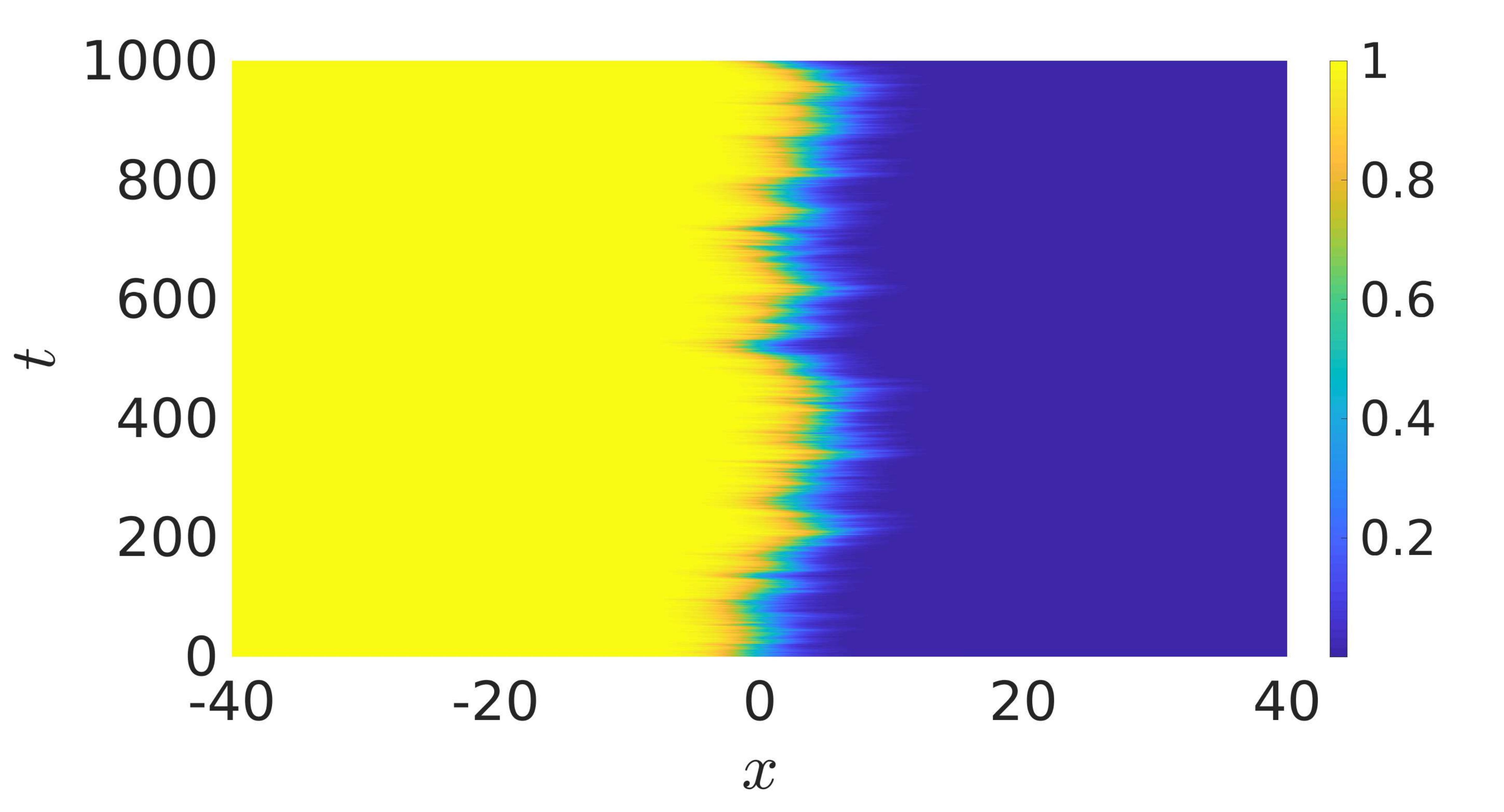}
  \caption{$U(x +c_\sigma t,t)$}
      \label{fig:UDifRefb}
\end{subfigure}%
\begin{subfigure}{.33\textwidth}
  \centering
\includegraphics[width=1\columnwidth]{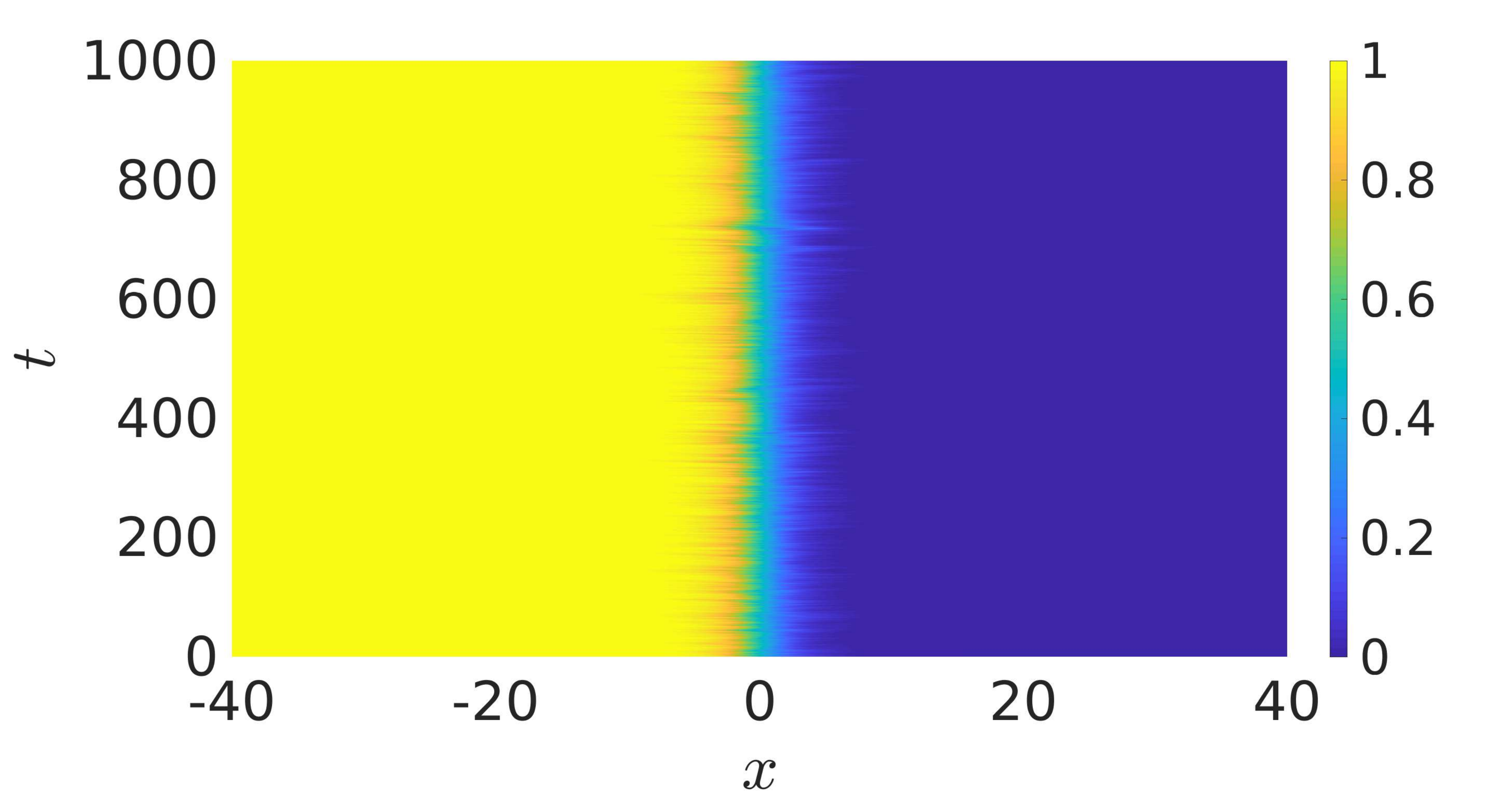}
  \caption{$U(x +\Gamma(t),t)$}
        \label{fig:UDifRefc}
\end{subfigure}
\caption{A single realization of  \sref{eq:resN:nageq} in the Stratonovich interpretation with initial condition $\Phi_\s$ in 3 different reference frames, with parameters $\a=0.25$, $\sigma=0.3$ and $\zeta=1$. We can clearly see in (a) that the deterministic speed underestimates the stochastic speed. Replacing $c_0$ with $c_\s$ in (b) captures the movement better, but the position is still fluctuating. Panel (c) shows that 
these fluctuations can be 
captured almost completely in the $\Gamma(t)$-frame. 
}
\label{fig:resN:UDifRef}
\end{figure}

In order to provide some
insight on the effectiveness of our
stochastic phase $\G(t)$,
Fig. \ref{fig:resN:UDifRef}
describes the behaviour of $U(t)$
for a single realization of \sref{eq:resN:nageq}
in three different reference frames.
The first panel shows the wave in the deterministic
co-moving frame, which clearly underestimates the wave speed. Replacing
the speed $c_0$ by $c_\s$ gives a better approximation, but the wave is still wandering. The right panel shows that these fluctuations can be largely eliminated by using $\Gamma(t)$,
confirming that this an appropriate representation for the position of the wave.

At leading order, the fluctuations around $c_\s t$
are described  by the scaled Brownian motion 
\begin{align}
\Gamma_0^{(1)}(t)=\int_0^t\ip{\psit,g(\Phi_0)dW^Q_s}_{L^2(\R)} .
\end{align}
The corresponding variance is given by
\begin{align}\label{eq:resN:varG}
    \mathrm{Var}\big(\sigma \G_0^{(1)}(t)\big)&= \sigma^2 \ip{q*(g(\Phi_0)\psit),g(\Phi_0)\psit}_{L^2(\R)}t,
    \end{align}
which exactly matches \cite[eq. (6.25)]{GarciaSpatiallyExtended}.
Since $E \, \Gamma_0^{(1)}(t) = 0$, the orbital drift corrections to the
limiting wave speed are only visible at second order in $\sigma$.
In particular, the lowest order contribution
given in \sref{eq:res:od} reduces to
\begin{align}\label{eq:resN:od}
\begin{split}
c_{0;2}^\mathrm{od}&=-\frac{1}{2}\int_0^\infty \sum_{k=0}^\infty\ip{f_\mathrm{cub}''(\Phi_0)
\Big(S(s)\mathcal{S}_0(0)[p*e_k] \Big)^2
,\psi_{\mathrm{tw}}}_{L^2(\R)}\, ds.\\
\end{split}
\end{align}
Here the square is taken
in a pointwise fashion, with
\begin{align}
\label{eq:resN:def:s0}
\begin{split}
\mathcal{S}_0(0)[p*e_k]&=g(\Phi_0)p*e_k-\Phi'_0 \ip{p*e_k,g(\Phi_0)\psit}_{L^2(\R)}.
\end{split}
\end{align}
In order to evaluate this expression for
$c_{0;2}^\mathrm{od}$,
we need to choose an 
appropriate orthonormal basis
for $L^2([-L,L];\R)$, 
where $[-L,L]$ is the domain that we use for the numerical simulations. Following \cite{Lord2012, Shardlow}, we take
\begin{align}\label{eq:resN:eigf}
e^{(L)}_{k,c}(x)=\frac{1}{\sqrt{L}}\cos(\frac{\pi kx}{L}),
\qquad
\qquad
e^{(L)}_{k,s}(x)=\frac{1}{\sqrt{L}}\sin(\frac{\pi kx}{L})
\end{align}
for all integers $k \ge 0$

and introduce the quantities
\begin{equation}
    \lambda_{k;\mathrm{apx}} = \mathrm{exp}[- \pi k^2 \zeta^2 / L^2].
\end{equation}
A short computation shows that
\begin{align}
    Qe^{(L)}_{k,c}=q*e^{(L)}_{k,c}=\int_{-L}^Lq(\cdot-y)
    e^{(L)}_{k,c}(y)dy\approx\int_{-\infty}^\infty q(\cdot-y)e^{(L)}_{k,c}(y)dy=
    \lambda_{k;\mathrm{apx}} e^{(L)}_{k,c}
\end{align}
and in the same fashion we find
$Q e^{(L)}_{k,s} \approx \lambda_{k;\mathrm{apx}} e^{(L)}_{k,s}$.
These observations can be used
to approximate the expression
\sref{eq:resN:od} by writing
\begin{align}\label{eq:resN:c:od:2:apx}
\begin{split}
c^{\mathrm{od}}_{0;2}
\approx - \frac{1}{2} \int_0^\infty \sum_{k=0}^{150}\sum_{\#\in\{c,s\}}\lambda_{k;\mathrm{apx}}
\ip{f_\mathrm{cub}''(\Phi_0)
\big(S(s)\mathcal{I}^{(L)}_{k\#} \big)^2
,\psi_{\mathrm{tw}}}_{L^2([-L,L];\R)}\, ds,
    \end{split}
\end{align}
in which we have
\begin{equation}
\label{eq:resN:def:i:l:k}
    \mathcal{I}^{(L)}_{k\#}
     = g(\Phi_0)e^{(L)}_{k\#}-\Phi'_0\ip{e^{(L)}_{k\#},g(\Phi_0)\psit}_{L^2([-L,L];\R)}.
\end{equation}

We verified numerically
that the resulting sum converges
exponentially fast in both $L$ and $k$.

In order to approximate the cubic coefficient
$c^{\mathrm{od}}_{\s;3}$, we use the fact
that
$\G_\s^{(3)}(t)$ depends only on $V_\s^{(1)}(t)$ and $V_\s^{(2)}(t)$.
In particular, we 
made the approximation
\begin{equation}
    \sigma^3 c^{\mathrm{od}}_{0;3}
    \approx c^{\mathrm{od}}_{\mathrm{cub}}(\s)
\end{equation}
by numerically computing
\begin{align}
c^{\mathrm{od}}_{\mathrm{cub}}(\s)
= \frac{2}{T}\int^T_{\frac{T}{2}}\frac{1}{t}E\big[
\G_{\mathrm{apx}}(t)
-c_\s t-\s\G_\s^{(1)}(t)-\s^2\G_\s^{(2)}(t)\big] \, dt,
\end{align}
in which 
\begin{equation}
    \Gamma_{\mathrm{apx}}(t)
    = c_\s t + \int_0^t a_\s\big(
    \s V_\s^{(1)}(s)+\s^2V_\s^{(2)}(s) \big) \,ds
    + \int_0^t b_\s( \s V_\s^{(1)}(s)+\s^2V_\s^{(2)}(s) \big) \,
     d W^Q_s
\end{equation}
denotes the value for $\Gamma(t)$
that is obtained by integrating
\sref{eq:res:Gamma:int}
using only the second order approximation
of $V$.

Putting everything together,
we obtain the prediction
\begin{align}
\label{eq:resN:c:pred}
    c^{\mathrm{pred}}_{\s;\mathrm{lim}}=c_0+\s^2[c_{0;2}+c_{0;2}^\mathrm{od}]+
    c^{\mathrm{od}}_{\mathrm{cub}}(\s)
    +\O(\s^4).
\end{align}
To get a feeling for the sizes of the perturbations in the Stratonovich interpretation, we
remark that our computations
for $\a=0.25$ and $\zeta=1$ yield
\begin{align}
\begin{split}
    c^{\mathrm{pred}}_{\s;\mathrm{lim}}
    &=0.3536+\s^2[0.056-0.0043]+0.0036\s^3+\O(\s^4).
\end{split}
\end{align}
Clearly, the contribution from the orbital drift is significantly smaller then the contribution from $c_\s$.

To test this prediction, we numerically computed
a proxy for the limiting wave speed by evaluating
the integral
%
\begin{align}
\label{eq:resN:od:obs}
 c_{\s;\mathrm{lim}}^{\mathrm{obs}}
 =c_\s+\frac{2}{T} \int_{\frac{T}{2}}^T \frac{1}{t}E[\Gamma(t)-c_\s t-\s\Gamma_\s^{(1)}(t)] \,dt ,
\end{align}
which computes the average speed over the interval $[T/2,T]$ in order to remove any transients from the data.
Note that subtracting $\G_\s^{(1)}(t)$ does not change the average but speeds up the convergence towards the average.
This computation is motivated by the plots
of $E[\Gamma(t)-c_\s t]$ contained in Fig. \ref{fig:Drifta},
which have a clear linear trend.
This validates the concept of a limiting wavespeed,
but also illustrates the need to include the orbital
drift corrections to the instantaneous wavespeed $c_{\sigma}$.

In Fig. \ref{fig:resN:CubicSpeed} we show the relative deviation of $c_{\s;\mathrm{lim}}$ from $c_0$, i.e. $(c_{\s;\mathrm{lim}}-c_0)/c_0$. The blue dots represent the numerically observed values.
The red dashed line 
shows the quadratic approximation  $c_0+\s^2[c_{0;2}+c_{0;2}^\mathrm{od}]$ and there is indeed a good correspondence.

We also provide a cubic approximation to the wave speed by adding the
term $c^{\mathrm{od}}_{\mathrm{cub}}(\s)$.
This indeed improves the prediction,
validating our computations. However, it also shows that the improvement is small and might not be worth the effort. 

For completeness, we also included the predictions \sref{eq:resN:odePhisne}
arising from the small noise expansion
technique.
The results show that these predictions capture
 the overall behaviour of the limiting speed 
  correctly, but the values deviate significantly.

\begin{figure}
    \centering
        \begin{subfigure}{0.5\textwidth}
\includegraphics[width=1\columnwidth]{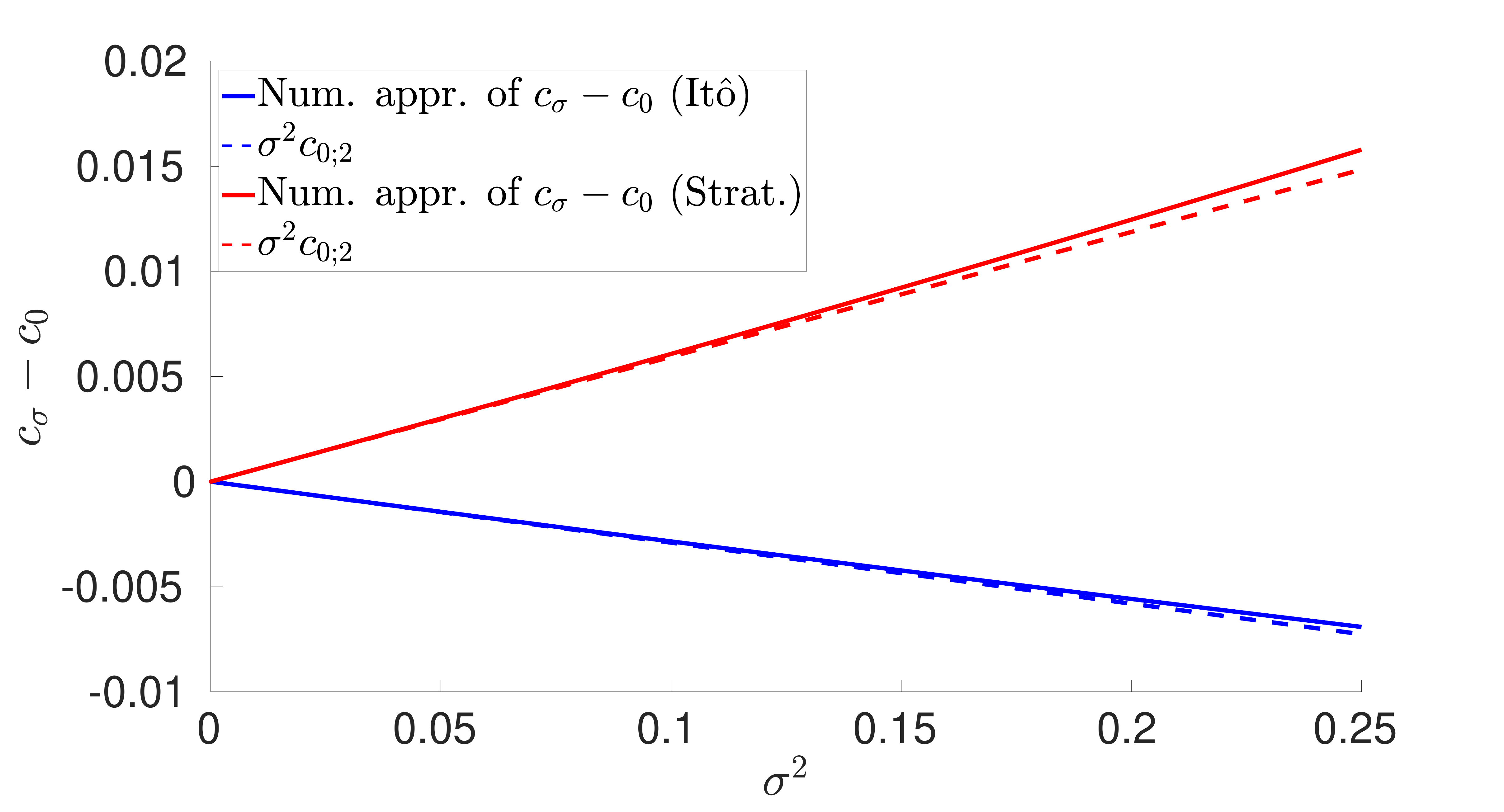}
        \caption{}
            \label{fig:resN:cs}
    \end{subfigure}%
    \begin{subfigure}{0.5\textwidth}
\includegraphics[width=1\columnwidth]{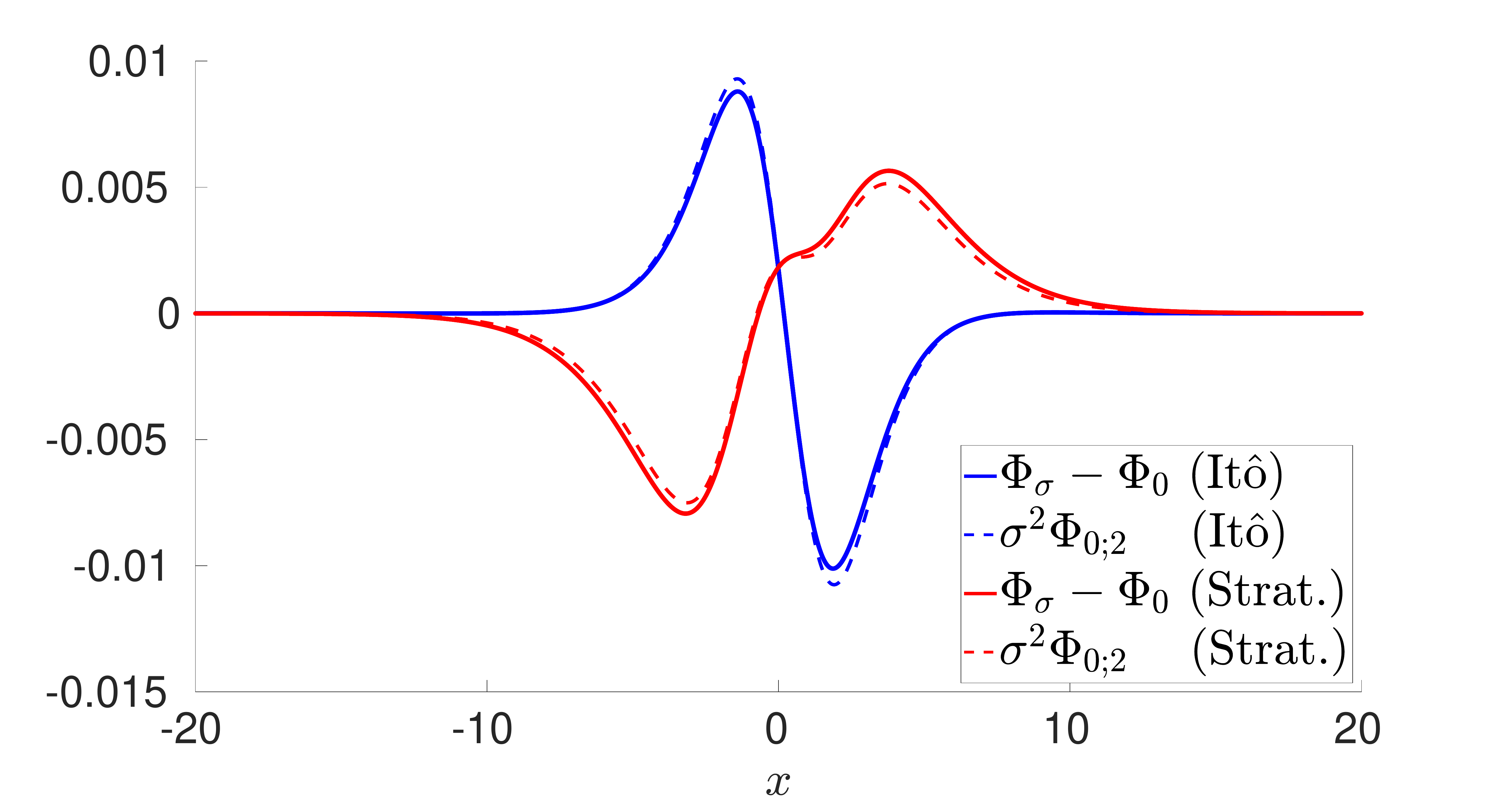}
        \caption{}
            \label{fig:resN:phi}
    \end{subfigure}
    \caption{These panels display the stochastic corrections $c_\s - c_0$ 
    and $\Phi_\s - \Phi_0$ for the wave speed (a) and wave profile (b), together
    with their leading order approximations.
    We chose $\a=0.25$ and $\zeta=1$, which results in $c_{0;2}=-0.0298$ (It\^o)
    and $c_{0;2}=0.0563$ (Stratonovich).
    The profiles in (b) were computed
    for $\sigma = 0.5$.
    }
    \label{fig:resN:csPhis}
\end{figure}

\begin{figure}
\centering
\begin{subfigure}{.5\textwidth}
  \centering
 		\def\svgwidth{\columnwidth}
    		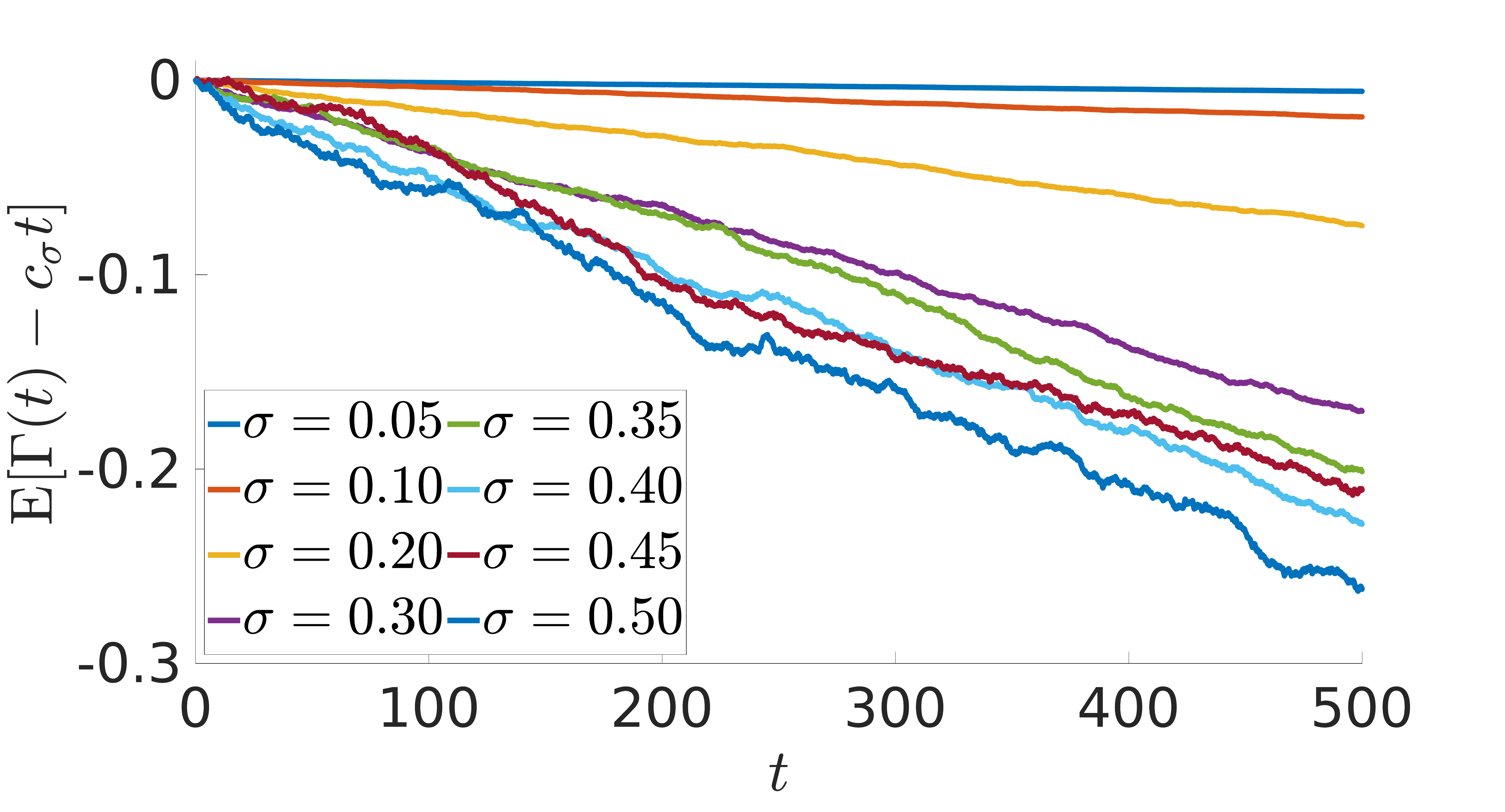
  \caption{}
  \label{fig:Drifta}
\end{subfigure}%
\begin{subfigure}{.5\textwidth}
\includegraphics[width=1\columnwidth]{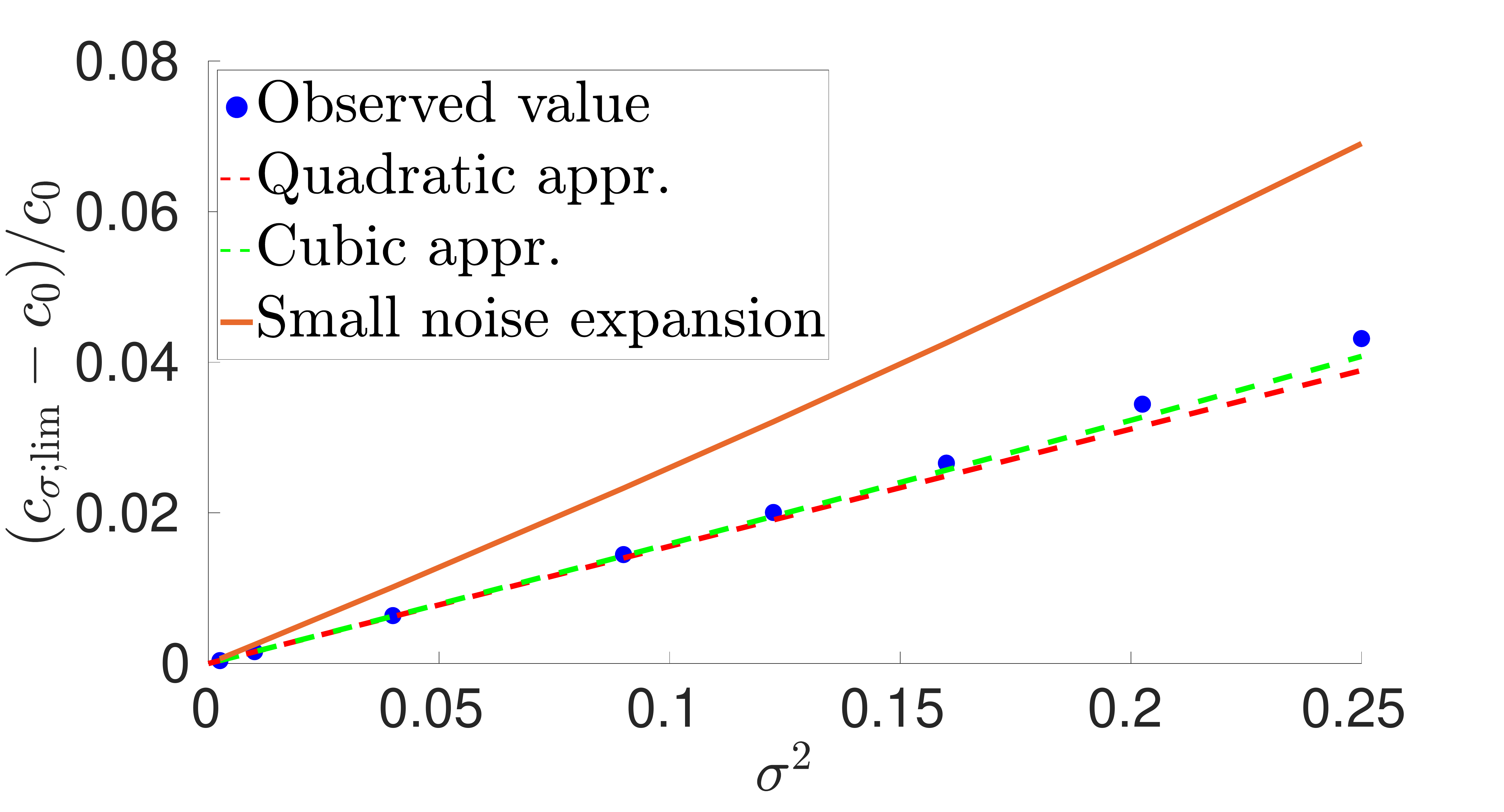}
  \caption{}
    \label{fig:resN:CubicSpeed}
\end{subfigure}
\caption{In (a) we computed the average $E[\Gamma(t)-c_\s t]$
over 1000 simulations of \sref{eq:res:V-final} for the Stratonovich interpretation, using
the procedure described in the main text for several values of $\s$.
Notice that a clear trend is visible, which validates the orbital drift principle. In (b) we show the relative deviation of $c_{\s;\mathrm{lim}}$ from $c_0$.
%
Here the observed limiting speed is computed by evaluating the average \sref{eq:resN:od:obs}
for the data in (a), while the 
quadratic and cubic approximations
were computed using the relevant
terms in \sref{eq:resN:c:pred}.
The orange line is the prediction
arising from
the small noise expansion
\sref{eq:resN:waveSNE}.
Both plots use $\a=0.25$ and $\zeta=1$.
}
\label{fig:Drift}
\end{figure}

\begin{figure}
\centering
\begin{subfigure}{.5\textwidth}
  \centering
\includegraphics[width=1\columnwidth]{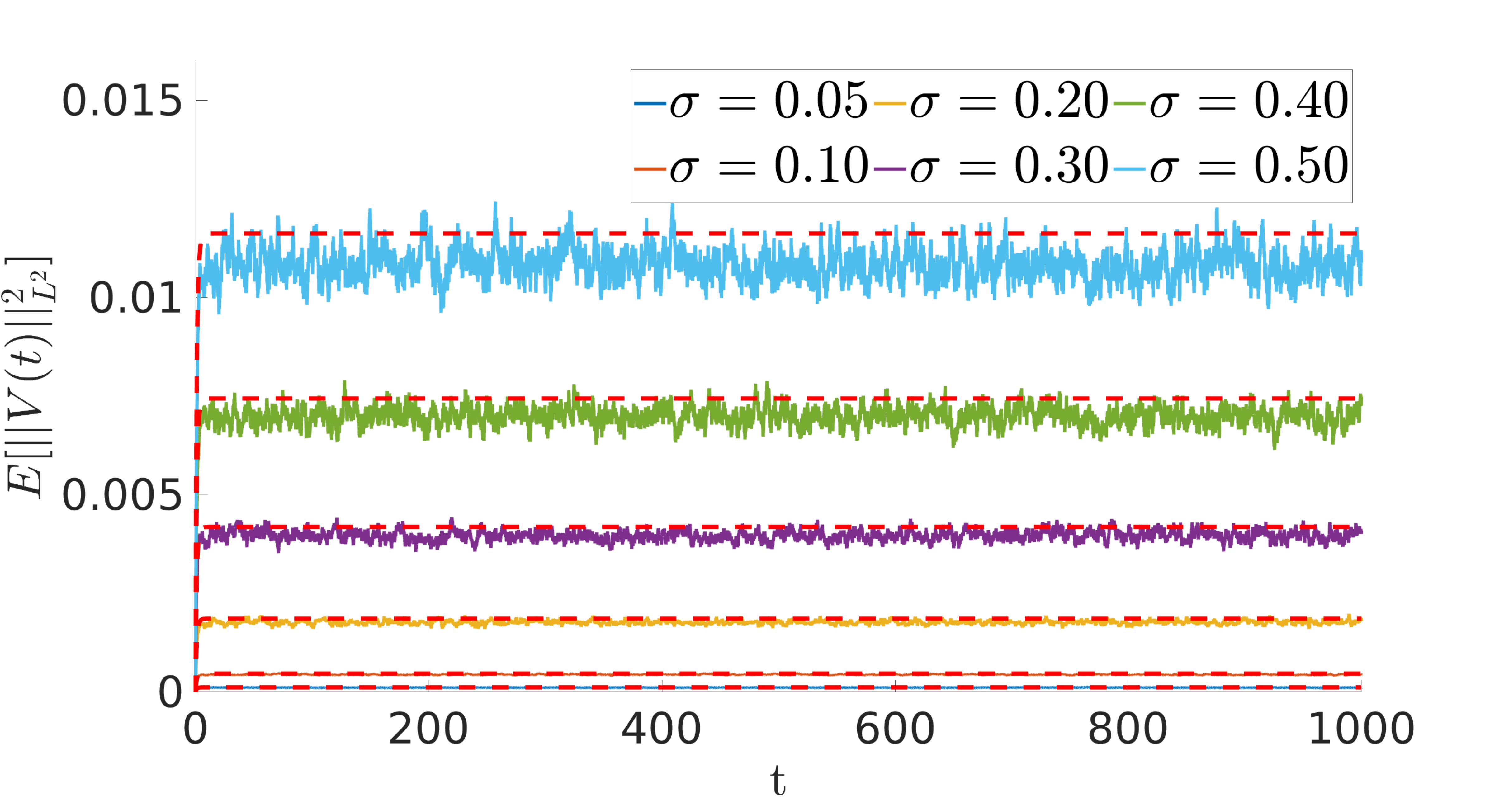}
  \caption{}
  \label{fig:resN:VvsT}
\end{subfigure}%
\begin{subfigure}{.5\textwidth}
  \centering
\includegraphics[width=1\columnwidth]{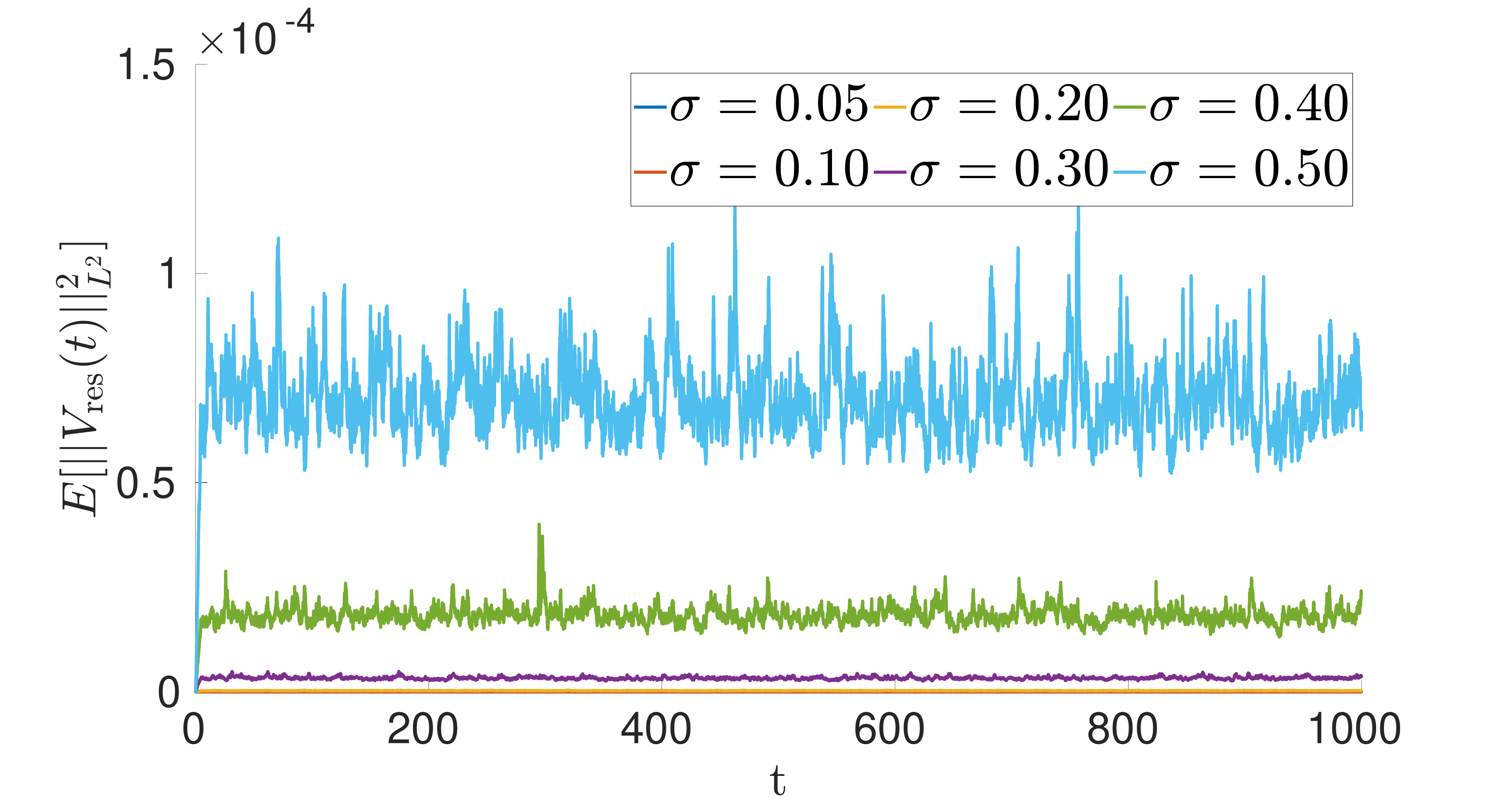}
\caption{}
    \label{fig:resN:RvsT}
\end{subfigure}
\caption{In (a) we computed the average $E[\nrm{V(t)}_{L^2}^2]$
over 1000 realization of  
\sref{eq:res:V-final} in the It\^o interpretation. The dashed line shows the numerical evaluation of the first order term
\sref{eq:resN:nrmV1}.
%
In (b) we computed the corresponding 
averages for the residual \sref{eq:resN:res} by evaluating and subtracting $\s V_\s^{(1)}(t)$ and $\sigma^2 V_\s^{(2)}(t)$ for every realization in (a). Note that both $V(t)$ and $V_\mathrm{res}(t)$ stabilize over time.}
\label{fig:VRvsT}
\end{figure}

\paragraph{Size of  $V(t)$}
Next, we turn our attention to the size
of the perturbation
$V(t)$ defined in \sref{eq:res:V-final}.
Although the leading order term $V_0^{(1)}(t)$ has zero mean,
this does not hold for its norm. Indeed, 
using \sref{eq:res:EnrmV} we find
\begin{align}\label{eq:resN:nrmV1}
\begin{split}
    E[\nrm{V_0^{(1)}(t)}_{L^2(\R)}^{2}]&=\int_0^t\nrm{S(s)\mathcal{S}_0(0)}^2_{HS\big(L^2_Q,L^2(\R)\big)}ds\\
    &=\int_0^t\sum_{k=0}^\infty\nrm{S(s)[g(\Phi_0)p*e_k-\Phi'_0\ip{g(\Phi_0)\psit,p*e_k}_{L^2(\R)}]}^2_{L^2(\R)} \, ds.
    \end{split}
\end{align}
This expectation can be approximated using the same basis functions and eigenvalues 
that we used for the orbital drift. Stated more concretely,
we recall \sref{eq:resN:def:i:l:k} and write
\begin{align}\label{eq:resN:nrmV1-num}
\begin{split}
    E[\nrm{V_0^{(1)}(t)}_{L^2(\R)}^{2}] \approx \int_0^t\sum_{k=0}^{150}\sum_{\#\in\{c,s\}}\lambda_{k;\mathrm{apx}}\nrm{S(s)\mathcal{I}^{(L)}_{k\#}}^2_{L^2([-L,L];\R)} \, ds.
    \end{split}
\end{align}

This function is represented by the red dashed line in Fig. \ref{fig:resN:VvsT}.
This agrees well with the numerical average
of $E[\nrm{V(t)}_{L^2(\R)}^2]$ that we computed
directly from our simulations.
The exponential behaviour for short time scales as well as
the longer term stabilization are nicely
captured by these results.
We note that we expect this limiting value
to be of order $\O(\s^2)$. This is confirmed
by Fig. \ref{fig:resN:VRvsS},
which shows how $E\nrm{V(T)}_{L^2(\R)}^2$
scales with $\s$ for $T = 1000$.

Similar behaviour was found during our simulations
for the residual
\begin{align}\label{eq:resN:res}
   V_\mathrm{res}(t)=V(t)-\s V_\s^{(1)}(t)-\s^2 V_\s^{(2)}(t).
\end{align}
Indeed, Fig. \ref{fig:resN:RvsT} shows that this residual also stabilizes exponentially fast to a small value which we expect to be $\O(\s^6)$,
as confirmed in Fig. \ref{fig:resN:VRvsS}.

We emphasize that we do not expect the running supremum
of $\norm{V(t)}_{L^2}$ to stabilize in the same fashion.
Indeed, we numerically computed $E[\sup_{0\leq s\leq t}\nrm{V(s)}^2_{L^2(\R)}]$ for $0\leq t\leq 1000$.
The results strongly suggest
that this supremum grows logarithmically in time (Fig. \ref{fig:resN:supa}) and scales as
$\s^2$ for large fixed $t$ (Fig. \ref{fig:resN:supb}).
This is hence significantly better
than the $\O(\s^2 t)$ bound that arises
from the Burkholder-Davis-Gundy inequality and confirms
our belief that our approach can be used
to track waves over time scales that are exponential
in $\sigma$.

\paragraph{Limiting wave profile}
Since $E[V_\s^{(1)}(t)] =0$, we expect the leading
order contribution to the average of $V(t)$ to be given by
$\sigma^2 E[V_0^{(2)}(t)]$.
Using \sref{eq:resN:def:s0} once more, 
we find that \sref{eq:res:EV2} can be written as
\begin{align}\label{eq:resN:EV2}
\begin{split}
E[V_0^{(2)}(t)]=\frac{1}{2}\int_0^tS(t-s)\int_0^s\sum_{k=0}^\infty\Big[& f''_\mathrm{cub}(\Phi_0)\left(S(s')\mathcal{S}_0(0)[p*e_k]\right)^2\\
 &-\Phi'_0\ip{f''_\mathrm{cub}(\Phi_0)\left(S(s')\mathcal{S}_0(0)[p*e_k]\right)^2,\psi_{\mathrm{tw}}}_{L^2(\R)}\Big] ds'ds.
    \end{split}
\end{align}
This can be evaluated using the same expressions for the eigenvalues and eigenfunctions that we used for the orbital drift.
In order to compare this to our simulations,
we numerically approximated
 $E[V(t)]$ by taking the average over
500 simulations of $V(t)-\s V_\s^{(1)}(t)$.
Since $E[V_\s^{(1)}(t)]=0$, this again
speeds up the convergence to the mean. The results are contained in
 Fig. \ref{fig:resN:EVEV2}, which
 shows that  $\s^2E[V^{(2)}_0(t)]$ is indeed very good approximation for $E[V(t)]$.
 These plots also 
 show that the average shape indeed appears to converge to a limit, motivating us to write
\begin{equation}
\label{eq:res:N:obs:shape}
    \Phi^{\mathrm{obs}}_{\s;\mathrm{lim}}
    = \Phi_\s + E[V(20)].
\end{equation}

We recall our prediction
\begin{equation}
\label{eq:res:N:pred:shape}
\Phi^{\mathrm{pred}}_{\s;\mathrm{lim}}
 =\Phi_0 + \s^2[\Phi_{0;2}+V_{0;2}^\mathrm{od}]+ O(\s^3)
\end{equation}
for the limiting wave profile. We can now numerically approximate 
the expression \sref{eq:res:odV}
for $V^\mathrm{od}_{0;2}$
by computing
\begin{align}
\begin{split}
V^\mathrm{od}_{0;2}\approx-
\frac{1}{2}\mathcal{L}_{\mathrm{tw}}^{-1}\int_0^{T}\sum_{k=0}^{150}\sum_{\#\in\{c,s\}}\l_{k;\mathrm{apx}}\Big[& f''_\mathrm{cub}(\Phi_0)\big(S(s)
  \mathcal{I}^{(L)}_{k\#}\big)^2\\
 &-\Phi'_0\ip{f''_\mathrm{cub}(\Phi_0)\big(S(s)\mathcal{I}^{(L)}_{k\#} \big)^2,\psi_{\mathrm{tw}}}_{L^2([-L,L];\R)}\Big] \, ds.
    \end{split}
\end{align}
To test our prediction, we compare 
$\Phi^{\mathrm{obs}}_{\s;\mathrm{lim}}-\Phi_0$ against 
$\s^2[\Phi_{0;2}+V_{0;2}^\mathrm{od}]$ 
for multiple values of $\s$.
The results are plotted in Fig. \ref{fig:resN:2ndOrder},
which again confirms that
there is a good match.

\begin{figure}
\centering
\begin{subfigure}{.5\textwidth}
  \centering
 		\def\svgwidth{\columnwidth}
    		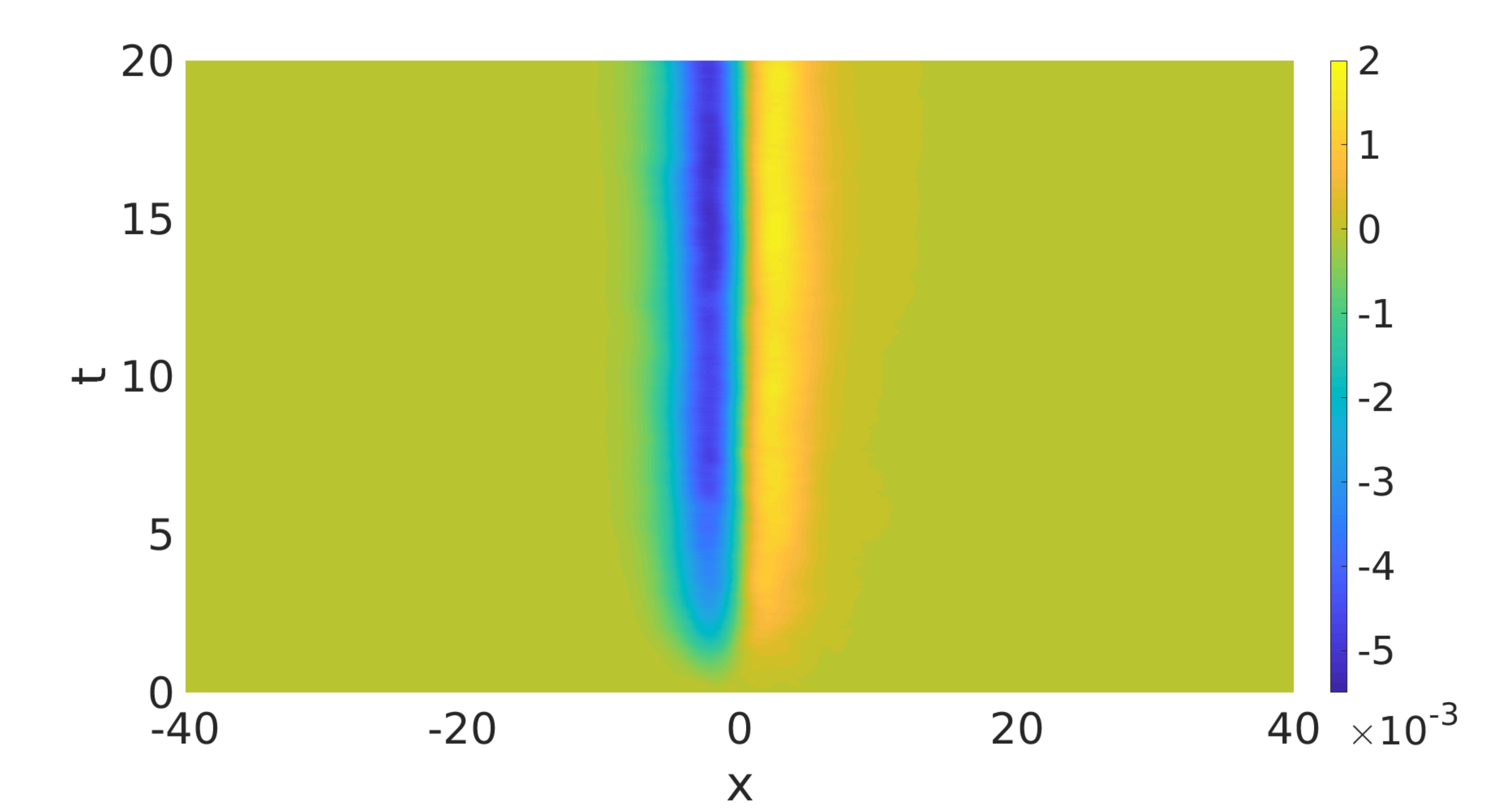
  \caption{$E[V(t)]$}
  \label{fig:resN:EV}
\end{subfigure}%
\begin{subfigure}{.5\textwidth}
  \centering
 		\def\svgwidth{\columnwidth}
    		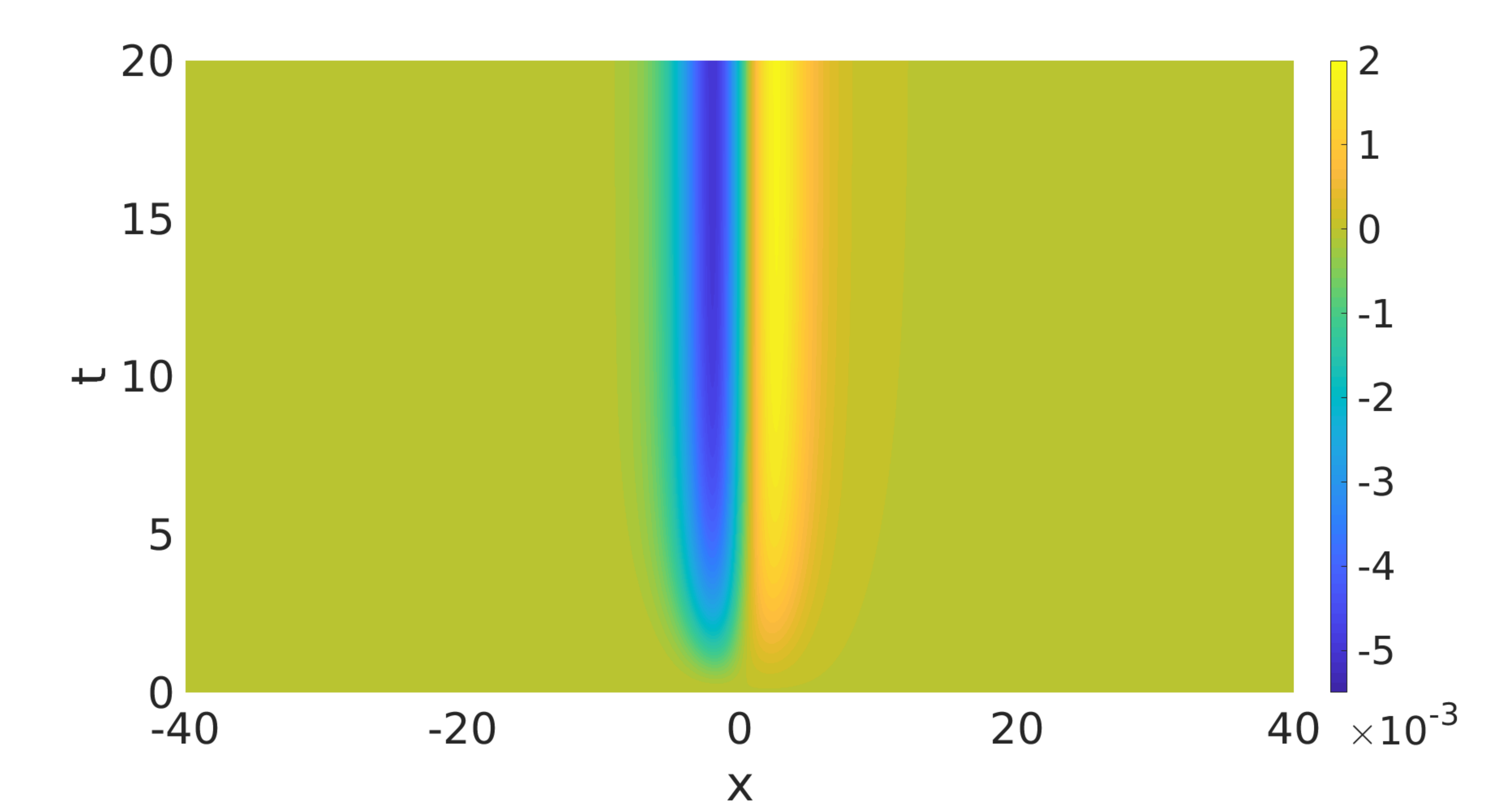
  \caption{$\s^2E[V_0^{(2)}(t)]$}
    \label{fig:resN:EV2}
\end{subfigure}
\caption{Panel (a) displays the average of $V(t)$ over 500 iterations
of \sref{eq:resN:nageq} in the Stratonovich interpretation for $\a=0.25$, $\zeta=1$ and $\s=0.5$. Panel (b) contains a numerical evaluation of \sref{eq:resN:EV2} 
that includes the
first 150 terms of the sum.
}
\label{fig:resN:EVEV2}
\end{figure}

\begin{figure}
\centering
\begin{subfigure}{.49\textwidth}
\centering
\includegraphics[width=1\columnwidth]{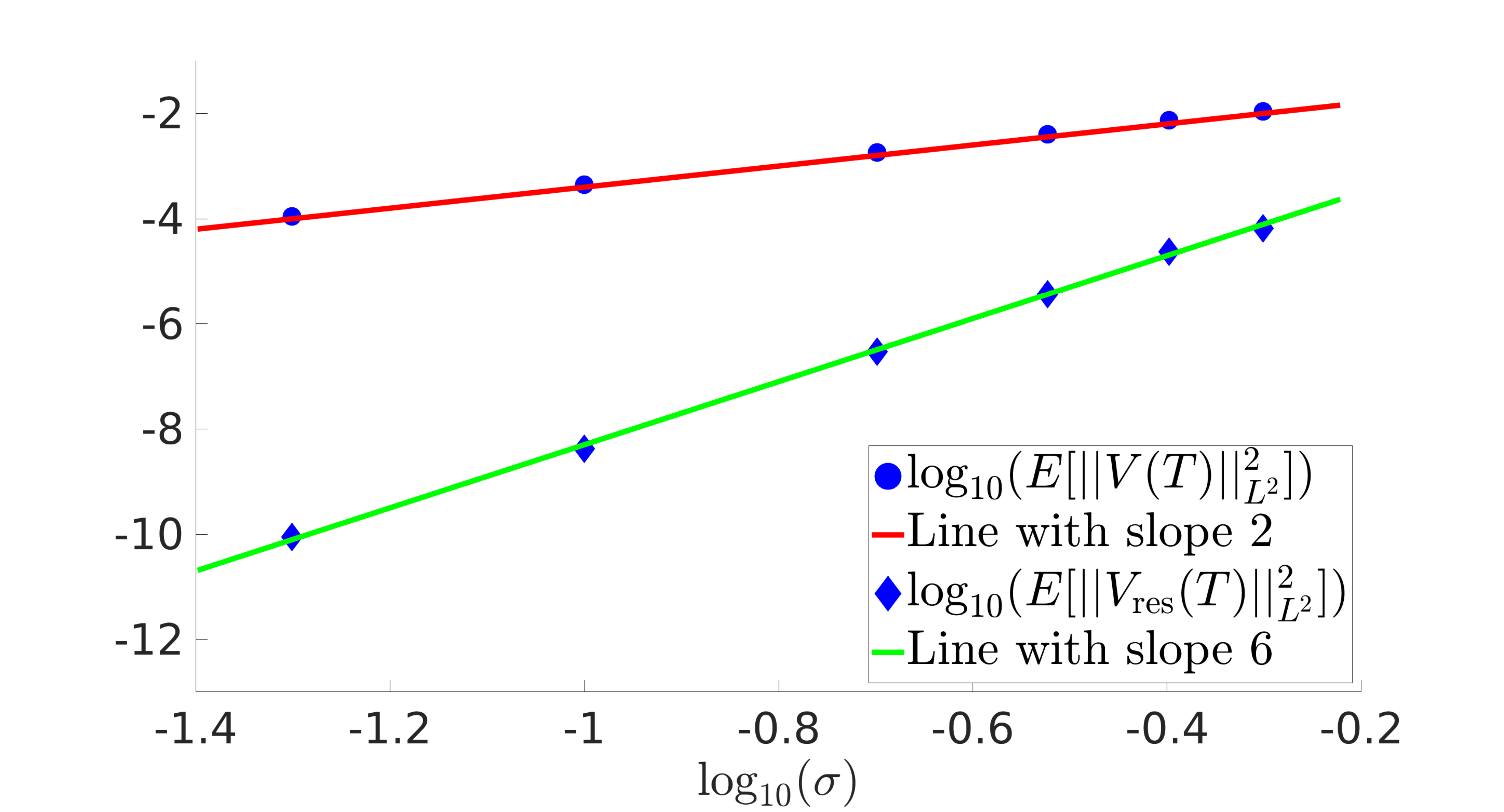}
    		\caption{}
    		\label{fig:resN:VRvsS}
\end{subfigure}
\begin{subfigure}{.49\textwidth}
\centering
 		\def\svgwidth{\columnwidth}
    		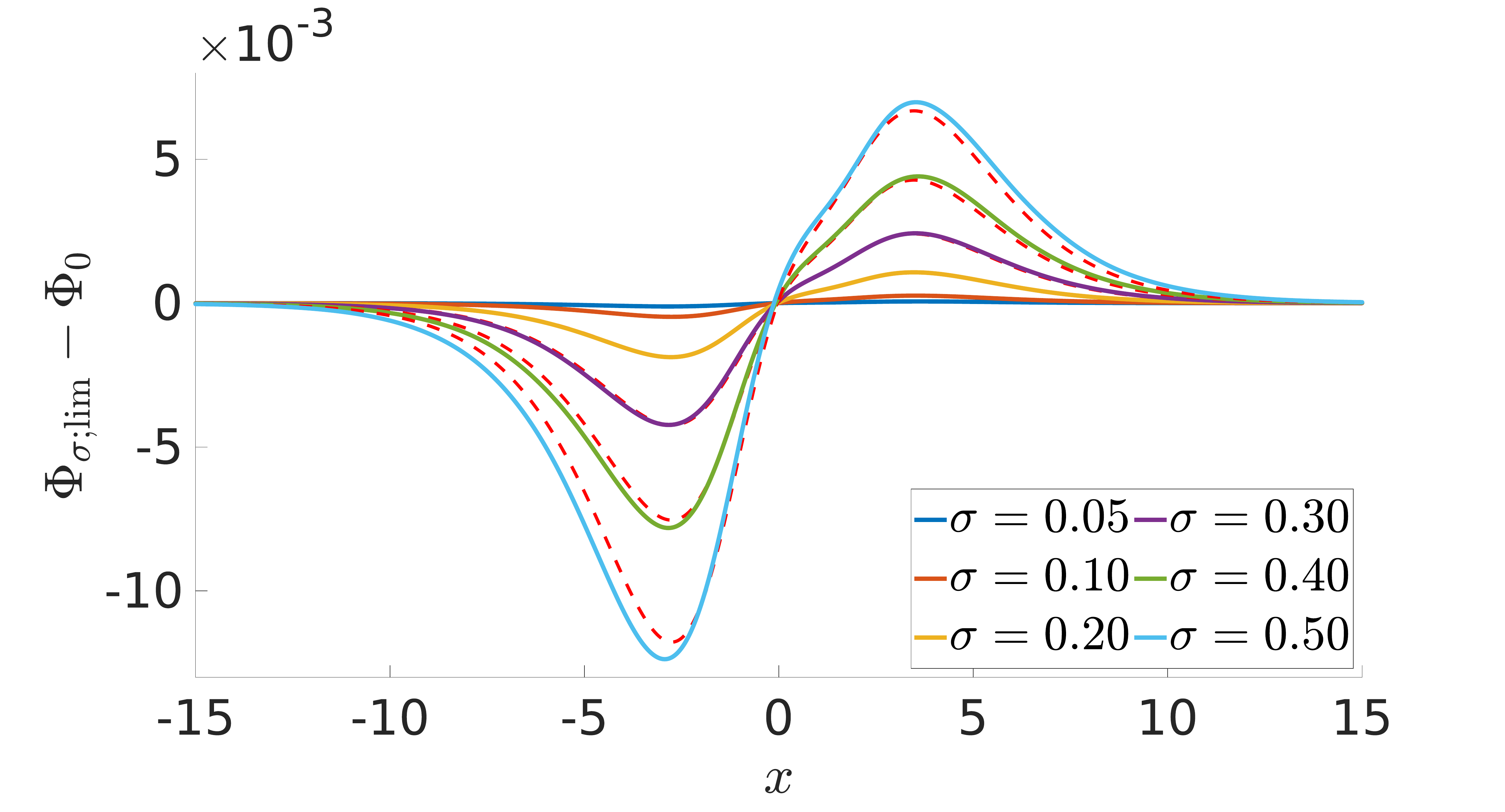
    		\caption{}
    		\label{fig:resN:2ndOrder}
\end{subfigure}
\caption{The datapoints in  (a) are computed from Fig. \ref{fig:VRvsT} by evaluating the expectations at  $T=1000$
and plotting them as function of $\s$. We observe that $E[\nrm{V(T)}_{L^2(\R)}^2]$ and $E[\nrm{V_\mathrm{res}(T)}_{L^2(\R)}^2]$ scale as $\O(\s^2)$ and $\O(\s^6)$ respectively, as predicted. Panel (b)
compares the observed (solid)
and predicted (dashed)
limiting deviations from $\Phi_0$
for multiple values of $\sigma$
in the Stratonovich interpretation,
see \sref{eq:res:N:obs:shape} and \sref{eq:res:N:pred:shape}.
}
\end{figure}

\begin{figure}
\begin{subfigure}{.49\textwidth}
    \centering
    \includegraphics[width=1\columnwidth]{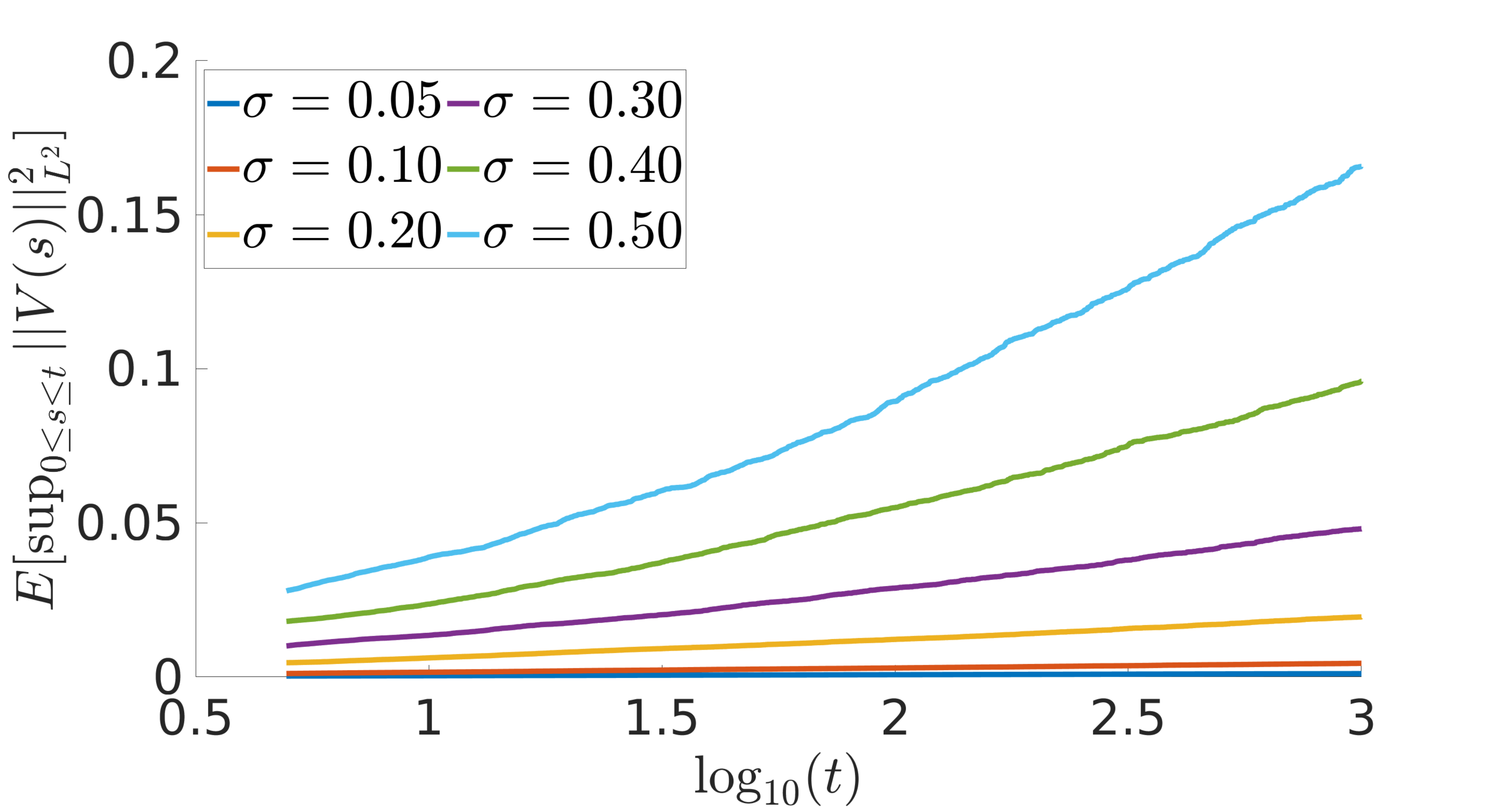}
\caption{}
    \label{fig:resN:supa}
\end{subfigure}
\begin{subfigure}{.49\textwidth}
      \centering
    \includegraphics[width=1\columnwidth]{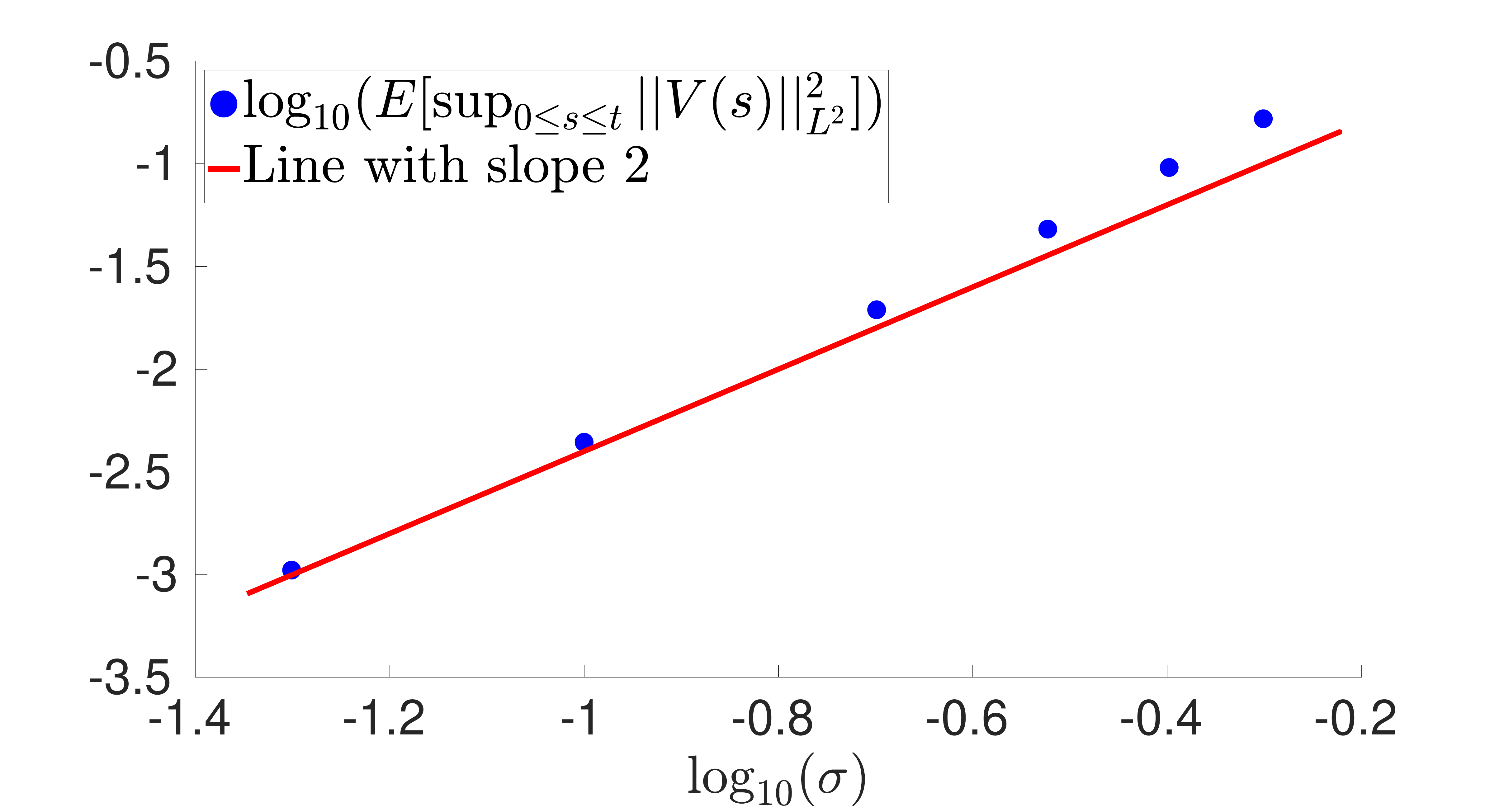}
    \caption{}
        \label{fig:resN:supb}
\end{subfigure}
    \caption{Panel (a) shows the numerical evaluation of $E[\sup_{0\leq s\leq t}\nrm{V(s)}^2_{L^2(\R)}]$ for different values of $\s$,
    where the average is computed over 500 iterations. 
    The trend lines indicate
    that this supremum admits logarithmic growth. Panel (b) plots the supremum at $t=1000$ against $\s$, illustrating the $\O(\s^2)$ behaviour. We used the It\^o interpretation with $\a=0.25$ and $\zeta=1$. }
        \label{fig:resN:sup}
\end{figure}


\section{Example II: The FitzHugh-Nagumo system}
\label{sec:resFHN}
In this section we repeat
the experiments from {\S}\ref{sec:resN}
for
the two-component
FitzHugh-Nagumo system
\begin{align}\label{eq:resFHN:fhn}
\begin{split}
dU &= \big[ U_{xx} +f_\mathrm{cub}(U) -W + \mu \s^2h^{(u)}(U,W) ] dt
+ \sigma  g^{(u)}( U,W)  d W^{Q_1}_t,
\\[0.2cm]
dW &= \big[ \varrho V_{xx} + \e(U-\gamma W)+\mu\s^2h^{(w)}(U,W)] dt+ \s g^{(w)}(U,W)dW^{Q_2}_t,
\end{split}
\end{align}
where $f_\mathrm{cub}$ is the same cubic polynomial as in
{\S}\ref{sec:resN} and $\varrho,\e,\g,\s>0$. We assume that the two processes $W^{Q_1}_t$ and $W^{Q_2}_t$ are independent, allowing us to write
\begin{align}\label{eq:resFHN:setup}
    g(U,W)=\begin{pmatrix}
    g^{(u)}(U,W)&0\\ 0&g^{(w)}(U,W)
    \end{pmatrix},
    \hspace{1cm}
    Q\begin{pmatrix}v_1\\v_2\end{pmatrix}=\begin{pmatrix}
    Q_1 v_1  &0\\0&Q_2v_2
    \end{pmatrix}
    =
    \begin{pmatrix}
    q_1 * v_1  &0\\0&q _2 *v_2
    \end{pmatrix}
\end{align}
for two convolution kernels $q_1$ and $q_2$.
In particular, we have $n=m=2$
and we assume that the combination
$\mathrm{diag}(q_1, q_2)$ satisfies (Hq).

Upon combining the general computations
in \cite[p. 123]{evans2012introduction}
with the abstract infinite-dimensional
framework developed in  \cite[\S 4.1]{Twardowska1996},
one can show that the It\^o-Stratonovich correction
term is given by 
\begin{align}
\big(h^{(u)}(U,W), h^{(w)}(U,W) \big)^T
= 
  \frac{1}{2}
  \Big(
  q_1(0)D_1g^{(u)}(U,W)g^{(u)}(U,W),
  q_2(0)D_2g^{(w)}(U,W)g^{(w)}(U,W)
  \Big)^T.
\end{align}
As usual, we can switch between the It\^o $(\mu =0)$
and Stratonovich ($\mu=1$) interpretations for the noise term.

All the expressions that we derive in this section
are valid for the general situation described in  \sref{eq:resFHN:setup}. However,
in order to generate our plots we used the specific choices
\begin{equation}\label{eq:resFHN:setupNum}
 g^{(u)}(U, W) = U,
 \qquad
 g^{(w)}(U,W) = 0,
    \qquad
    q_1(x) = q_2(x) =\frac{1}{2}e^{\frac{-\pi x^2}{4}},
\end{equation}
together with the parameter values
$\a=0.1$, $\varrho=0.01$, $\e=0.01$ and $\gamma=5$.
Although we were unable to find prior work
to which our results can be compared,
we do point out that computations for the somewhat related Barkley model
are discussed in \cite{Garcia2001}.

\begin{figure}
\begin{subfigure}{0.5\textwidth}
    \centering
\includegraphics[width=1\columnwidth]{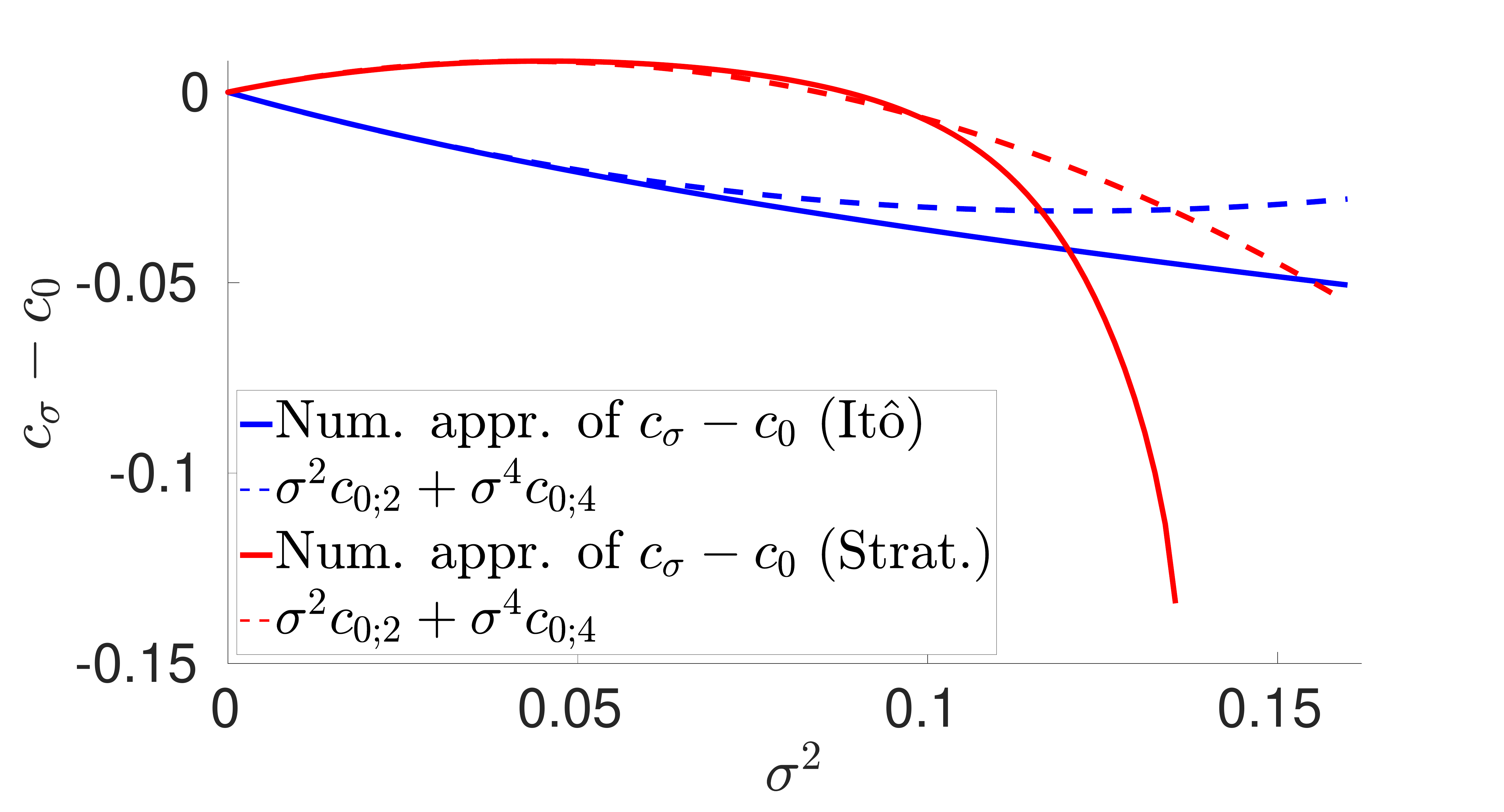}
        \caption{}
            \label{fig:resFHN:cs}
    \end{subfigure}%
    \begin{subfigure}{0.5\textwidth}
\includegraphics[width=1\columnwidth]{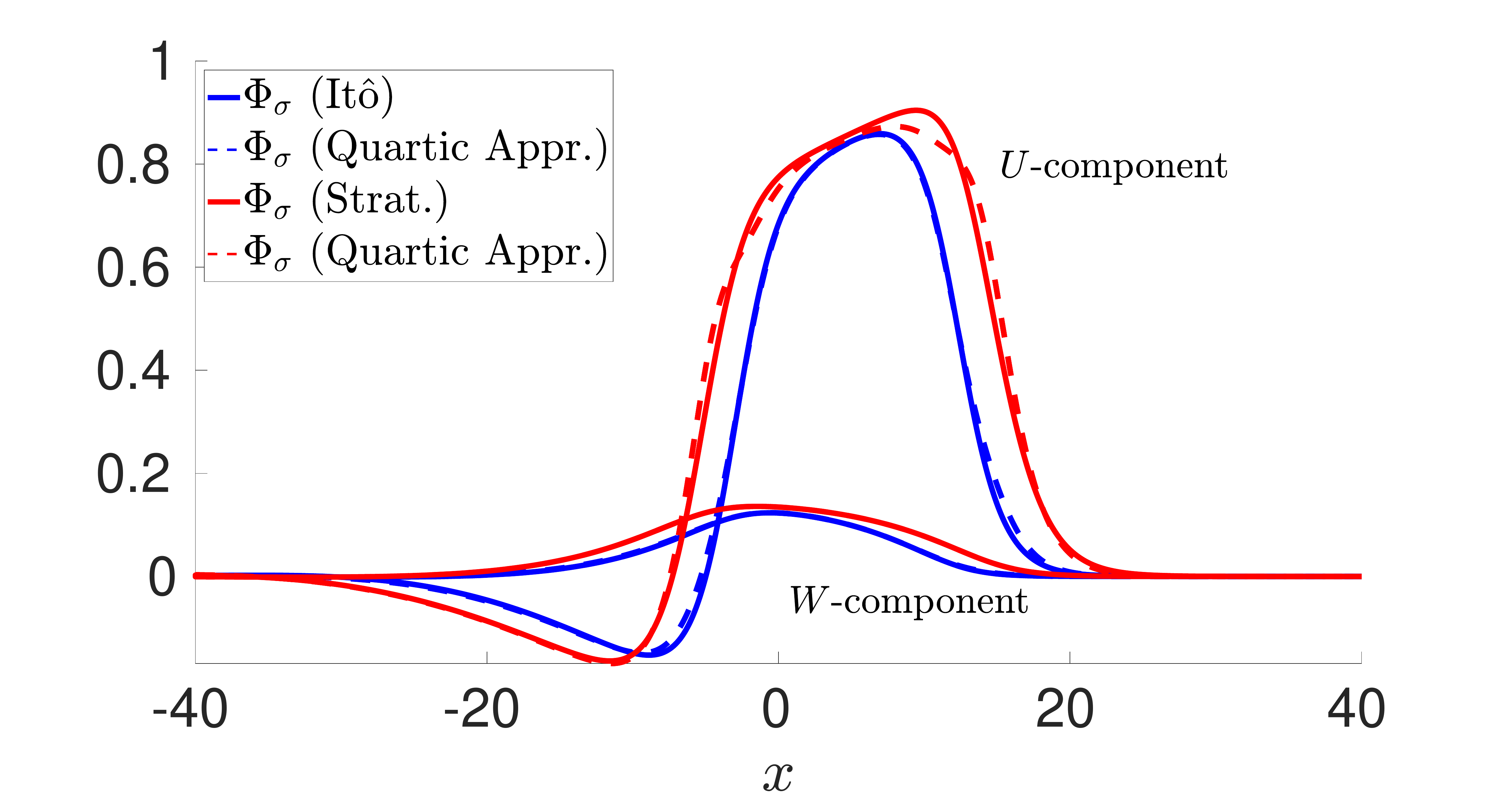}
        \caption{}
            \label{fig:resFHN:phi}
    \end{subfigure}
    \caption{
    These panels display the stochastic corrections $c_\s - c_0$ 
    for the wave speed (a) and the
    stochastic wave profiles $\Phi_\s$
    for $\sigma = 0.3$ (b), 
    together
    with their quartic approximations.
    }
    \label{fig:resFHN:Phiscs}
\end{figure}

\paragraph{Computing $(\Phi_\s,c_\s)$.}
Assume for the moment that
the deterministic travelling wave ODE
\begin{align}
\label{eq:resFHN:TWodeDet}
\begin{split}
   \p_{\xi\xi}\Phi^{(u)}_0 +c_0 \p_{\xi}\Phi^{(u)}_0+f_\mathrm{cub}(\Phi_0^{(u)})  -\Phi^{(w)}_0=&\,0,\\
  \varrho\p_{\xi\xi}\Phi^{(w)}_0
      +  c_0 \p_{\xi}\Phi^{(w)}_0+ \e(\Phi^{(u)}_0-\gamma \Phi^{(w)}_0)=&\,0
\end{split}
\end{align}
has a spectrally stable 
wave solution $\Phi_0=(\Phi^{(u)}_0,\Phi^{(w)}_0)$.
We then recall the associated linear operator
$\mathcal{L}_{\mathrm{tw}}: H^2(\R, \R^2) \to L^2(\R, \R^2)$
that acts as
\begin{equation}
\begin{array}{lcl}
\mathcal{L}_{\mathrm{tw}}
 & = & \left(
     \begin{array}{cc}
         \partial_{\xi \xi} + c_0 \partial_\xi + f'_{\mathrm{cub}}(\Phi_0^{(u)}) & - 1
         \\[0.2cm]
         \e  & \varrho \partial_{\xi \xi} + c_0 \partial_\xi   - \e \gamma
         \\[0.2cm]
     \end{array}
 \right) ,
\end{array}
\end{equation}
together with the formal adjoint operator 
that is given by
\begin{equation}
\begin{array}{lcl}
\mathcal{L}^*_{\mathrm{tw}}
 & = & \left(
     \begin{array}{cc}
         \partial_{\xi \xi} - c_0 \partial_\xi + f'_{\mathrm{cub}}(\Phi_0^{(u)}) & \e
         \\[0.2cm]
         -1  & \varrho \partial_{\xi \xi} - c_0 \partial_\xi   - \e \gamma
         \\[0.2cm]
     \end{array}
 \right) .
\end{array}
\end{equation}
The spectral stability
implies that $\mathcal{L}^*_{\mathrm{tw}}$
admits an eigenfunction
$\psi_{\mathrm{tw}}=(\psi^{(u)}_{\mathrm{tw}},\psi^{(w)}_{\mathrm{tw}})$
that can be normalized in such a way that
\begin{equation}
\label{eq:resFHN:norm:adj}
\langle \partial_\xi \Phi_0 , \psi_{\mathrm{tw}} \rangle_{L^2(\R,\R^2)} = 1.
\end{equation}
To summarize, we have
\begin{equation}
\mathcal{L}_{\mathrm{tw}} \partial_\xi ( \Phi^{(u)}_0, \Phi^{(w)}_0 )^T = 0,
\qquad \qquad
\mathcal{L}^*_{\mathrm{tw}}  ( \psi^{(u)}_{\mathrm{tw}}, \psi^{(w)}_{\mathrm{tw}})^T = 0.
\end{equation}

The existence of such spectrally stable
waves has been obtained in various
parameter regions \cite{alexander1990topological,cornwell2017opening,cornwell2017existence,jones1991construction},
but no explicit expressions
are available for $(\Phi_0,c_0)$.
However, they can readily be computed
numerically.

Upon writing $\Phi_\s=(\Phi^{(u)}_\s,\Phi^{(w)}_\s)$,
the stochastic wave equation $F_\s(\Phi_\s,c_\s)=0$ becomes
\begin{align}
\label{eq:resFHN:TWode}
\begin{split}
   \p_{\xi\xi}\Phi^{(u)}_\s +c_\s \p_{\xi}\Phi^{(u)}_\s+f_\mathrm{cub}(\Phi_\s^{(u)})  -\Phi^{(w)}_\s=&-\frac{\s^2}{2}\tilde{b}(\Phi_\s)\p_{\xi\xi}\Phi^{(u)}_\s-\mu\s^2 h^{(u)}(\Phi_\s)\\
   &\qquad +\s^2\frac{\p_\xi[g^{(u)}(\Phi_\s)q_1*(g^{(u)}(\Phi_\s)\psi^{(u)}_{\mathrm{tw}})]}
     {\ip{\p_\xi\Phi_\s,\psi_{\mathrm{tw}}}_{L^2(\R,\R^2)}} ,
     \\
  \varrho\p_{\xi\xi}\Phi^{(w)}_\s
      +  c_\s \p_{\xi}\Phi^{(w)}_\s+ \e(\Phi^{(u)}_\s-\gamma \Phi^{(w)}_\s)   =&-\frac{\s^2}{2}\tilde{b}(\Phi_\s)\p_{\xi\xi}\Phi^{(w)}_\s-\mu\s^2h^{(w)}(\Phi_\s)\\
   &\qquad +\s^2\frac{\p_\xi[g^{(w)}(\Phi_\s)q_2*(g^{(w)}(\Phi_\s)\psi^{(w)}_{\mathrm{tw}})]}
     {\ip{\p_\xi\Phi_\s,\psi_{\mathrm{tw}}}_{L^2(\R,\R^2)}} ,
\end{split}
      \end{align}
where $\tilde{b}$ is given by
\begin{align}
    \tilde{b}(\Phi)=-\frac{\ip{ q_1*(g^{(u)} (\Phi )\psi^{(u)}_{\mathrm{tw}}),g^{(u)} (\Phi )\psi^{(u)}_{\mathrm{tw}}}_{L^2(\Real)}+\ip{ q_2*( g^{(w)} (\Phi )\psi^{(w)}_{\mathrm{tw}}),g^{(w)} (\Phi )\psi^{(w)}_{\mathrm{tw}}}_{L^2(\Real)}}
    {\ip{\p_\xi\Phi_\s,\psi_{\mathrm{tw}}}_{L^2(\R,\R^2)^2}}.
\end{align}

Using \sref{eq:resFHN:TWode}
to evaluate \sref{eq:res:c02},
we find that the lowest order correction to the speed $c_\s$ reduces to
\begin{align}
\label{eq:resFHN:melnikov}
\begin{split}
c_{0;2} =& - \frac{1}{2} \tilde{b}(\Phi_0) \langle \p_{\xi \xi} \Phi_0 , \psi_{\mathrm{tw}} \rangle_{L^2(\Real , \Real^2) }
 -\ip{g^{(u)}(\Phi_0)q_1*(g^{(u)}(\Phi_0)\psi^{(u)}_{\mathrm{tw}}), \p_\xi\psi^{(u)}_{\mathrm{tw}}}_{L^2(\Real)}\\
  &\qquad -\ip{g^{(w)}(\Phi_0)q_2*(g^{(w)}(\Phi_0)\psi^{(w)}_{\mathrm{tw}}), \p_\xi\psi^{(w)}_{\mathrm{tw}}}_{L^2(\Real)}
  -\mu \ip{h(\Phi_0),\psit}_{L^2(\R,\R^2)}.
  \end{split}
\end{align}
In Fig. \ref{fig:resFHN:cs} we numerically computed $c_\s$ for the two interpretations. 
It turns out that the second order approximation above is only accurate
for a range of $\s$ that is much smaller 
than we saw for the Nagumo equation.
By also including the quartic term $c_{0;4}$ 
in our expansion we are able
to track $\Phi_\s$ reasonably well
up to $\s=0.3$. This is more
than sufficient for practical purposes,
as our simulations
of the full system \sref{eq:resFHN:fhn}
revealed that the pulse 
is unstable for  values of $\s$ larger than approximately $\s=0.15$.

Fig. \ref{fig:resFHN:phi} displays the shape of the instantaneous stochastic wave profile $\Phi_\s$ for the two different interpretations. It is striking that the wave becomes significantly wider for Stratonovich noise.

\begin{figure}
\centering
\begin{subfigure}{.33\textwidth}
  \centering
  \includegraphics[width=1\linewidth]{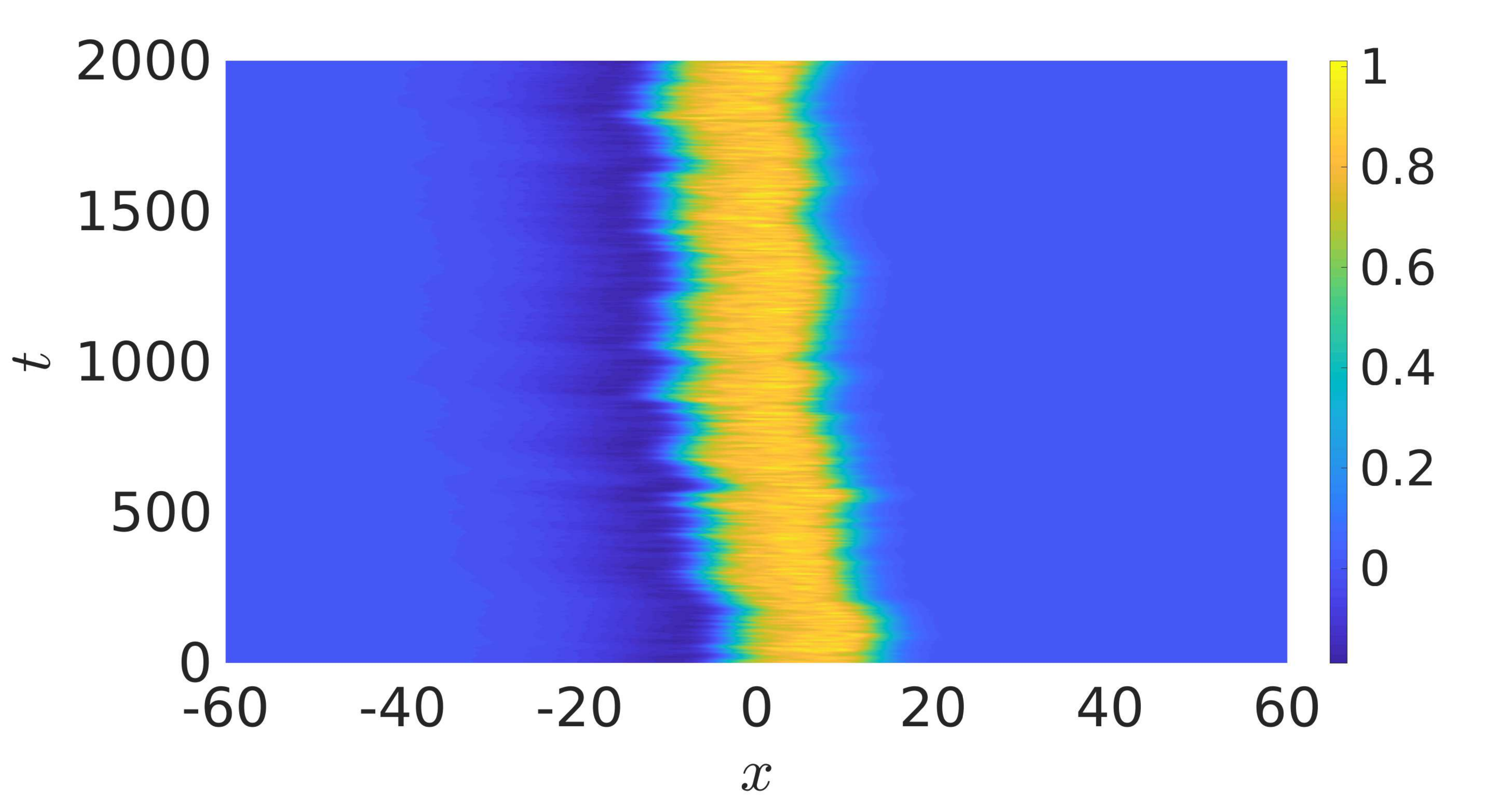}
  \caption{$U(x +c_0t,t)$}
    \label{fig:resFHN:UDifRefFHNa}
\end{subfigure}%
\begin{subfigure}{.33\textwidth}
  \centering
  \includegraphics[width=1\linewidth]{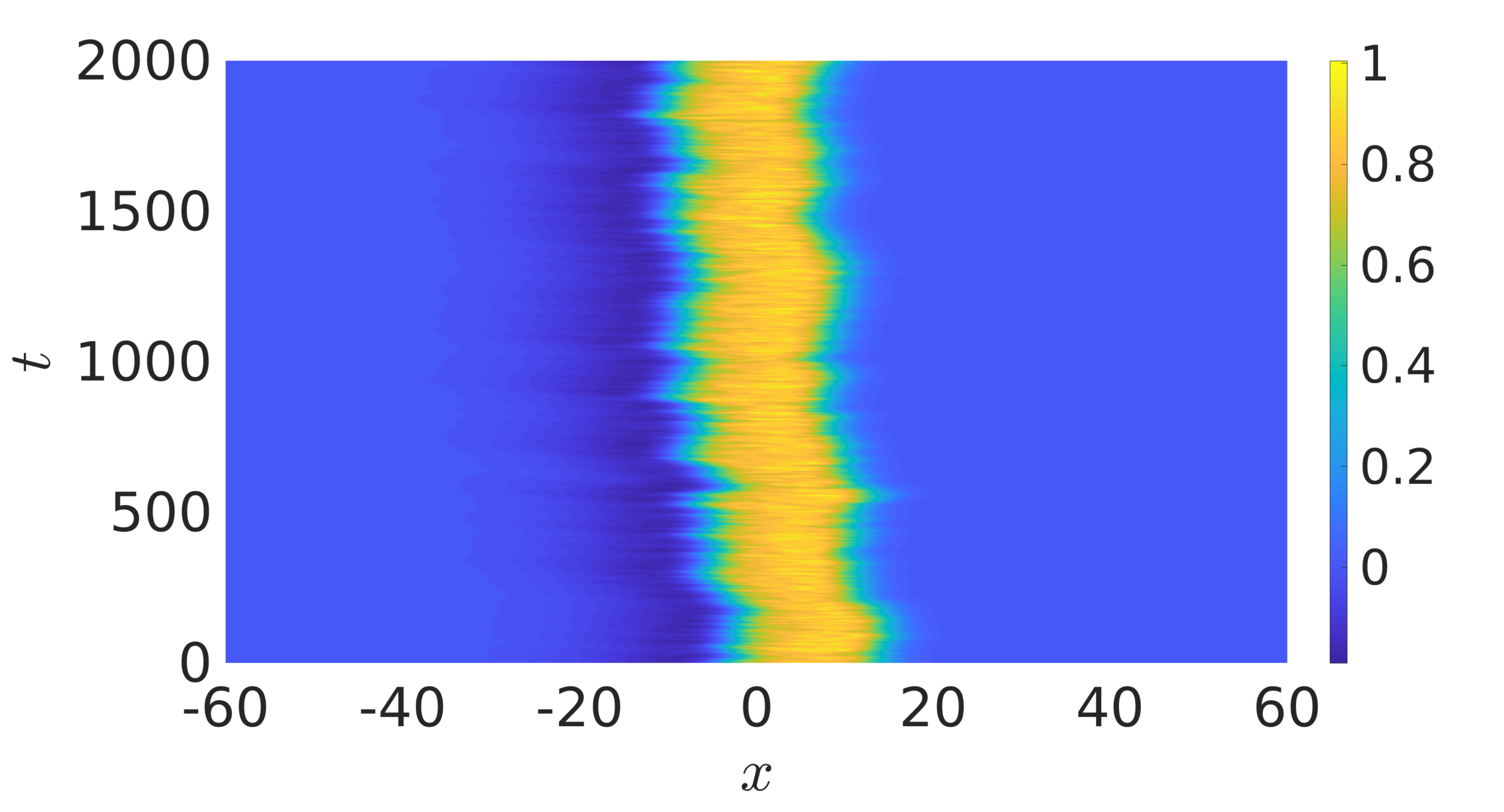}
  \caption{$U(x +c_\sigma t,t)$}
      \label{fig:resFHN:UDifRefFHNb}
\end{subfigure}%
\begin{subfigure}{.33\textwidth}
  \centering
  \includegraphics[width=1\linewidth]{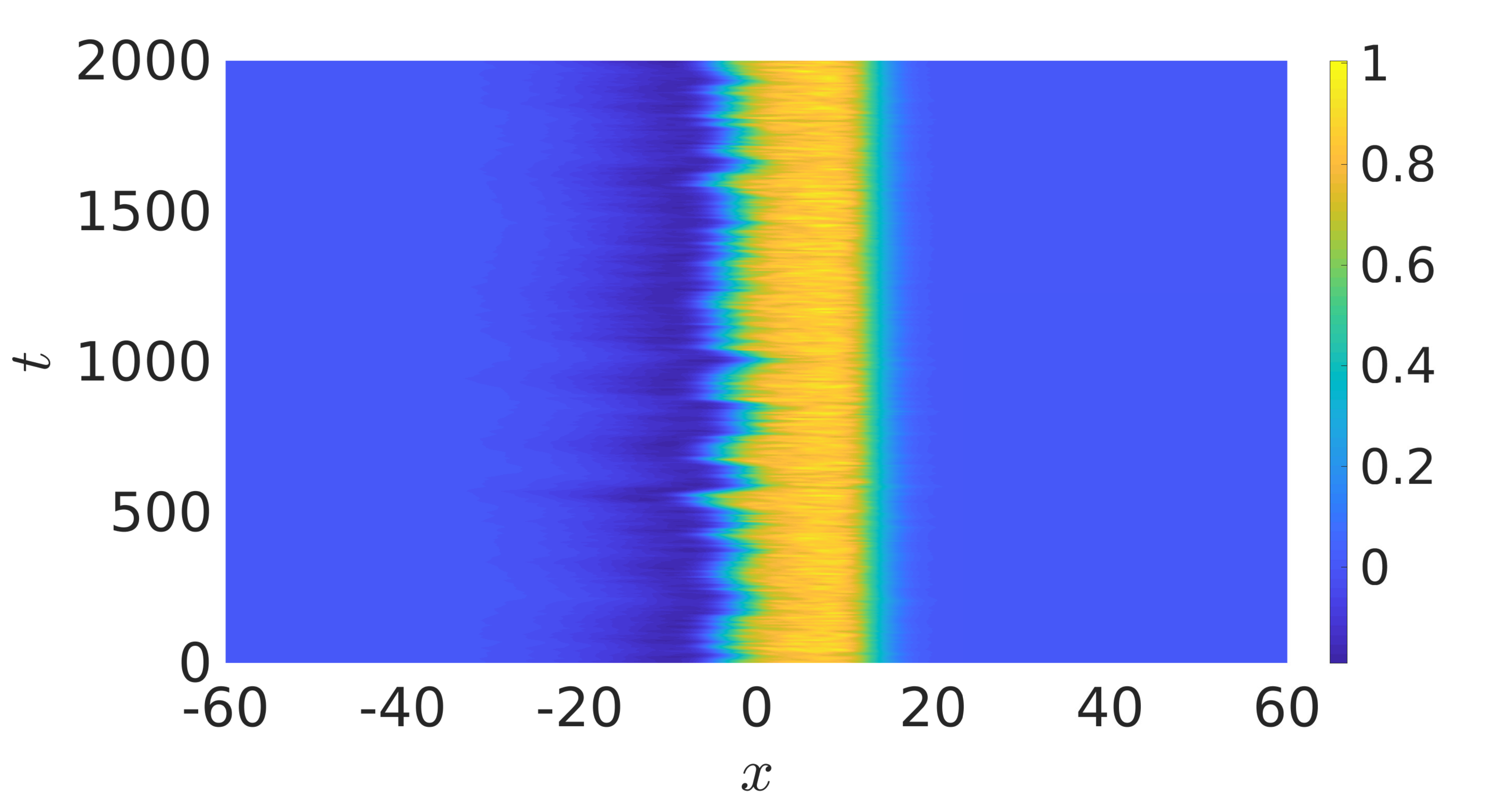}
  \caption{$U(x +\Gamma(t),t)$}
        \label{fig:resFHN:UDifRefFHNc}
\end{subfigure}
\caption{A single realization of the $U$-component of  \sref{eq:resFHN:fhn}
with $\s = 0.1$ in 3 different reference frames. The initial condition is given by $U(0) =\Phi_\s$.}
\label{fig:resFHN:UDifRefFHN}
\end{figure}

\paragraph{Limiting wave speed}
In Fig. \ref{fig:resFHN:UDifRefFHN} we illustrate the behaviour
of a representative sample solution to \sref{eq:resFHN:fhn}
by plotting it in three different moving frames.
Fig. \ref{fig:resFHN:UDifRefFHNa} clearly shows that the deterministic
speed $c_0$ overestimates the actual speed as the wave
moves to the left. The situation is slightly improved
in Fig. \ref{fig:resFHN:UDifRefFHNb}, where we use a frame
that travels with the stochastic speed $c_{\sigma}$.
However, the position of the wave now fluctuates
around a position that still moves slowly to the left
as a consequence of the orbital drift. This is remedied
in Fig. \ref{fig:resFHN:UDifRefFHNc}, where we use the full stochastic
phase $\Gamma(t)$. This again validates the idea of using $\Gamma(t)$ as position of the wave.

As in the previous example, 
the variance of $\G(t)$ is well-described
by the variance of the leading
order term $\Gamma_0^{(1)}$,
which is given by
\begin{align}\label{eq:resFHN:varG}
\begin{split}
    \mathrm{Var}\big(\G_0^{(1)}(t)\big)&=\ip{q_1*(g^{(u)}(\Phi_0)\psit^{(u)}),g^{(u)}(\Phi_0)\psit^{(u)}}_{L^2(\R)}t
    \\ & \qquad +\ip{q_2*(g^{(w)}(\Phi_0)\psit^{(w)}),g^{(w)}(\Phi_0)\psit^{(w)}}_{L^2(\R)}t.
    \end{split}
    \end{align}
In order to explain the drift
observed in Fig. \ref{fig:resFHN:UDifRefFHN},
we split the semigroup $S(t)$
into its two rows by writing
$S(t)= \big(S^{(u)}(t), S^{(w)}(t) \big)^T$.
The coefficient
\sref{eq:res:od}
can now be computed as
\begin{align}
\label{eq:resFHN:od}
\begin{split}
c^{\mathrm{od}}_{0;2}=\lim_{t\to\infty}t^{-1}E[\G_0^{(2)}(t)]&=-\int_0^\infty \sum_{k=0}^\infty\ip{K_0(s)[\sqrt Qe_k,\sqrt Qe_k],\psi_{\mathrm{tw}}}_{L^2(\R,\R^2)}ds\\
&= -\frac{1}{2}\int_0^\infty\sum_{k=0}^\infty
\ip{f''_\mathrm{cub}(\Phi_0^{(u)})
  \big(S^{(u)}(s) \mathcal{I}_k \big)^2, \psi^{(u)}_{\mathrm{tw}}}_{L^2(\R)} \, ds.
  \end{split}
\end{align}
Here  $\mathcal{I}_k$
is given by 
 \begin{align}\label{eq:resFHN:I}
    \mathcal{I}_k & = \begin{pmatrix}g^{(u)}(\Phi_0)p_1*e^{(u)}_k\\g^{(w)}(\Phi_0)p_2*e^{(w)}_k\end{pmatrix}
    -\alpha_k \p_\xi\begin{pmatrix}\Phi_0^{(u)}\\\Phi_0^{(w)}\end{pmatrix},
\end{align}
in which $(e_k)=(e_k^{(u)},e_k^{(w)})$ is a basis
of $L^2(\R,\R^2)$ and $\alpha_k$
is given by 
\begin{equation}
\alpha_k = 
\frac{\ip{p_1*e^{(u)}_k,g^{(u)}(\Phi_0)\psi^{(u)}_{\mathrm{tw}}}_{L^2(\R)}\hspace{-0.1cm}+\hspace{-0.1cm}\ip{p_2*e^{(w)}_k,g^{(w)}(\Phi_0)\psi^{(w)}_{\mathrm{tw}}}_{L^2(\R)}}{\ip{\p_\xi\Phi_0,\psi_{\mathrm{tw}}}_{L^2(\R,\R^2)}}.
\end{equation}
It is important to note here that the two components in the equation above mix even when $g^{(w)}=0$ due to the presence of the semigroup.

\begin{figure}
\centering
\begin{subfigure}{.5\textwidth}
  \centering
  \includegraphics[width=1\linewidth]{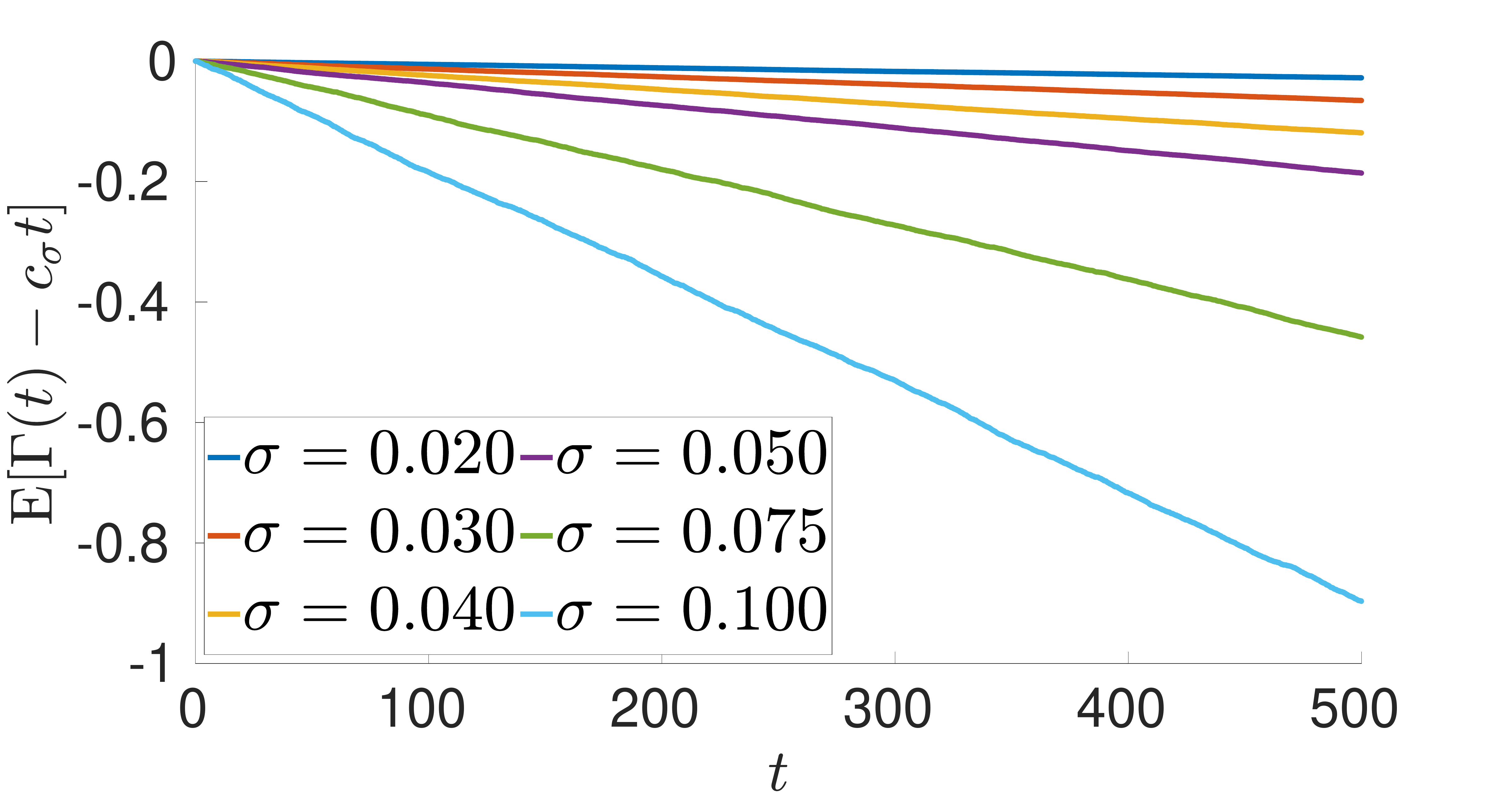}
  \caption{}
  \label{fig:DriftFHNa}
\end{subfigure}%
\begin{subfigure}{.5\textwidth}
  \centering
\includegraphics[width=1\columnwidth]{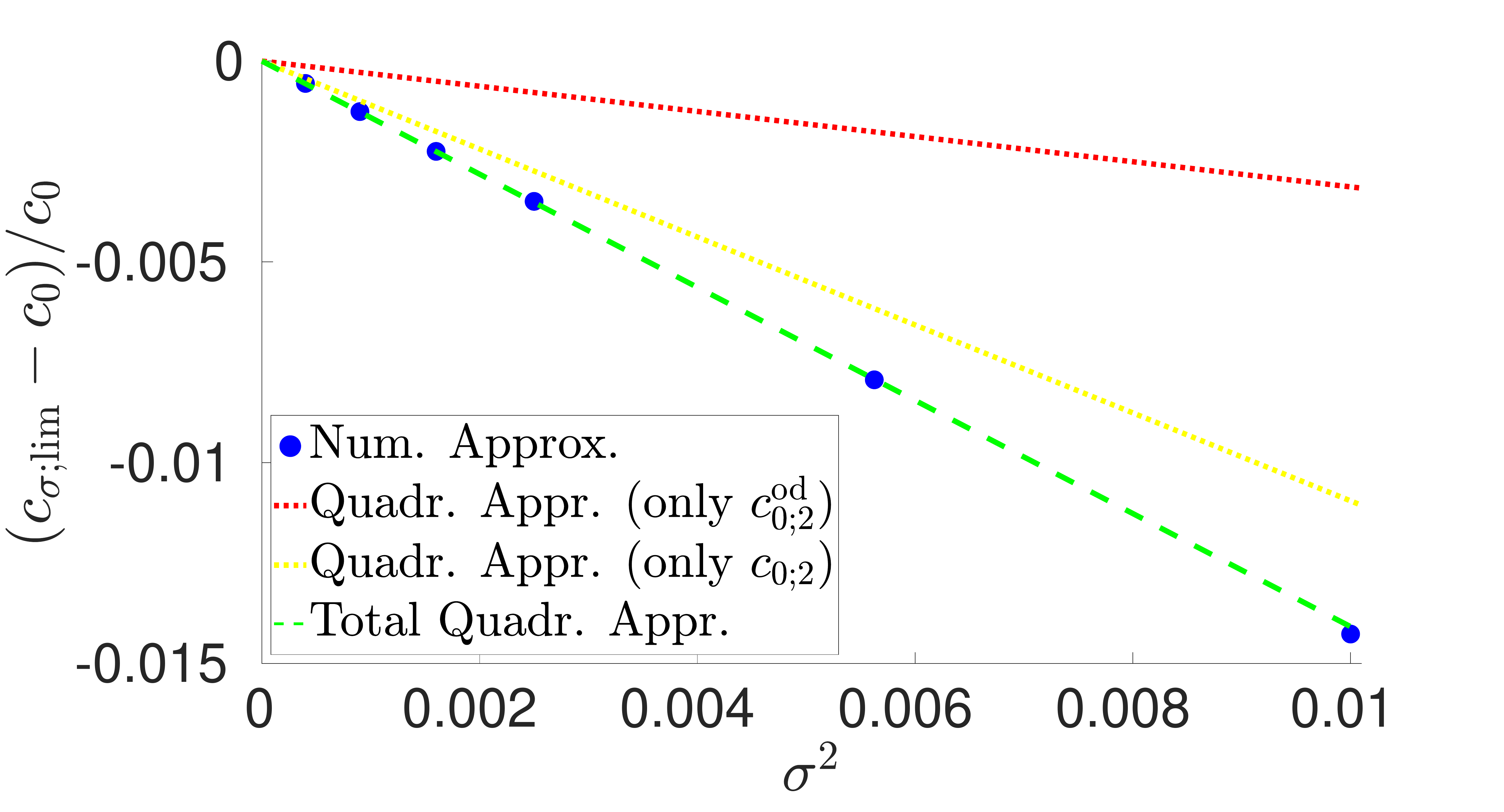}
  \caption{}
    \label{fig:DriftFHNb}
\end{subfigure}
\caption{In (a) we computed the average $E[\Gamma(t)-c_\s t]$
over 500 simulations of  \sref{eq:res:V-final}
with $\mu = 0$
for several values of $\s$, using
the procedure described in the main text.
Notice that a clear trend is visible. 
 In (b) we show the relative deviation of $c_{\s;\mathrm{lim}}$ from $c_0$.
%
Here the observed limiting speed is computed by evaluating the average \sref{eq:resFHN:od:obs}
for the data in (a), while the 
various quadratic predictions
are obtained from the relevant
terms in \sref{eq:res:fhn:c:pred}.
}
\label{fig:DriftFHN}
\end{figure}

In order to evaluate \sref{eq:resFHN:od} numerically, we reuse the basis  \sref{eq:resN:eigf} for $L^2([-L, L];\R)$ to construct a basis for $L^2([-L, L];\R)\times L^2([-L, L];\R)$. Because $Q$ is diagonal we can also recycle the approximate eigenvalues $\lambda_{k;\mathrm{apx}}$.
For the It\^o interpretation
and the parameter values
used in Fig. \ref{fig:DriftFHNa}, we obtain
\begin{align}
\label{eq:res:fhn:c:pred}
\begin{split}
    c^{\mathrm{pred}}_{\s;\mathrm{lim}}&=
    c_0+\s^2[c_{0;2}+c^\mathrm{od}_{0;2}]+\O(\s^3)\\
    &=0.4693-\s^2[0.5138+0.1470]+\O(\s^3).
\end{split}
\end{align}
Clearly, for the FitzHugh-Nagumo equation the influence of the orbital drift is significant.

To validate our predictions, we again
numerically compute
\begin{align}
\label{eq:resFHN:od:obs}
 c_{\s;\mathrm{lim}}^{\mathrm{obs}}
 =c_\s+\frac{2}{T} \int_{\frac{T}{2}}^T \frac{1}{t}E[\Gamma(t)-c_\s t-\s\Gamma_\s^{(1)}(t)] \, dt
\end{align}
and compare the outcome with \sref{eq:res:fhn:c:pred}.
Fig. \ref{fig:DriftFHNb} shows that the total observed speed is
indeed well approximated by the two leading order corrections, $\s^2c_{0;2}$ and $\sigma^2 c_{0;2}^{\mathrm{od}}$.

\paragraph{Size of $V(t)$}
We now turn our attention to the perturbation
\begin{equation}
    V(t)=(V^{(u)}(t),V^{(w)}(t))=
    \big(U(\cdot+\Gamma(t),t),W(\cdot+\Gamma(t),t)\big)
    -(\Phi_\s^{(u)}, \Phi_\s^{(w)} )
\end{equation} 
introduced in \sref{eq:res:V-final}.
As in {\S}\ref{sec:resN},
Figs. \ref{fig:resFHN:VvsTFHN} 
and \ref{fig:resFHN:RVvsSigmaFHN} 
show
that
$E\nrm{V(t)}_{L^2(\R,\R^2)}^2$ 
stabilizes exponentially fast to a fixed value
of size $O(\s^2)$.
These curves are 
nicely captured by the red dashed
lines, which describe the
integral
\begin{align}\label{eq:resFHN:nrmV1}
\begin{split}
    E[\nrm{V_0^{(1)}(t)}_{L^2(\R,\R^2)}^{2}]&=
    \int_0^t\sum_{k=0}^\infty\nrm{S(s)
      \mathcal{I}_k}^2_{L^2(\R,\R^2)}ds
    \end{split}
\end{align}
that measure the size
of the
first order approximation
\begin{equation}
    V_0^{(1)}(t)=\big(V_0^{(u,1)}(t),V_0^{(w,1)}(t)\big)^T.
\end{equation}

Fig. \ref{fig:resFHN:RvsTFHN} shows that $E\nrm{V_\mathrm{res}(t)}^2_{L^2(\R,\R^2)}$ also stabilizes over time, but Fig. \ref{fig:resFHN:RVvsSigmaFHN} 
indicates that the expected $\O(\s^6)$ scaling is not achieved (although
the behaviour is significantly better than
$\O(\s^4)$).
We expect that this can be improved
by utilizing more advanced numerical schemes,
but do not pursue this further here.

\paragraph{Limiting Wave Profile}
Turning our attention to the average
shape of $V(t)$, we
recall \sref{eq:resFHN:I}
and note that
\sref{eq:res:EV2} can be computed
as
\begin{align}\label{eq:resFHN:EV2fhn}
\begin{split}
E\,[ V_0^{(2)}(t)]
=&
\frac{1}{2}\int_0^tS(t-s)\int_0^s\sum_{k=0}^\infty\Big[\begin{pmatrix} f_\mathrm{cub}''(\Phi_0)
\big( S^{(u)}(s') \mathcal{I}_k \big)^2\\0\end{pmatrix}
\\
& \qquad \qquad
-\Phi'_0\ip{f_\mathrm{cub}''(\Phi_0)\big( S^{(u)}(s') \mathcal{I}_k \big)^2,\psi^{(u)}_{\mathrm{tw}}}_{L^2(\R)} 
\Big]\, ds'ds .
    \end{split}
\end{align}

In Fig. \ref{fig:resFHN:EVfhn}
we compare this second-order
expression with
the numerical average of
$E[V(t)]$ over 500 simulations
of \sref{eq:resFHN:fhn}.
To speed up the convergence of the average, we subtract both $\sigma V_\s^{(1)}(t)$ and the stochastic integral of $\sigma^2 V_\s^{(2)}(t)$
from $V(t)$. This does not change
the outcome as both terms have
zero expectation.

Notice that these two processes
are almost indistinguishable from each other.
To illustrate this, we provide
snapshots of both processes
at $t=50$ in Fig. \ref{fig:resFHN:evt50}
for various values  of $\s$.
Notice that the second-order
approximants follow
the intricate shape of $E[V(t)]$ 
very closely.

\begin{figure}
\centering
\begin{subfigure}{.5\textwidth}
  \centering
  \includegraphics[width=1\linewidth]{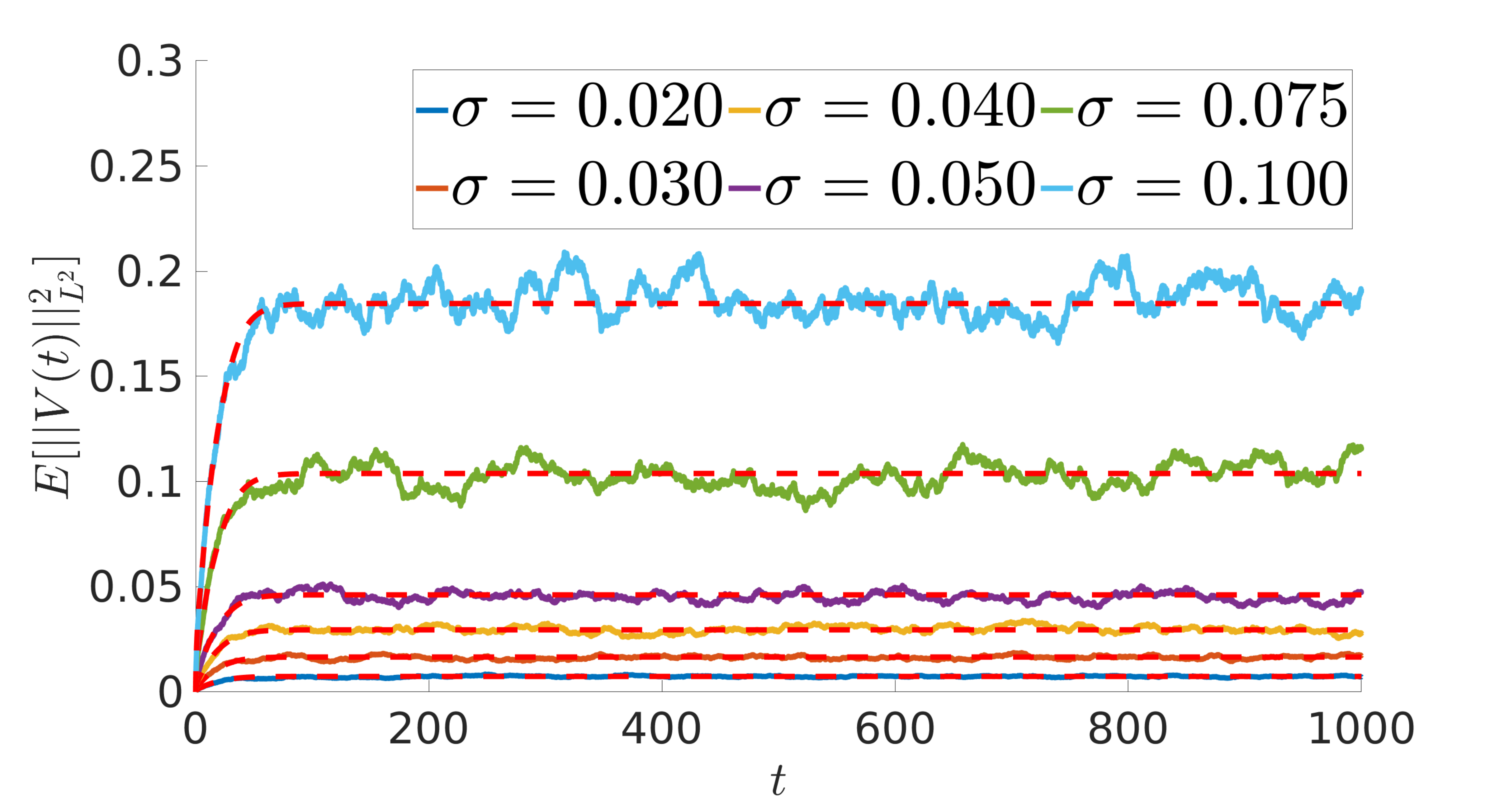}
  \caption{}
  \label{fig:resFHN:VvsTFHN}
\end{subfigure}%
\begin{subfigure}{.5\textwidth}
  \centering
  \includegraphics[width=1\linewidth]{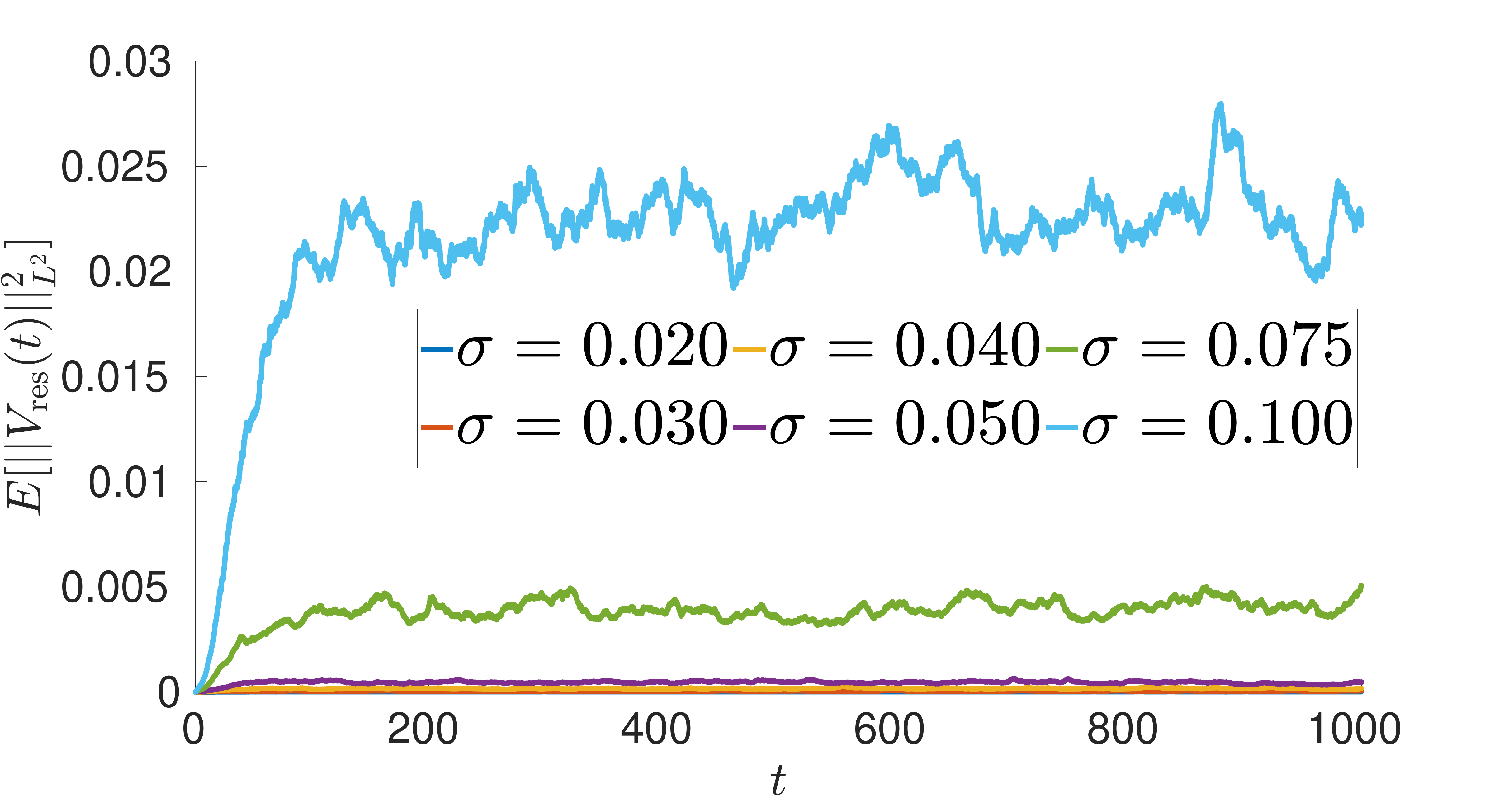}
  \caption{}
    \label{fig:resFHN:RvsTFHN}
\end{subfigure}
\caption{
In (a) we computed the average $E[\nrm{V(t)}_{L^2}^2]$
over 500 realization of  
\sref{eq:res:V-final} in the It\^o interpretation. The dashed line shows the numerical evaluation of the first order term
\sref{eq:resFHN:nrmV1}.
%
In (b) we computed the corresponding 
averages for the residual \sref{eq:resN:res} by evaluating and subtracting $\s V_\s^{(1)}(t)$ and $\sigma^2 V_\s^{(2)}(t)$ for every realization in (a). Again, both $V(t)$ and $V_\mathrm{res}(t)$ stabilize over time.
}
\label{fig:resFHN:VRvsTFHN}
\end{figure}

\begin{figure}
\centering
\begin{subfigure}{.5\textwidth}
  \centering
 		\def\svgwidth{\columnwidth}
    		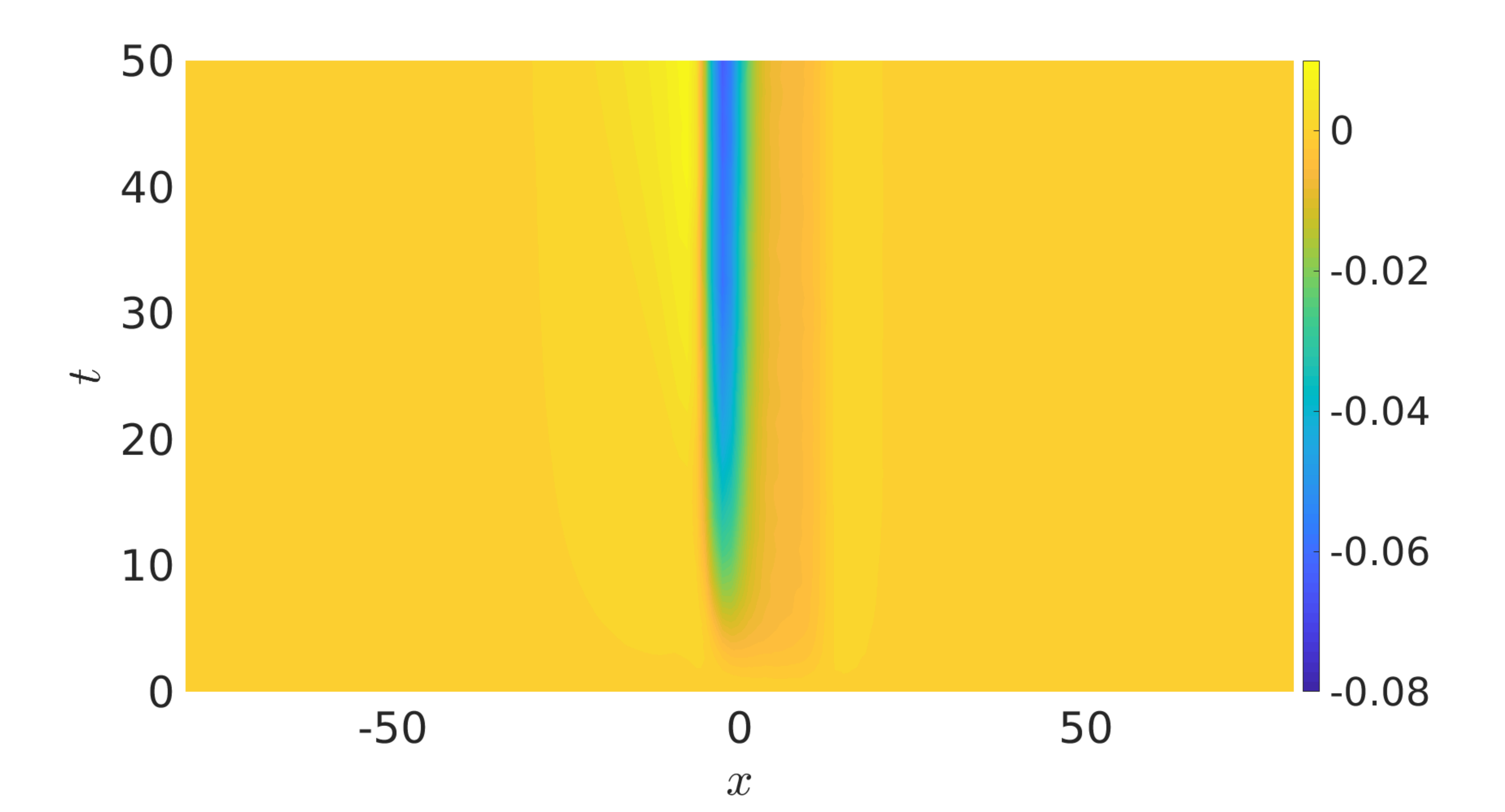
  \caption{$E[V^{(u)}(t)]$}
  \label{fig:resFHN:EV}
\end{subfigure}%
\begin{subfigure}{.5\textwidth}
  \centering
 		\def\svgwidth{\columnwidth}
    		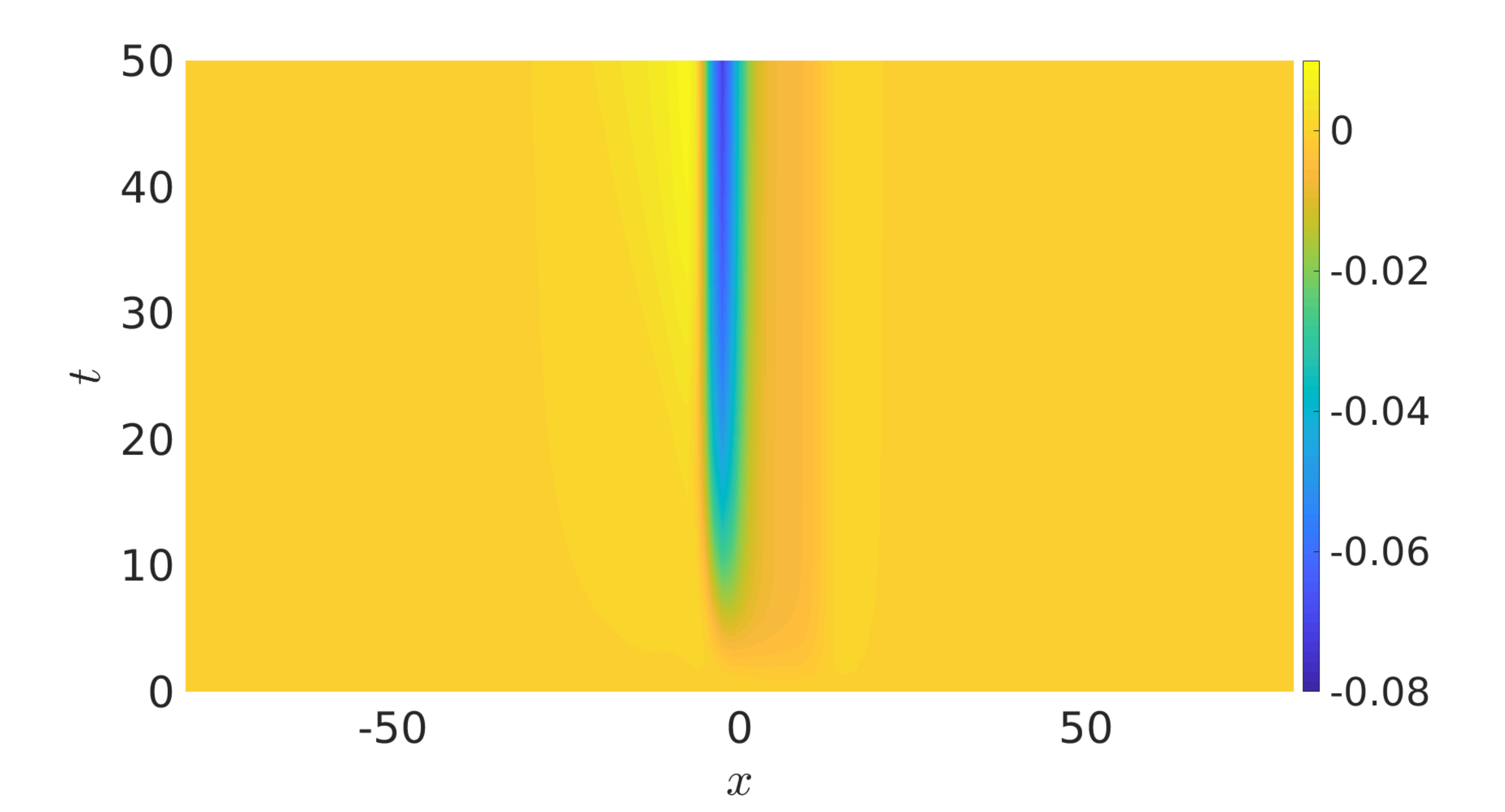
  \caption{$\s^2E[V_0^{(u,2)}(t)]$}
   \label{fig:resFHN:EV2}
\end{subfigure}
\caption{Panel (a) shows the average of the first component of $V(t)$, computed over 500 iterations of \sref{eq:resFHN:fhn}
with $\sigma = 0.1$ and $\mu=0$. Panel (b) shows the first component of the numerical evaluation of \sref{eq:resFHN:EV2fhn}. As before,
there is a good correspondence between the two figures. 
}
\label{fig:resFHN:EVfhn}
\end{figure}

\begin{figure}
\centering
\begin{subfigure}{.5\textwidth}
  \centering
\includegraphics[width=1\columnwidth]{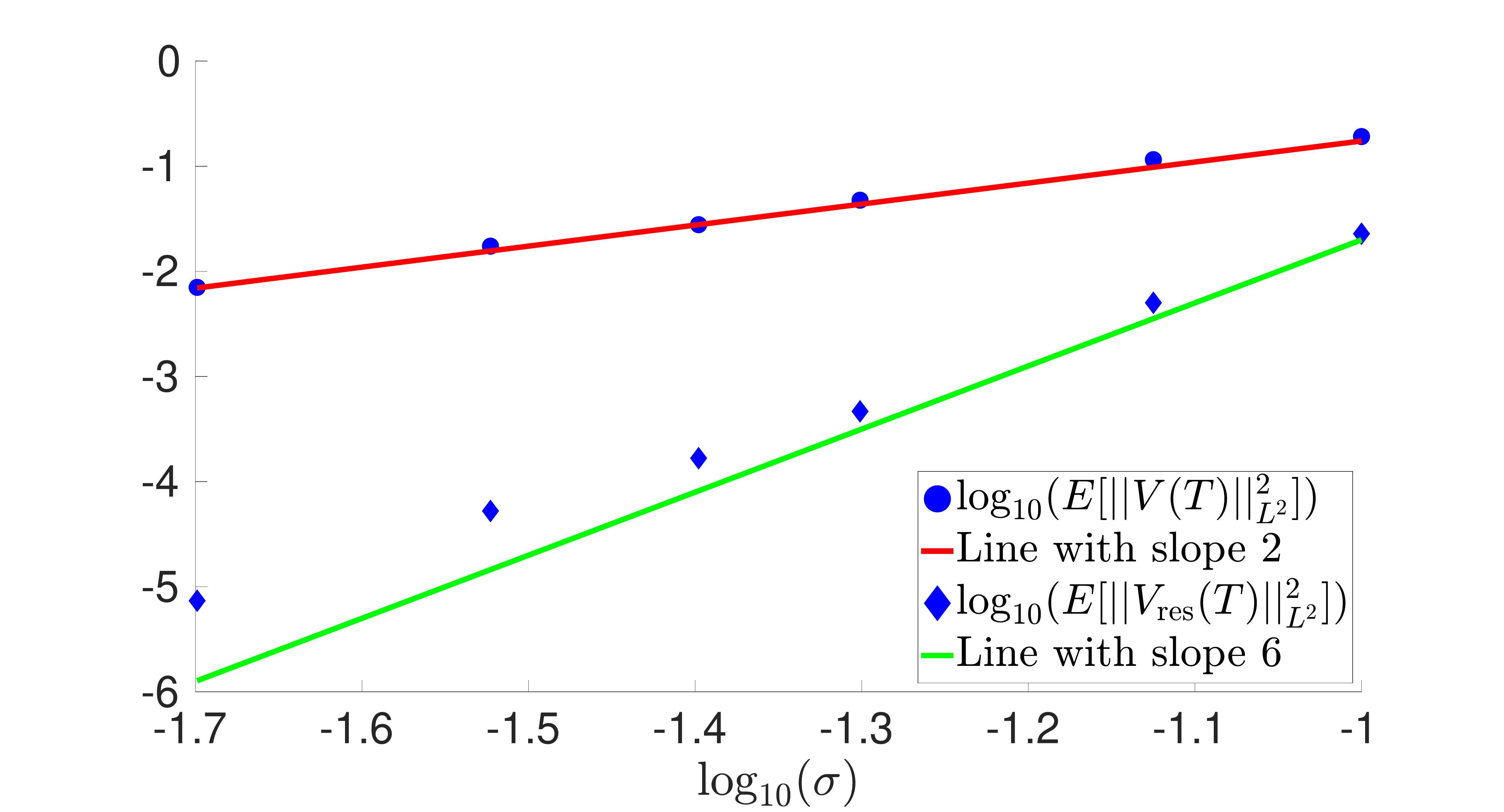}
  \caption{}
  \label{fig:resFHN:RVvsSigmaFHN}
\end{subfigure}%
\begin{subfigure}{.5\textwidth}
  \centering
  \includegraphics[width=1\linewidth]{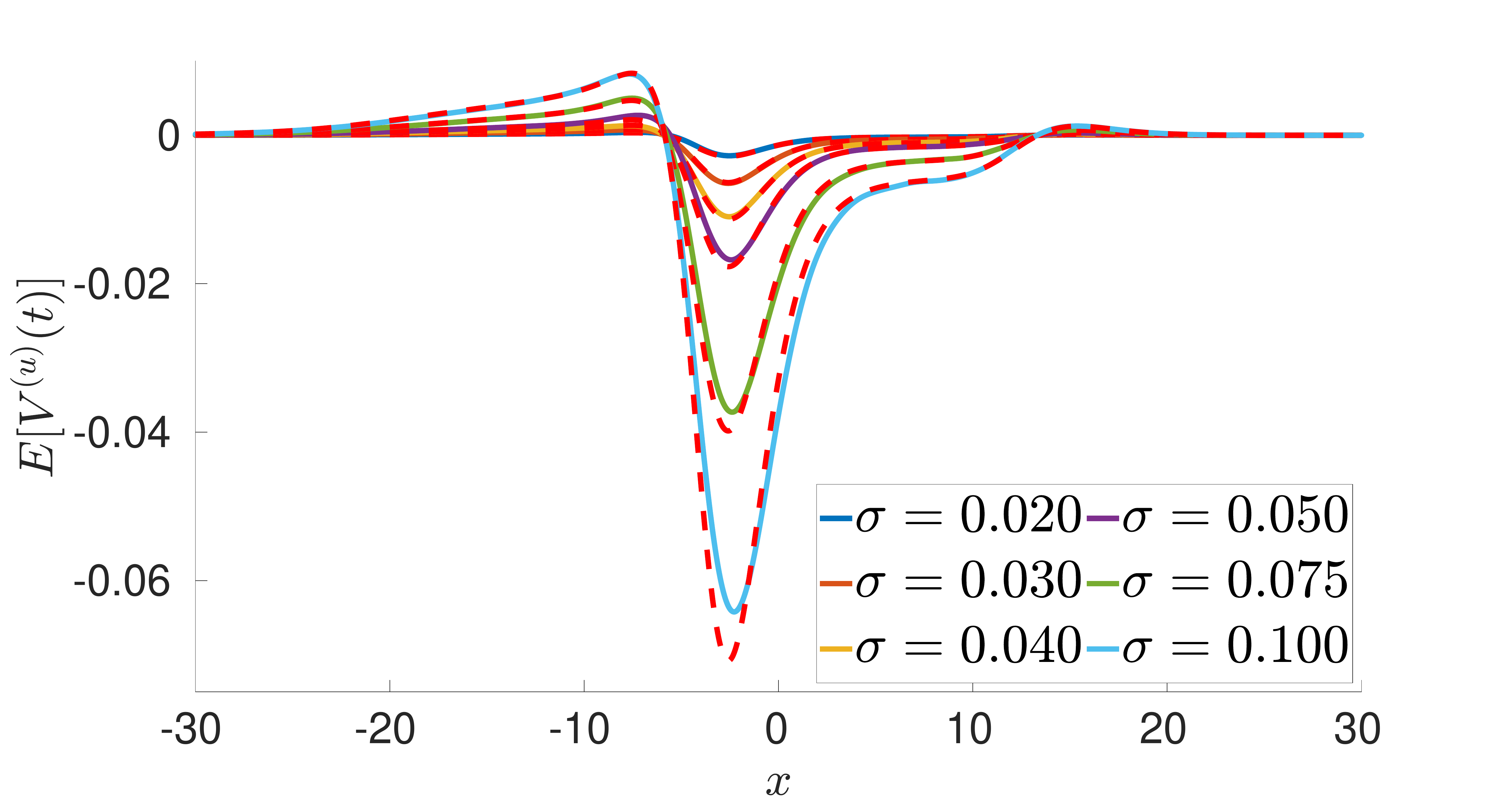}
  \caption{}
    \label{fig:resFHN:evt50}
\end{subfigure}
\caption{Panel (a) is computed from Fig. \ref{fig:resFHN:VRvsTFHN} by evaluating the expectations at the last time step $T=1000$ and plotting them as function of $\s$. We observe that $E[\nrm{V(T)}_{L^2(\R,\R^2)}^2]$ scales as $\O(\s^2)$ as predicted and $E[\nrm{V_\mathrm{res}(T)}_{L^2(\R,\R^2)}^2]$  scales significantly faster then $\O(\s^4)$, but not as the predicted $\O(\s^6)$. 
Panel (b) is computed from Fig. \ref{fig:resFHN:EVfhn}, by evaluating 
$E[V^{(u)}(50)]$. 
The dashed lines correspond to
the second order predictions $\s^2E[V_0^{(u,2)}(50)]$.
}
\label{fig:VRvsSFHN}
\end{figure}


\newpage
\section{The stochastic phase-shift}
\label{sec:sps}
In this section we derive the SPDE \sref{eq:res:V-final}
that we used to describe the behaviour
of the phase-shifted perturbation
\begin{equation}
\label{eq:sps:def:V}
V(t) = T_{-\Gamma(t)} [ X(t)+\Phi_\mathrm{ref} ]
  - \Phi_\s
\end{equation}
introduced in \sref{eq:res:defV}. 
Here $T_{\gamma}$ stands for the 
right-shift operator\footnote{
These operators will always carry a subscript
and should not be confused
with the time $T$ introduced in {\S}\ref{sec:res}.
} 
$T_{\gamma} U = U(\cdot - \gamma)$.
We recall from {\S}\ref{sec:res} 
that the process $X$ is a solution
to the SPDE
\begin{align}\begin{split}
dX &=
  \big[ \rho \partial_{xx}(X+\Phi_\mathrm{ref}) + f(X+\Phi_\mathrm{ref})+\s^2 h(X+\Phi_\mathrm{ref}) \big] dt
  + \sigma g(X+\Phi_\mathrm{ref}) d W^Q_t\\
\end{split}
\end{align}
posed on the Hilbert space $L^2(\R,\R^n)$.
In addition, the phase $\Gamma(t)$
was assumed to satisfy the SODE
\begin{align}
\label{eq:sps:eqn:for:gm}
    d \Gamma=  \big[ c_{\sigma} + \overline{a}_{\sigma}(U,\G) \big] \, dt
     + \sigma \overline{b}(U,\G) dW^Q_t ,
\end{align}
with nonlinearities $\overline{a}_\s$
and $\overline{b}$ that were only defined
locally.

In {\S}\ref{sec:sps:bck} we sketch
how the noise process $dW^Q_t$ can be rigorously
constructed. We subsequently
introduce several
cut-off functions in {\S}\ref{sec:sps:prlm}
that allow us to define $\overline{a}_\s$
and $\overline{b}$ in such a way that
\sref{eq:sps:eqn:for:gm} remains
well-posed globally.
This allows us to formulate
an appropriate It\^o lemma
in {\S}\ref{sec:sps:ito}, 
which we use in \S\ref{sec:compps}
to perform the computations that lead to \sref{eq:res:V-final}.

\subsection{Background}
\label{sec:sps:bck}
In this section we briefly recall some of the functional analysis needed to set up the rigorous framework to study
SPDEs. In order to ease the comparison with the earlier work in \cite{Hamster2017},
it turns out to be convenient to work in an abstract setting for the moment.
In particular, we consider noise that lives in an arbitrary separable Hilbert space $\mathcal{W}$
and pick a non-negative symmetric operator 
$Q \in \mathcal{L}( \mathcal{W} , \mathcal{W})$.
We then write\footnote{In the literature, the pair
$(\mathcal{W}_Q,\mathcal{W})$ is often denoted as $(U_0,U)$, but in our setting this might be confusing with the solution $U(t)$.
}
\begin{equation}
    \mathcal{W}_Q = Q^{1/2} ( \mathcal{W} ),
\end{equation}
which is again a separable Hilbert space with inner product
\begin{align}
    \ip{v,w}_{\mathcal{W}_Q}=\ip{Q^{-1/2}v,Q^{-1/2}v}_{\mathcal{W}}.
\end{align}

We now fix an orthonormal basis $(e_k)$ for $\mathcal{W}$, which means that
$(\sqrt{Q} e_k)$ is a basis for $\mathcal{W}_Q$. For any
Hilbert space $\mathcal{H}$, we recall that a linear map
$\Lambda: \mathcal{W}_Q \to \mathcal{H}$ is contained
in the set of Hilbert-Schmidt operators $HS(\mathcal{W}_Q,\mathcal{H})$
if it satisfies $\langle \Lambda, \Lambda \rangle_{HS(\mathcal{W}_Q,\mathcal{H})} < \infty$.
Here the inner product is given by
\begin{align}\begin{split}
    \langle \Lambda_1, \Lambda_2 \rangle_{HS(\mathcal{W}_Q,\mathcal{H})}:&=
    \sum_{k=0}^\infty \ip{\Lambda_1 \sqrt Q{e}_k,\Lambda_2 \sqrt Q{e}_k}_{\mathcal{H}} .
\end{split}
\end{align}
The construction in \cite[{\S}2.5]{Concise} allow us to define
a Hilbert space $\mathcal{W}_{\mathrm{ext}} \supset \mathcal{W}$ so that the inclusion
$\mathcal{W}_{Q} \subset \mathcal{W}_{\mathrm{ext}}$ is such a Hilbert-Schmidt operator.
This (non-unique) extension space is the key ingredient
that allows our noise process to be rigorously constructed.

Turning to this task,
we  introduce a complete probability space
   $\Big(\Omega, \mathcal{F},  \mathbb{P} \Big)$,
together with a normal filtration $( \mathcal{F}_t)_{t \ge 0}$
and a set of independent $(\mathcal{F}_{t})$-Brownian motions
$(\beta_k)$. Following \cite[Eq. (2)]{Karczewska2005},
we introduce the formal sum
\begin{equation}
\label{eq:sps:bck:def:w:q:t}
    W^Q_t = \sum_{k=0}^\infty \sqrt{Q} e_k  \beta_k(t) ,
\end{equation}
which converges in $L^2(\Omega,\mathcal{F},P; \mathcal{W}_{\mathrm{ext}})$
for every $t \ge 0$.  We will refer to this limiting process $W^Q_t$
as a $(\mathcal{F}_t , Q)$-cylindrical Wiener process.
The computations in \cite[Prop. 2]{Karczewska2005}
show that the formal sums
\begin{equation}
    \langle W^Q_t , w \rangle_\mathcal{W} = \sum_{k=0}^\infty \langle \sqrt{Q} e_k, w \rangle_\mathcal{W} \, \beta_k(t),
    \qquad \qquad w \in \mathcal{W}
\end{equation}
define  scalar Wiener processes that satisfy
\begin{equation}
   E \Big[ \langle W^Q_t , w_1 \rangle_\mathcal{W} \langle W^Q_s , w_2 \rangle_\mathcal{W} \Big]
    = (t \wedge s) \langle Q w_1, w_2 \rangle_{\mathcal{W}}.
\end{equation}

For any Hilbert space $\mathcal{H}$ and any $T>0$, we
follow the convention in \cite{Concise,revuz2013continuous}
and introduce the space
\begin{equation}
\label{eq:sps:def:n2}
\begin{array}{lcl}
\mathcal{N}^2\big( [0,T] ; (\mathcal{F}_t ) ;
   \mathcal{H} \big)
& = & \{ X \in L^2\big( [0 , T] \times \Omega ;
  dt \otimes \mathbb{P} ;  \mathcal{H} \big)
    :
\\[0.2cm]
& & \qquad \qquad X \hbox{ has a }
    (\mathcal{F}_t)\hbox{-progressively measurable version}
 \}.
\end{array}
\end{equation}
For any process $B \in \mathcal{N}^2\big( [0,T] ; (\mathcal{F}_t ) ;
   HS(\mathcal{W}_Q,\mathcal{H}) \big)$,
we now use \cite[Eq (7)]{Karczewska2005} to define
the stochastic integral
\begin{equation}
    \int_0^t B(s) \, d W^Q_s
    = \lim_{m \to \infty} \sum_{k=0}^m \int_0^t B(s)[\sqrt{Q} e_k] \, d
    \beta_{k}(s)
\end{equation}
for all $0 \le t \le T$. This limit can be taken
directly in $L^2(\Omega,\mathcal{F},P; \mathcal{H} )$ and hence
avoids the use of the external space. In this setting, the It\^o isometry can
be stated as
\begin{equation}
\label{eq:sps:ito:wrst:q:s}
    E \langle \int_0^t B_1(s) \, d W^Q_s , \int_0^t B_2(s) \, d W^Q_s \rangle_\mathcal{H}
    = E \, \int_0^t \langle B_1(s) , B_2(s) \rangle_{HS(\mathcal{W}_Q,\mathcal{H}) } \, ds .
\end{equation}

Returning to our main SPDE \sref{eq:mr:main:spde},
we assume for the moment
that $g(U)$ is a Hilbert-Schmidt operator from $\mathcal{W}_Q$ into $L^2(\R,\R^n)$
for every $U \in \mathcal{U}_{H^1}$.
The formal adjoint
\begin{equation}
    g^\mathrm{adj}(U): L^2(\R,\R^n) \to \mathcal{W}_Q
\end{equation}
is then defined in such a way that
\begin{equation}
    \langle g(U) [w] , \psi \rangle_{L^2(\R,\R^n)}
      = \langle w, g^\mathrm{adj}(U) \psi \rangle_{\mathcal{W}_Q}
      = \langle Q^{-1/2}w , Q^{-1/2} g^\mathrm{adj}(U)[ \psi ] \rangle_{\mathcal{W}} 
\end{equation}
holds for any $w \in \mathcal{W}$ and $\psi \in L^2(\R,\R^n)$.
This point of view allows us to unify the framework of this paper
with the setup used in 
\cite{Hamster2017,Hamster2018} where scalar noise is considered.

Indeed, for the setting described in {\S}\ref{sec:int}-\ref{sec:resFHN}
we can take $\mathcal{W} = L^2(\R,\R^m)$
and $\mathcal{W}_Q = L^2_Q$.
A simple computation
shows that
\begin{equation}
\label{eq:sps:def:g:adj:spec}
    g^\mathrm{adj}(U)[\psi] = Q g(U)^T \psi,
\end{equation}
in which the matrix transpose is taken in a pointwise fashion.
However, for $\mathcal{W} = \Real^m$ we must take
\begin{equation}
    g^\mathrm{adj}(U)[\psi] = Q \int_\R g\big(U(x)\big)^T \psi(x) \,dx ,
\end{equation}
which for $m=1$ reduces further to
\begin{equation}
\label{eq:sps:previous:1d}
    g^\mathrm{adj}(U)[\psi] = Q \langle g(U),  \psi \rangle_{L^2(\R)}.
\end{equation}
We shall see that \sref{eq:sps:previous:1d} can be used to recover
the results in \cite{Hamster2017,Hamster2018}
from
the expressions that we  derive in this section.

\subsection{Construction of $\overline{a}_\s$, $\overline{b}$, $\Phi_\s$ and $c_\s$}
\label{sec:sps:prlm}
In order to ensure that the SDE for the phase $\G(t)$ is well-defined
and admits global solutions, we need
to define the functions $\overline{a}_{\s}$ and
$\overline{b}$ appearing in \sref{eq:res:Gamma} in such a way
that $\overline{b}$ is globally bounded, while the singularities
in \sref{eq:derphase:b} and \sref{eq:derphase:a} are avoided.

To achieve this, we pick
a $C^\infty$-smooth non-decreasing cut-off function
\begin{equation}
\chi_{\mathrm{low}}: \Real \to [\frac{1}{4}, \infty),
\end{equation}
that satisfies the identities
\begin{equation}
\chi_{\mathrm{low}}(\vartheta) = \frac{1}{4} \hbox{ for } \vartheta \le \frac{1}{4},
\qquad
\chi_{\mathrm{low}}(\vartheta) = \vartheta \hbox{ for } \vartheta \ge \frac{1}{2}.
\end{equation}
In addition, we choose
a $C^\infty$-smooth non-increasing cut-off function
\begin{align}
    \chi_\mathrm{high}:\R^+\to [0,1],
\end{align}
for which we have
\begin{equation}
\chi_{\mathrm{high}}(\vartheta) = 1 \hbox{ for } \vartheta \le K_{\mathrm{up}},
\qquad
\chi_{\mathrm{high}}(\vartheta) = 0  \hbox{ for } \vartheta \ge K_{\mathrm{up}} + 1
\end{equation}
for some sufficiently large $K_{\mathrm{up}}>>1$.

For convenience, we now introduce the notation
\begin{equation}
\label{eq:sps:def:cutoffs}
    \overline{\chi}_l(U, \Gamma) =
    \big[\chi_{\mathrm{low} } \big(\bip{\p_\xi U,T_\G\psit}_{{L^2(\R,\R^n)}}\big)  \big]^{-1},
    \qquad
    \overline{\chi}_{h}(U, \Gamma) =
      \chi_{\mathrm{high}}( \norm{U - T_{\Gamma} \Phi_{\mathrm{ref}}}_{L^2(\R,\R^n)}).
\end{equation}
We remark that $\overline{\chi}_l$ and $\overline{\chi}_h$
are both uniformly bounded.
Whenever $\nrm{U - T_\Gamma \Phi_0}_{L^2(\R,\R^n)}$
is sufficiently small, we have
\begin{equation}
    \overline{\chi}_l(U, \Gamma) =\big[ \langle \partial_\xi U , T_\G \psit \rangle_{L^2(\R,\R^n)}\big]^{-1},
\qquad
\overline{\chi}_l(U, \Gamma) = 1.
\end{equation}

We now define
\begin{equation}
\label{eq:sps:cutoffb}
\overline{b}(U,\G)[v]=
- \overline{\chi}_h(U, \G)^2 \overline{\chi}_l(U,\G)
    \bip{ g(U)v, T_\G \psit }_{L^2(\R,\R^n)},
\end{equation}
noting that the square on the high cut-off
is simply for administrative reasons that 
will become clear in the sequel.
A short computation shows that
\begin{align}
\label{eq:sps:hs:norm:ovl:b}
\begin{split}
\nrm{\overline{b}(U,\G)}^2_{HS(L^2_Q, \R)}
 &=
\overline{\chi}_h(U, \G)^4 \overline{\chi}_l(U,\G)^2
\sum_{k=0}^\infty\ip{g(U) \sqrt{Q} e_k,T_{\Gamma} \psit}_{L^2(\R,\R^n)}^2
\\
&=\overline{\chi}_h(U, \G)^4 \overline{\chi}_l(U,\G)^2
\sum_{k=0}^\infty\ip{\sqrt{Q} e_k, g^\mathrm{adj}(U) T_{\Gamma} \psit}_{L^2_Q}^2
\\
 & =
\overline{\chi}_h(U, \G)^4 \overline{\chi}_l(U,\G)^2
\ip{g^\mathrm{adj}(U) T_{\Gamma} \psit, g^\mathrm{adj}(U) T_{\Gamma} \psit}_{L^2_Q}
\\
& =
\overline{\chi}_h(U, \G)^4 \overline{\chi}_l(U,\G)^2
\ip{g(U) g^\mathrm{adj}(U) T_{\Gamma} \psit, T_{\Gamma} \psit}_{L^2(\R,\R^n)}.
\end{split}
\end{align}

At this point, it is convenient
to introduce the notation
%
\begin{equation}
\label{eq:sps:defs:T:ABC}
\begin{array}{lcl}
  \mathcal{K}_{\s;A}(U)
    & = & \rho U'' + f(U) + \sigma^2 h(U) ,
\\[0.2cm]
   \mathcal{K}_B(U, \Gamma)
     & = & \frac{1}{2} \nrm{\overline{b}(U,\G)}^2_{HS(L^2_Q, \R)}
     U'' ,
 \\[0.2cm]
 \mathcal{K}_C(U, \Gamma)
  & = & - \overline{\chi}^2_h(U,\Gamma) \overline{\chi}_l( U , \Gamma)
    g(U) g^\mathrm{adj}(U) T_{\Gamma} \psit ,
\end{array}
\end{equation}
together with 
\begin{equation}
\mathcal{K}_{\sigma}(U, \Gamma ,c )
 =
 c U' + \mathcal{K}_{\s;A}(U)
    + \sigma^2 \mathcal{K}_B(U,\G) + \sigma^2
      \Big[\mathcal{K}_C(U,\G) \Big]'.
\end{equation}
In order to relate this back to
{\S}\ref{sec:res}, we write
\begin{equation}
F_{\sigma}(U, c) =  \mathcal{K}_\s(U , 0 , c)
\end{equation}
and note that this expression
reduces to \sref{eq:mr:def:F:i}
whenever $\norm{U - \Phi_0}_{L^2(\R,\R^n)}$ is sufficiently small on account of 
\sref{eq:sps:hs:norm:ovl:b}
and \sref{eq:sps:def:g:adj:spec}.
We are now in a position to construct
the instantaneous stochastic waves 
$(\Phi_\s, c_\s)$ by looking for zeroes
of $F_{\sigma}$.

\begin{prop}
\label{prp:var:swv:ex}
Suppose that (Hq), (HEq), (HDt), (HSt) and (HTw)
are all satisfied and pick a sufficiently large
constant $K > 0$. Then there exists $\delta_{\sigma} > 0$
so that
for every $0 \le \sigma \le \delta_{\sigma}$,
there is a unique pair
\begin{equation}
(\Phi_{\sigma},c_{\sigma}) \in   \mathcal{U}_{H^2}  \times \Real
\end{equation}
that satisfies the system
\begin{equation}
\label{eq:mr:prop:swv:eq}
  \mathcal{K}_\s(\Phi_\s,0,c_\s) = 0
\end{equation}
and admits the bound
\begin{equation}
\label{eq:mr:prp:swv:bnd:phi:sigma}
\norm{\Phi_{\sigma} - \Phi_0}_{H^2(\R,\R^n)} + \abs{c_{\sigma} - c_{0} }
\le K \sigma^2 .
\end{equation}
\end{prop}
\begin{proof}
On account of the estimates in 
Appendix \ref{sec:app:est},
the bounds in \cite[\S 7]{Hamster2017}
can be transferred to the current context.
The result can hence be established by 
following the proof of \cite[Prop. 2.2]{Hamster2017}.
\end{proof}

Having defined $\overline{b},\Phi_\s$ and $c_\s$, $\overline{a}_\s$ can now be written as
\begin{equation}
    \overline{a}_{\sigma}(U,\G)=
    - \overline{\chi}_l( U,\G)
    \bip{   \mathcal{K}_{\s}( U, \G, c_\s), T_{\Gamma} \psit
     }_{{L^2(\R,\R^n)}}.
\end{equation}
The commutation relations
\begin{equation}
T_{\Gamma}  f(U) =  f (T_{\Gamma} U),
\qquad
T_{\Gamma}  g(U)[w] =  g (T_{\Gamma} U)[T_{\G}w],
\qquad
T_{\Gamma}  g^{\mathrm{adj}}(U)[\psi] =  
  g^{\mathrm{adj}}( T_{\Gamma} U)[T_{\G}\psi],
\end{equation}
the latter of which exploits the translation invariance of $Q$,
allow us to conclude the crucial identities
\begin{equation}
    \overline{a}_\s(U, \G) = \overline{a}_\s(T_{-\G} U , 0),
    \qquad
    \overline{b}(U,\Gamma)[w]
      = \overline{b}(T_{-\G}U , 0)[T_{-\Gamma} w].
\end{equation}
This motivates the definitions
\begin{align}
 a_\s(V) = \overline{a}_\s(\Phi_\s + V, 0 ),
 \qquad \qquad
    b_\s(V)= \overline{b}(\Phi_\s + V,0)
\end{align}
that were introduced in {\S}\ref{sec:mr:st}.
In order to see that these
expressions reduce to
\sref{eq:derphase:b} and \sref{eq:derphase:a}
when $\norm{V}_{L^2(\R,\R^n)}$ is small,
we note that $F_\s(\Phi_\s, c_\s)  = 0$
and $\ip{\mathcal{L}_{\mathrm{tw}}V , \psit}_{L^2(\R,\R^n)} = 0$.

\subsection{It\^o lemma}
\label{sec:sps:ito}

Our goal here is to apply an appropriate version
of the It\^o lemma to the combined
stochastic process $\mathbf{Z}(t)=\big(X(t),\G(t)\big)$,
which takes values in the Hilbert spaces
\begin{equation}
\mathcal{H}^{1}_{\mathbf{Z}} = H^1(\R,\R^n) \times \Real,
\qquad
    \mathcal{H}_{\mathbf{Z}} = L^2(\R,\R^n) \times \Real,
\qquad
  \mathcal{H}^{-1}_{\mathbf{Z}} = H^{-1}(\R,\R^n) \times \Real.
\end{equation}
Indeed, upon defining nonlinearities
\begin{equation}
    \mathbf{A}_\s: \mathcal{H}^{1}_{\mathbf{Z}} \to \mathcal{H}^{-1}_{\mathbf{Z}},
    \qquad
    \mathbf{B}: \mathcal{H}^{1}_{\mathbf{Z}} \to HS\big( L^2_Q, \mathcal{H}_{\mathbf{Z}} \big)
\end{equation}
that act as
\begin{equation}
    \mathbf{A}_\s(X,\G)=\Big(
    \mathcal{K}_{\sigma;A}(X + \Phi_{\mathrm{ref}} ),
    c_\s+\bar a_\s(X+\Phi_\mathrm{ref},\G)
    \Big),
\end{equation}
together with
\begin{equation}
\mathbf{B}(X,\G)=\Big(g(X+\Phi_\mathrm{ref}),\overline{b}(X+\Phi_\mathrm{ref},\G) \Big),
\end{equation}
the coupled system for $\mathbf{Z}$ can formally be written as
\begin{align}
  \label{eq:ph:eq:Z}
    d \mathbf{Z}=\mathbf{A}_\s(\mathbf{Z})\,dt+\s\mathbf{B}(\mathbf{Z})\,dW^Q_t.
\end{align}

Our first result here clarifies how solutions
to this system should be interpreted. 
We emphasize
that our phase $\Gamma$ is 
almost surely continuous,
unlike its counterpart in \cite{Inglis} which admits jumps. This is a
direct consequence of the fact that $\Gamma$ is defined
to be the solution of an SODE rather than the minimizer of a distance
functional. The cut-off functions introduced in 
{\S}\ref{sec:sps:prlm} ensure that the phase
$\Gamma$ remains well-defined even if the orthogonality
condition $\langle V, \psit \rangle_{L^2(\R,\R^n)}=0$ can no longer be maintained.

\begin{prop}
\label{prp:phs:main:ex}
Suppose that (Hq), (HEq), (HDt), (HSt) and (HTw) are all satisfied
and fix $T > 0$, $0 \le \sigma \le \delta_\s$ and $c_\s \in \Real$.
In addition, pick an initial condition $\mathbf{Z}_0 \in \mathcal{H}_{\mathbf{Z}}$.
Then there is a unique map $\mathbf{Z}: [0,T] \times \Omega \to \mathcal{H}_{\mathbf{Z}}$
that is of class\footnote{Recall definition \sref{eq:sps:def:n2} for $\mathcal{N}^2$.}
$\mathcal{N}^2 \big( [0 , T] ; (\mathcal{F}_t ) ;
        \mathcal{H}^1_{\mathbf{Z}} \big)$
and satisfies the following properties.
\begin{itemize}
\item[(i)]{
  For almost all $\omega \in \Omega$,
  the map $t \mapsto \mathbf{Z}(t, \omega)$ is
  of class  $C([0,T]; \mathcal{H}_{\mathbf{Z}})$.
}
\item[(ii)]{
  For all $t \in [0,T]$, the map
  $\omega \mapsto \mathbf{Z}(t, \omega) \in \mathcal{H}_{\mathbf{Z}}$
  is $(\mathcal{F}_t)$-measurable.
}
\item[(iii)]{

  We have the inclusion
    $\mathbf{B} (\mathbf{Z}) \in
      \mathcal{N}^2 \big( [0 , T] ; (\mathcal{F}_t ) ;
        HS(L^2_Q , \mathcal{H}_{\mathbf{Z}}   )
        \big).$
}
\item[(iv)]{
  For almost all $\omega \in \Omega$,
  the identity
  \begin{equation}
  \label{eq:phs:id:for:z:weak}
  \mathbf{Z}(t) = \mathbf{Z}_0
    + \int_0^t \mathbf{A}_\s\big(\mathbf{Z}(s) \big) \, ds
    + \sigma \int_0^t \mathbf{B}\big( \mathbf{Z}(s) \big) \, dW^Q_s
  \end{equation}
  holds for all $0 \le t \le T$.
}
\end{itemize}
\end{prop}
\begin{proof}
In light of the estimates obtained
in Appendix \ref{sec:app:est}, we can closely follow the
proof of \cite[Prop 2.1]{Hamster2017}.
Indeed, the existence of the $dt \otimes \mathbb{P}$
version of $X$ that
is $(\mathcal{F}_t)$-progressively
measurable as a map into $H^1(\R,\R^n)$
follows from \cite[Ex. 4.2.3]{Concise}.
The main result
from \cite{LiuRockner} with $\alpha = 2$ and $\beta = 4$
can be used to verify the remaining statements concerning $X$.

As in \cite{Hamster2017}, the techniques developed
in \cite[Chapter 3]{Concise}
can be used to treat the second component of \sref{eq:ph:eq:Z}
as an SDE for $\Gamma$ with random coefficients. The key ingredient
is \cite[Thm. 3.1.1]{Concise}, which however is stated
only for finite dimensional noise. We claim
here that the conclusions also extend to 
the current setting 
where a cylindrical $Q$-Wiener process drives the stochastic terms.
To see this, we note that the  It\^o formula used in line 3.1.14
of the proof and the Burkholder-Davis-Gundy inequality used
on page 56 both extend naturally to our infinite-dimensional
setting. Most importantly, the local martingale defined in 3.1.14 remains a local martingale.
The remaining details can now easily be filled in by the
interested reader.
\end{proof}

The main ingredient to compute the equation for $V$
is the It\^o lemma. There are many versions 
available in the literature, but we choose to apply the formulation in \cite{DaPratomild} to our framework.
Note here that $D\phi$ and $D^2\phi$ are Frechet derivatives.

\begin{lem}\label{lem:ito}
Consider the setting of Proposition \ref{prp:phs:main:ex}
and pick a functional $\phi\in C^2(\mathcal{H}^{-1}_{\mathbf{Z}} ,\R)$.
Then for almost all $\omega \in \Omega$, the identity
\begin{align}
\label{eq:sps:itolemma}
\begin{split}
    \phi\big(\mathbf{Z}(t)\big)=&
      \phi\big(\mathbf{Z}(0)\big)
      +\int_0^t D\phi\big(\mathbf{Z}(s)\big)
        [\mathbf{A}_\s\big(\mathbf{Z}(s)\big)] \, ds
        + \sigma\int_0^tD\phi\big(\mathbf{Z}(s)\big)[\mathbf{B}\big(\mathbf{Z}(s)\big)] \, dW^Q_s\\
    &+\frac{1}{2} \s^2 \sum_{k=0}^\infty\int_0^tD^2\phi\big(\mathbf{Z}(s)\big)
      [\mathbf{B}\big(\mathbf{Z}(s)\big)\sqrt Qe_k,
        \mathbf{B}\big(\mathbf{Z}(s)\big)\sqrt Qe_k] \, ds
    \end{split}
\end{align}
holds for all $t > 0$.
\end{lem}
\begin{proof}
Item (iii) of Proposition \ref{prp:phs:main:ex}
and the identity \sref{eq:phs:id:for:z:weak}
allow us to interpret $\mathbf{Z}(t)$ as a (standard)
It\^o process on $\mathcal{H}^{-1}_{\mathbf{Z}}$
in the sense of \cite[Def. 1]{DaPratomild}, with
$S_{s,t} = I$. In particular, we can apply
\cite[Thm. 1]{DaPratomild} to obtain the result.
\end{proof}

\subsection{SPDE for $V$}
\label{sec:compps}

The defining identity \sref{eq:mr:defn:R}
for $\mathcal{R}_\s$ can be formulated
as
\begin{equation}
\begin{array}{lcl}
    \mathcal{R}_\s(V) & = & F_{\s}(\Phi_\s + V, c_\s) + a_\s( V) [ \Phi_\s' + V' ]
\\[0.2cm]
    & = & \mathcal{K}_\s(\Phi_\s + V, 0, c_\s) +  a_\s( V) [ \Phi_\s' + V' ],
\end{array}
\end{equation}
which is now well-defined as an element
of $H^{-1}(\R,\R^n)$ for all $V \in H^1(\R,\R^n)$.
Recalling the definition
\begin{align}
    \mathcal{S}_\s(V)[w]=
     g(\Phi_\s+V)[w]
       +\p_\xi(\Phi_\s+V)b_\s(V)[w] ,
\end{align}
we now set out to establish
the following result.

\begin{prop}
\label{prp:sps:props:v}
Suppose that (Hq),  (HEq), (HDt), (HSt) and (HTw) all hold.
Then the map
\begin{equation}
V: [0, T] \times \Omega \to L^2(\R,\R^n)
\end{equation}
defined by \sref{eq:sps:def:V}
is of class
$\mathcal{N}^2
       \big( [0 , T] ; (\mathcal{F}_t ) ; H^1(\R,\R^n) \big)$
and 
satisfies the following properties.
\begin{itemize}
\item[(i)]{
  For almost all $\omega \in \Omega$,
  the map
  $t \mapsto V(t,\omega)$
  is of class $C\big([0,T]; L^2(\R,\R^n)\big)$.
}
\item[(ii)]{
  For all $t \in [0,T]$, the map
  $\omega \mapsto V(t, \omega)
    \in L^2(\R,\R^n)$
  is $(\mathcal{F}_t)$-measurable.
}
\item[(iii)]{
  We have the inclusion

    $ \mathcal{S}_{\sigma}(V)
     \in
       \mathcal{N}^2 \big( [0 , T] ; (\mathcal{F}_t ) ; HS\big(L^2_Q , L^2(\R,\R^n)\big) \big). $
}
\item[(iv)]{
  For almost all $\omega \in \Omega$, we have the inclusion
  \begin{equation}
    \mathcal{R}_{\sigma}\big( V(\cdot, \omega) \big)
      \in L^1\big([0,T]; H^{-1}(\R,\R^n)\big)
  \end{equation}
and the identity
  \begin{equation}
  \label{eq:sps:int:eg:for:v}
  \begin{array}{lcl}
  V(t) & = & V(0) + \int_0^t
     \mathcal{R}_{\sigma}\big( V(s) \big)\, ds
    + \sigma \int_0^t
        \mathcal{S}_{\s}\big(V(s)\big)
       d W^Q_s
  \end{array}
  \end{equation}
  holds for all $0 \le t \le T$. }
\end{itemize}
\end{prop}

Our main task here is to establish
\sref{eq:sps:int:eg:for:v}.
Taking derivatives of translation operators typically requires extra regularity of the underlying function,
which prevents us from applying an It\^o formula directly to \sref{eq:sps:def:V}.
In order to circumvent this technical issue,
we pick a test function $\zeta \in C_c^\infty(\Real,\Real^n)$
and consider 
the map
\begin{equation}
    \phi_{\zeta}: H^{-1}(\R,\R^n) \times \Real \to \Real
\end{equation}
that acts as
\begin{equation}
\begin{array}{lcl}
\phi_{\zeta}\big( X , \Gamma \big)
  & = &
   \langle X + \Phi_\mathrm{ref} - T_{\Gamma} \Phi_\s, T_{\Gamma} \zeta
       \rangle_{H^{-1}; H^1}.
\\[0.2cm]
\end{array}
\end{equation}
Here $\langle \cdot, \cdot \rangle_{H^{-1}; H^1}$
denotes the duality pairing
between $H^{-1}(\R,\R^n)$
and $H^1(\R,\R^n)$, which
coincides with the inner product
on $L^2(\R,\R^n)$
when both factors are from this space; see \cite[{\S}2]{Hamster2017}.
This map does have sufficient smoothness for our purposes here and allows us to write
\begin{equation}
\label{eq:sps:vt:with:zeta:ip}
    \langle V(t), \zeta \rangle
      = \phi_{\zeta} \big( X(t) , \G(t) \big).
\end{equation}

We now introduce the notation
\begin{equation}
\begin{array}{lcl}
   \overline{\mathcal{R}}_{\s;\zeta}(U , \G)
     & = &
      \bip{ \mathcal{K}_\s(U, \G , c_\s)
       + \overline{a}_{\s}(U, \G)  U'
   , T_{\G} \zeta  }_{H^{-1} ; H^1} ,
\end{array}
\end{equation}
together with
\begin{equation}
\begin{array}{lcl}
   \overline{\mathcal{S}}_{\s;\zeta}(U , \G) [w]
     & = &
      \langle g(U)[w] , T_\G \zeta \rangle_{L^2(\R,\R^n)}
       + \langle U' , T_\G \zeta \rangle_{L^2(\R,\R^n)} \overline{b} (U,\G) [w] .
\end{array}
\end{equation}
As usual, we have
\begin{equation}
    \overline{\mathcal{R}}_{\s;\zeta}(U , \Gamma)
    = \overline{\mathcal{R}}_{\s;\zeta}(T_{-\Gamma} U , 0),\
     \qquad
     \overline{\mathcal{S}}_{\s;\zeta}(U, \Gamma)[w]
     =  \overline{\mathcal{S}}_{\s;\zeta}(T_{-\G} U, 0)[ T_{-\G} w].
\end{equation}
In addition, we note that
\begin{equation}
\langle \mathcal{R}_{\sigma}\big(V\big), \zeta \rangle_{H^{-1};H^1}
= \overline{\mathcal{R}}_{\s;\zeta}( \Phi_\s + V, 0),
\qquad
\langle \mathcal{S}_{\sigma}\big(V \big)[w], \zeta \rangle_{L^2}
= \overline{\mathcal{S}}_{\s;\zeta}( \Phi_\s + V, 0)[w].
\end{equation}
These auxiliary functions can be used
to formulate the equation that arises
when applying Lemma \ref{lem:ito}
to the functional $\phi_{\zeta}$. 

\begin{lem}
\label{lem:sps:ito:comp}
Suppose that (Hq), (HEq), (HDt), (HSt) and (HTw) all hold.
Then for almost all $\omega \in \Omega$,
the identity
\begin{align}
\label{eq:sps:ito:i}
\begin{split}
\phi_{\zeta}\big(X(t), \Gamma(t) \big)
  =&  \phi_{\zeta}\big(X(0), \Gamma(0) \big)
 + \int_0^t \overline{\mathcal{R}}_{\s;\zeta}\big(U(s), \G(s)  \big) \, ds
+ \int_0^t \overline{\mathcal{S}}_{\s;\zeta}\big(U(s),\G(s) \big) \, d W_s^Q
\end{split}
\end{align}
holds for all $0 \le t \le T$,
in which we have used $U(s) = X(s) + \Phi_{\mathrm{ref}}$.
\end{lem}
\begin{proof}
For convenience, we introduce the splitting
\begin{equation}
\phi_{\zeta}(X,\Gamma) = \phi_{1;\zeta}(X,\Gamma) + \phi_{2;\zeta}(\Gamma)
\end{equation}
with
\begin{equation}
\begin{array}{lcl}
\phi_{1;\zeta}\big( X , \Gamma \big)
  & = &
   \langle X, T_{\Gamma} \zeta
       \rangle_{H^{-1} ; H^1 },
\\[0.2cm]
\phi_{2;\zeta}\big( \Gamma \big)
 & = &
   \langle  \Phi_{\mathrm{ref}} - T_{\Gamma} \Phi_\s , T_{\Gamma} \zeta
       \rangle_{H^{-1} ; H^1 }
\\[0.2cm]
 & = &
     \langle T_{-\Gamma} \Phi_{\mathrm{ref}} 
      - \Phi_\s , \zeta
       \rangle_{L^2(\R,\R^n)}.
\end{array}
\end{equation}
We note that $\phi_{1;\zeta}$ and $\phi_{2;\zeta}$ are both $C^2$-smooth,
with derivatives given by
\begin{align}
\begin{split}
D \phi_{1;\zeta}(X, \Gamma)[\widetilde{X}, \widetilde{\Gamma}]&=D_1\phi_{1;\zeta}(X, \Gamma)[\widetilde{X}]
  +D_2 \phi_{1;\zeta}(X, \Gamma)[\widetilde{\Gamma}]
 \\
 &= \langle \widetilde{X} , T_{\Gamma} \zeta \rangle_{H^{-1}; H^1}
  - \widetilde{\Gamma} \langle X , T_{\gamma} \zeta'  \rangle_{H^{-1}; H^1} ,
 \\
D \phi_{2;\zeta}(\Gamma)[\widetilde{\Gamma}]
 &= - \widetilde{\Gamma} \langle \Phi_{\mathrm{ref}} , T_{\Gamma} \zeta'  \rangle_{L^2(\R,\R^n)} ,
\end{split}
\end{align}
together with
\begin{align}
\begin{split}
D^2 \phi_{1;\zeta}(X, \Gamma)[\widetilde{X}, \widetilde{\Gamma}][\widetilde{X}, \widetilde{\Gamma}]=
&D_{1}^2 \phi_{1;\zeta}(X, \Gamma)[\widetilde{X},\widetilde{X}]
+2D_{1,2} \phi_{1;\zeta}(X, \Gamma)[\widetilde{X},\widetilde{\Gamma}]\\
&+D_{2}^2 \phi_{1;\zeta}(X, \Gamma)[ \widetilde{\Gamma},\widetilde{\Gamma}] \\
 =&-2\widetilde{\Gamma} \langle \widetilde{X} , T_{\gamma} \zeta' \rangle_{H^{-1}; H^1}
   + \beta^2 \langle X , T_{\gamma} \zeta'' \rangle_{H^{-1}; H^1} ,
\\
D^2 \phi_{2;\zeta}( \Gamma)[ \widetilde{\Gamma}, \widetilde{\Gamma}]
 =& \widetilde{\Gamma}^2 \langle \Phi_{\mathrm{ref}} , T_{\Gamma} \zeta'' \rangle_{L^2}.
\end{split}
\end{align}

We hence see that
\begin{equation}
\begin{array}{lcl}
D \phi_{\zeta}\big(\mathbf{Z}(s) \big)[\mathbf{A}_\s\big(\mathbf{Z}(s)\big)]
& = & \langle \mathcal{K}_{A;\s} \big(U(s)\big) , T_{\G(s)} \zeta \rangle_{H^{-1};H^1}
\\[0.2cm]
& & \qquad
- \big[c_\s + \overline{a}_\s\big(U(s),\G(s)\big)\big] \langle U(s), T_{\G(s)} \zeta' \rangle_{L^2(\R,\R^n)}
\\[0.2cm]
D \phi_{\zeta}\big(\mathbf{Z}(s) \big)[\mathbf{B}\big(\mathbf{Z}(s)\big)w]
& = & \langle g \big(U(s)\big)[w] , T_{\G(s)} \zeta \rangle_{H^{-1};H^1}
\\[0.2cm]
& & \qquad
-  \overline{b}\big(U(s),\G(s)\big)[w] \langle U(s), T_{\G(s)} \zeta' \rangle_{L^2(\R,\R^n)}.
\end{array}
\end{equation}
Upon writing
\begin{equation}
\begin{array}{lcl}
\mathcal{I}_k( U, \Gamma)
& = & 
- 2 \overline{b}(U,\G\big)[\sqrt{Q} e_k]
  \langle g(U)[ \sqrt{Q} e_k], T_{\G} \zeta' \rangle_{L^2(\R,\R^n)}
  \\[0.2cm]
& & \qquad
  + \Big(\overline{b}(U,\G)[\sqrt{Q}e_k] \Big)^2
    \langle U, T_{\G} \zeta'' \rangle_{L^2(\R,\R^n)},
\\[0.2cm]
\end{array}
\end{equation}
we also observe that
\begin{equation}
      D^2 \phi_{\zeta}\big(\mathbf{Z}(s) \big)
\big[\mathbf{B}\big(\mathbf{Z}(s)\big)\sqrt{Q}e_k,
  \mathbf{B}\big(\mathbf{Z}(s)\big)\sqrt{Q}e_k
\big] = \mathcal{I}_k\big(U(s),\G(s)\big) .
\end{equation}
A short computation yields
\begin{equation}
\begin{array}{lcl}
\mathcal{I}_k(U,\G)
& = & 
2 \overline{\chi}_h(U,\G)^2
  \overline{\chi}_{l}(U,\G)
      \bip{ g(U) [\sqrt{Q} e_k], T_{\G} \psit }_{L^2(\R,\R^n)}
      \\[0.2cm]
      & & \qquad \qquad 
     \times \langle g\big(U\big) [\sqrt{Q} e_k], T_{\G} \zeta' \rangle_{L^2(\R,\R^n)}
  \\[0.2cm]
& & \qquad
  + \Big(\overline{b}(U,\G)[\sqrt{Q} e_k] \Big)^2
    \langle U, T_{\G} \zeta'' \rangle_{L^2(\R,\R^n)}
\\[0.2cm]
& = &
 2 \overline{\chi}_h(U,\G)^2
  \overline{\chi}_{l}(U,\G)
      \bip{  \sqrt{Q} e_k, g^\mathrm{adj}(U) T_{\G} \psit }_{L^2_Q}
    \\[0.2cm]
    & & \qquad \qquad
    \times \langle  \sqrt{Q} e_k, g^\mathrm{adj}(U) T_{\G} \zeta' \rangle_{L^2_Q}
  \\[0.2cm]
& & \qquad
  + \Big(\overline{b}(U,\G)[\sqrt{Q} e_k] \Big)^2
    \langle U, T_{\G} \zeta'' \rangle_{L^2(\R,\R^n)}.
\end{array}
\end{equation}
In particular, we see that
\begin{equation}
\begin{array}{lcl}
\sum_{k=0}^\infty \mathcal{I}_k(U,\G) & = &
 2 \overline{\chi}_h(U,\G)^2
  \overline{\chi}_{l}(U,\G)
  \bip{
  g^\mathrm{adj}(U) T_{\G} \psit,
  g^\mathrm{adj}(U) T_{\G} \zeta'
  }_{L^2_Q}
\\[0.2cm]
& & \qquad
 + \norm{\overline{b}(U,\G) }^2_{HS(L^2_Q , \R)}
    \langle U, T_{\G} \zeta'' \rangle_{L^2(\R,\R^n)},
\\[0.2cm]
\end{array}
\end{equation}
which yields
\begin{equation}
\begin{array}{lcl}
\sum_{k=0}^\infty \mathcal{I}_k\big(U(s),\G(s)\big) 
 & = &
- 2
  \bip{
  \mathcal{K}_C \big(U(s), \G(s)\big),
   T_{\G(s)} \zeta'
  }_{L^2(\R,\R^n)}
\\[0.2cm]
& & \qquad
 + \norm{\overline{b}\big(U(s),\G(s)\big) }^2_{HS(L^2_Q , \R)}
    \langle U(s), T_{\G(s)} \zeta'' \rangle_{L^2(\R,\R^n)}.
\end{array}
\end{equation}
The derivatives can now be transferred from $\zeta$
to yield the desired expression.
\end{proof}

\begin{cor}
\label{cor:sps:weak:form:with:zeta}
Suppose that (Hq), (HEq), (HDt), (HSt) and (HTw) all hold and pick a test-function $\zeta \in C^\infty_c(\Real,\Real^n)$.
Then for almost all $\omega \in \Omega$,
the map $V$
defined by \sref{eq:sps:def:V}
satisfies
the identity
\begin{equation}
\label{eq:sps:itg:weak:formulation}
  \begin{array}{lcl}
  \langle V(t) , \zeta \rangle_{L^2(\R,\R^n)}
   & = & \langle V(0) , \zeta \rangle_{L^2}
    + \int_0^t
     \langle \mathcal{R}_{\s}\big( V(s) \big)  , \zeta \rangle_{H^{-1}; H^1}   \, ds
    \\[0.2cm]
    & & \qquad
    + \sigma \int_0^t
        \langle \mathcal{S}_{\s}\big(V(s)\big)
          T_{-\Gamma(s)} dW^Q_s , \zeta \rangle_{L^2(\R,\R^n)}
  \end{array}
  \end{equation}
for all $0 \le t \le T$.
\end{cor}
\begin{proof}
In view of Lemma \ref{lem:sps:ito:comp},
the result  follows from \sref{eq:sps:vt:with:zeta:ip}
together with
\begin{equation}
\overline{\mathcal{R}}_{\s;\zeta}\big(U(s) ; \G(s)\big)
=  \overline{\mathcal{R}}_{\s;\zeta}( T_{-\G(s)} U(s), 0)
= \overline{\mathcal{R}}_{\s;\zeta}\big( \Phi_\s + V(s), 0 \big)
= \langle \mathcal{R}_{\sigma}\big(V(s) \big), \zeta \rangle_{H^{-1};H^1}
\end{equation}
and a similar identity involving $\mathcal{S}_{\s}$.
\end{proof}

\begin{proof}[Proof of Proposition \ref{prp:sps:props:v}]
As a preparation, we
modify the definition
\sref{eq:sps:bck:def:w:q:t} and 
define a new process $\widetilde{W}^Q_t$
via the formal sum
\begin{equation}
\widetilde{W}^Q_t = \sum_{k=0}^\infty \int_0^t T_{-\Gamma(s)} \sqrt{Q} e_k  d\beta_k(s).
\end{equation}
The estimates in
\cite[{\S}2]{Karczewska2005}
all remain valid since $T_{-\Gamma(s)}$
is an isometry. In particular,
we can replace the 
$T_{-\Gamma(s)} dW^Q_s$ term in \sref{eq:sps:itg:weak:formulation}
by $d \widetilde{W}^Q_s$. 
The proof of Prop. 5.1 in \cite{Hamster2017}
can then be readily applied
to the current setting, yielding
all the desired properties after
replacing $d W^Q_s$  by
$d \widetilde{W}^Q_s$ in \sref{eq:sps:int:eg:for:v}.

The key issue here is that
- by design - stochastic integrals
with respect to
$d W^Q_s$ and
$d \widetilde{W}^Q_s$
are indistinguishable
from each other in the sense
that they generate the same statistical properties.
To see this, we pick 
a Hilbert space $\mathcal{H}$ together with two processes 
\begin{equation}
    B_1, B_2 \in  \mathcal{N}^2\big( [0,T] ; (\mathcal{F}_t ) ;
   HS(L^2_Q, \mathcal{H}) \big)
\end{equation}
and consider the shifted inner product
\begin{equation}
\begin{array}{lcl}
\mathcal{I}_{1,2} & = &
E\,  \langle \int_0^t  B_1(s) \, d \widetilde{W}^Q_s , 
      \int_0^t B_2(s) \, d \widetilde{W}^Q_s \rangle_{\mathcal{H}}
\\[0.2cm]
& = & 
E\,  \langle \int_0^t  B_1(s) \,T_{-\Gamma(s)}  d W^Q_s , 
      \int_0^t B_2(s) \, T_{-\Gamma(s)} d W^Q_s \rangle_{\mathcal{H}}.
\end{array}
\end{equation}
The translational invariance of $\sqrt{Q}$ allows us to write
\begin{equation}
    T_{\gamma} \sqrt{Q} e_k = \sqrt{Q} T_{\gamma} e_k
\end{equation}
for any $\gamma \in \R$. In view of the fact that 
$(T_{\gamma}e_k)$ is also an orthonormal basis for $L^2(\R,\R^m)$,
we have
\begin{equation}
\begin{array}{lcl}
\langle B_1(s) T_{-\Gamma(s)} , B_2(s) T_{-\Gamma(s)} \rangle_{HS(L^2_Q,\mathcal{H} ) }
& = & 
\sum_{k=0}^\infty \langle B_1(s) T_{-\Gamma(s)} \sqrt{Q} e_k , B_2(s) T_{-\Gamma(s)} \sqrt{Q} e_k \rangle_{\mathcal{H}  }
\\[0.2cm]
& = & 
\sum_{k=0}^\infty \langle B_1(s)  \sqrt{Q}T_{-\Gamma(s)} e_k , B_2(s) \sqrt{Q} T_{-\Gamma(s)} e_k \rangle_{\mathcal{H}  }
\\[0.2cm]
& = & 
\langle B_1(s)  , B_2(s) 
   \rangle_{HS(L^2_Q, \mathcal{H} ) }
\end{array}
\end{equation}
for all $0 \le s \le t$.
The It\^o isometry \sref{eq:sps:ito:wrst:q:s}
hence allows us to compute
\begin{equation}
\begin{array}{lcl}
   \mathcal{I}_{1,2}
& = & 
   E \, \int_0^t \langle B_1(s) T_{-\Gamma(s)} , B_2(s) T_{-\Gamma(s)} \rangle_{HS(L^2_Q,\mathcal{H} ) } \, ds 
\\[0.2cm]
& = & 
   E \, \int_0^t \langle B_1(s)  , B_2(s) 
   \rangle_{HS(L^2_Q, \mathcal{H} ) } \, ds 
\\[0.2cm]
& = & 
E \, \langle \int_0^t  B_1(s) \,  d W^Q_s , 
      \int_0^t B_2(s) \,  d W^Q_s \rangle_{ \mathcal{H} }.
\end{array}
\end{equation}
Applying this with $\mathcal{H}= \R$,
$t = t_1 \wedge t_2$ and
\begin{equation}
B_1(s) = \mathbf{1}_{s < t_1} \langle \cdot , w_1 \rangle_{L^2(\R,\R^m)},
\qquad
\qquad
B_2(s) = \mathbf{1}_{s < t_2}\langle \cdot , w_2 \rangle_{L^2(\R,\R^m)},
\end{equation}
we recover
the familiar correlations
\begin{equation}
   E \Big[ \langle \widetilde{W}^Q_{t_1} , w_1 \rangle_{L^2(\R,\R^m)} \langle \widetilde{W}^Q_{t_2} , w_2 \rangle_{L^2(\R,\R^m)} \Big]
    = (t_1 \wedge t_2) \langle Q w_1, w_2 \rangle_{L^2(\R,\R^m)}.
\end{equation}
In view of \cite[Defn 2.1]{Dalang}, this means that
$\widetilde{W}^Q_t$ is also a $(\mathcal{F}_t , Q)$-cylindrical Wiener process. We therefore follow the convention in 
\cite[{\S}2.2.2]{Lord2012}
and drop the distinction between $W^Q_t$ and $\widetilde{W}^Q_t$.
\end{proof}

\section{Stability}
\label{sec:stb}
Our goal here is to provide
a rigorous formulation of the two stability
results provided in {\S}\ref{sec:mr:st}
and give a brief outline of their proofs.
Given our preparatory work in {\S}\ref{sec:sps}
and Appendix \ref{sec:app:est}, 
we can appeal to \cite{Hamster2018} for many
of the details. However, we will need
to generalize a stochastic time transformation
result to our setting of cylindrical $Q$-Wiener
processes.

 Given an initial condition $U_0 \in \mathcal{U}_{H^1}$
 that is sufficiently close to $\Phi_\s$, it is possible
 to find a corresponding $(\Gamma_0, V_0)$ so that
 $U_0 = T_{\Gamma_0} [ V_0 + \Phi_\s]$ with $\langle V_0, \psit \rangle = 0$;
 see \cite[Prop. 2.3]{Hamster2017}.
 Recalling the function $V$ defined by
 \begin{equation}
   \label{eq:stb:eqn:for:V}
     \begin{array}{lcl}
  V(t) & = & V(0) + \int_0^t
     \mathcal{R}_{\sigma}\big( V(s) \big)\, ds
    + \sigma \int_0^t
        \mathcal{S}_{\s}\big(V(s)\big)
       d W^Q_s ,
 \end{array}
 \end{equation}
 we fix a sufficiently small $\epsilon > 0$
 and introduce the scalar function
 \begin{equation}\label{eq:DefNeps}
N_{U_0} (t) =
\norm{V(t)}_{L^2(\R,\R^n)}^2
 + \int_0^t e^{- \e (t - s) }
    \norm{  V(s)}_{H^1(\R,\R^n)}^2 \, ds .
\end{equation}
In addition, for any $\eta > 0$ we introduce the
$(\mathcal{F}_t)$-stopping time
\begin{equation}
t_{\mathrm{st}}(U_0,T,\eta)
 = \inf\Big\{0 \leq t < T:
     N_{U_0}(t)
     > \eta
  \Big\} ,
\end{equation}
writing $t_{\mathrm{st}}(U_0, T,\eta) = T$
if the set is empty. 

The small (but fixed)
parameter $\eta > 0$ allows
us to keep the nonlinearities in the
problem under control. Our main technical
result provides a bound for $N_{U_0}$
in terms of the initial perturbation
and the noise strength.

\begin{prop}
\label{prp:nls:general}
Assume that (Hq), (HEq), (HDt), (HSt) and (HTw) are satisfied
and pick
two sufficiently small
constants $\delta_{\eta} > 0$
and $\delta_{\sigma} > 0$.
Then there exists a constant
$K > 0$ so that for any $T > 0$,
any $0 < \eta \le \delta_{\eta}$
and any $0 \le \sigma \le \delta_{\sigma}T^{-1/2}$
we have the bound
\begin{equation}
\label{eq:nls:prp:general:estimate}
\begin{array}{lcl}
E \, \big[\sup_{0\leq t\leq t_{\mathrm{st}}(T,\e,\eta) }N_{U_0}(t)\big]
&\leq &
   K \Big[ \nrm{V(0)}^2_{H^1}+ \sigma^2 T \Big] .
\end{array}
\end{equation}
\end{prop}

In a standard fashion, this bound
can be used to show that the probability
of hitting $\eta$ can be made arbitrarily small
by reducing the noise strength
and the size of the initial perturbation.
Indeed, upon writing
\begin{equation}
p(U_0,T,\eta) = P\Big(
 \sup_{0 \leq t \leq T} \big[N_{U_0}(t)\big]
 > \eta
\Big),
\end{equation} 
we can compute
\begin{equation}
\begin{array}{lcl}
\eta p(U_0,T,\eta)
& = & \eta P\big( t_{\mathrm{st}} < T\big)
\\[0.2cm]
& = & E \, \big[
      \mathbf{1}_{t_{\mathrm{st}} < T}
      N_{U_0}\big(t_{\mathrm{st}}\big)
    \big]
\\[0.2cm]
& \le & E \,  \big[N_{U_0}\big(t_{\mathrm{st}}  \big)\big]\\[0.2cm]
& \leq & E\,  \big[\sup_{0\leq t\leq t_{\mathrm{st}} }N_{U_0}(t)\big]
\\[0.2cm]
& \le & K \big[ \norm{V(0)}_{H^1}^2 + \sigma^2 T \big].
\end{array}
\end{equation} 
This is the rigorous interpretation
of the informal statement contained in
Theorem \ref{thm:mr:stb}.

We now set out to quantify the residual
resulting from the expansion process
outlined in {\S}\ref{sec:res}.
To this end, we 
take $U_0 = \Phi_\s$ (i.e. $V(0) = 0$)
and make the decomposition
$V(t) = V_{\mathrm{apx}}(t) + V_{\mathrm{res}}(t)$.
Here
\begin{equation}
    V_{\mathrm{apx}}(t) = \sigma V^1_\s(t) +  
      \sigma^2 V^2_\s(t)
\end{equation}
denotes the second order approximation
obtained formally in {\S}\ref{sec:mr:st}.
We subsequently introduce the
scalar quantity 
\begin{equation}\label{eq:DefNeps:res}
\begin{array}{lcl}
N_{\mathrm{res}} (t)
 &= &
\sigma^4 \norm{V_{\mathrm{apx}}(t)}_{L^2(\R,\R^n)}^2
+ \norm{V_{\mathrm{res}}(t)}_{L^2(\R,\R^n)}^2
\\[0.2cm]
& & \qquad
 + \int_0^t e^{- \e (t - s) }
    \Big[ \sigma^4 \norm{V_{\mathrm{apx}}(s)}^2_{H^1(\R,\R^n)}
     +
    \norm{  V_{\mathrm{res}}(s)}_{H^1(\R,\R^n)}^2 \Big]\, ds ,
\end{array}
\end{equation}
together with the
$(\mathcal{F}_t)$-stopping time
\begin{equation}
t_{\mathrm{st}}(T,\sigma,\eta;\mathrm{res})
 = \inf\Big\{0 \leq t < T:
     N_{\mathrm{res}}(t)
     >  \sigma^4 \eta
  \Big\} ,
\end{equation}
writing $t_{\mathrm{st}}(T,\sigma,\eta;\mathrm{res}) = T$
if the set is empty. 
Note that the scalings imply that
$V_{\mathrm{apx}}$ remains bounded by $\eta$ as long
as the stopping time is not hit,
which allows the nonlinear terms
to be controlled in the same fashion as
in the proof of Proposition \ref{prp:nls:general}.
Since all the quadratic terms
have now been accounted for,
we arrive at the following estimate.

\begin{cor}
\label{cor:nls:general:res}
Assume that (Hq), (HEq), (HDt), (HSt) and (HTw) are satisfied
and pick
two sufficiently small
constants $\delta_{\eta} > 0$
and $\delta_{\sigma} > 0$.
Then there exists a constant
$K > 0$ so that for any $T > 0$,
any $0 < \eta \le \delta_{\eta}$
and any $0 \le \sigma \le \delta_{\sigma}T^{-1/2}$,
we have the bound
\begin{equation}
\label{eq:nls:prp:general:estimate:res}
\begin{array}{lcl}
E\, \big[\sup_{0\leq t\leq t_{\mathrm{st}}(T,\s,\eta;\mathrm{res}) }N_{\mathrm{res}}(t) \big]
&\leq &
   K \sigma^6 T.
\end{array}
\end{equation}
\end{cor}

In order to turn this into a probability estimate,
we write
\begin{equation}
p_{\mathrm{res}}(T, \sigma,\eta) = P\Big(
 \sup_{0 \le t \le T}
 N_{\mathrm{res}}(t) >  \sigma^4 \eta
\Big) 
\end{equation}
and compute
\begin{equation}
\begin{array}{lcl}
\sigma^4 \eta p_{\mathrm{res}}(T,\s,\eta)
& = & \sigma^4 \eta P\big( t_{\mathrm{st}}(T,\sigma,\eta;\mathrm{res}) < T\big)
\\[0.2cm]
& = & E \, \big[
      \mathbf{1}_{t_{\mathrm{st}}(T,\sigma,\eta;\mathrm{res}) < T}
      N_{\mathrm{res}}\big(t_{\mathrm{st}}(T,\s,\eta;\mathrm{res})\big)
    \big]
\\[0.2cm]
& \le & K \sigma^6 T.
\end{array}
\end{equation}
In particular,
in the setting of Corollary
\ref{cor:nls:general:res}
we find
\begin{equation}
p_{\mathrm{res}}(T,\s, \eta) \le \eta^{-1} K \sigma^2 T,
\end{equation}
which is the quantitative
version of Corollary \ref{cor:mr:res:bnd}.

\subsection{Stochastic time transform}

We now set out to outline how
the techniques developed
in \cite{Hamster2018} can be used
to establish
Proposition \ref{prp:nls:general}.
The key issue is that we cannot
study \sref{eq:stb:eqn:for:V} or
its mild counterpart in a direction fashion
because it is a quasi-linear system.
The offending component is $\mathcal{K}_B$,
which represents an extra nonlinear
- but spatially homogeneous - 
diffusive term that arises
as a consequence of the It\^o lemma.

Our strategy is to partially eliminate these
terms by appropriate time transforms.
In particular, for each component
$1 \le i \le n$ we define the function
\begin{equation}
    \kappa_{\s;i}(V) =
    1 + \frac{1}{2 \rho_i} \norm{b_\s(V)}^2_{HS\big(L^2_Q , L^2(\R,\R^n)\big)}
\end{equation}
and observe that
$\rho_i \kappa_{\s;i}(V)$ corresponds
precisely with the coefficient in front
of $V_i''$ that appears in $\mathcal{R}_\s(V)$.
In order
to reset this single coefficient
to the value $\rho_i$,
we introduce the
(faster) transformed time
\begin{equation}
\tau_{i}(t) = \int_0^t
  \kappa_{\sigma;i}\big( V(s) \big) \, d s
  \ge t .
\end{equation}
The map $t \mapsto \tau_i(t)$
is a continuous
strictly increasing $(\mathcal{F}_t)$-adapted
process that hence admits an inverse $t_i(\tau)$, i.e.,
\begin{equation}
    \tau_i (t_i (\tau)\big)  = \tau,
    \qquad
    t_i\big( \tau_i(t) \big) = t.
\end{equation}
This allows us to define the time-transformed
function
\begin{equation}
\overline{V}_{(i)}(\tau) =
  V\big( t_i(\tau) \big) ,
\end{equation}
for which an appropriate SPDE can be derived.

\begin{lem}
Consider the setting of Proposition \ref{prp:sps:props:v}
and pick $1 \le i \le n$.
Then there exists a filtration
$(\overline{\mathcal{F}}_{\tau})_{\tau \ge 0}$
together with a cylindrical $(\overline{\mathcal{F}}_{\tau},Q)$-Wiener
process $\overline{W}^Q_{\tau}$
so that the following properties hold.
\begin{itemize}
\item[(i)]{
  For almost all $\omega \in \Omega$,
  the map $\tau \mapsto \overline{V}_{(i)}(\tau; \omega)$
  is of class $C\big([0,T]; L^2 \big)$.
}
\item[(ii)]{
  For all $\tau \in [0, T]$, the map
  $\omega \mapsto \overline{V}_{(i)}(\tau,\omega)$
  is $(\overline{\mathcal{F}}_{\tau})$-measurable.
}
\item[(iii)]{
  The map $\tau \mapsto \kappa_i^{-1/2}\big(\overline{V}_{(i)}(\tau)\big) \mathcal{S}_{\sigma}\big(\overline{V}_{(i)}(\tau)\big)$
   is of class  $\mathcal{N}^2\big([0,T];
          (\overline{\mathcal{F}})_{\tau} ; HS\big(L^2_Q, L^2(\R,\R^n) \big) \big)$.
}
\item[(iv)]{
  For almost all $\omega \in \Omega$,
  the identity
  \begin{equation}
  \label{eq:stb:spde:for:ovl:v}
  \begin{array}{lcl}
    \overline{V}_{(i)}(\tau)
     & = & \overline{V}_{(i)}(0)
     + \int_0^\tau \kappa_i^{-1}\big(\overline{V}_{(i)}(\tau') \big)
         \mathcal{R}_{\sigma}\big(\overline{V}_{(i)}(\tau')\big)
       
       \, d \tau'
\\[0.2cm]
& & \qquad \qquad \qquad
     + \sigma \int_0^\tau  \kappa_i^{-1/2}\big(\overline{V}_{(i)}(\tau')\big) \mathcal{S}_{\sigma}\big(\overline{V}_{(i)}(\tau')\big)
       \, d \overline{W}^Q_{\tau'}
  \end{array}
  \end{equation}
  holds for all $0 \le \tau \le T$.
}
\end{itemize}
\end{lem}
\begin{proof}
Recall the set of independent $(\mathcal{F}_t)$-Brownian
motions $\beta_k$ used to define $W^Q$ in 
{\S}\ref{sec:sps:bck}.
Following the proof of \cite[Lem. 6.2]{Hamster2017},
we now construct the processes 
\begin{equation}
\overline{\beta}_k(\tau)
= \int_0^\tau \frac{1}{\sqrt{\partial_\tau t_i(\tau')}}
    \,   d \beta_{k}\big(t_i(\tau')\big).
\end{equation}
These are independent Brownian
motions with respect to the
filtration $(\mathcal{F}_t)$
defined in \cite[Eq. (6.14)]{Hamster2017}.
As explained in {\S}\ref{sec:sps:bck},
the sum
\begin{equation}
    \overline{W}^Q_{\tau} = \sum_{k} \sqrt{Q} e_k  \overline{\beta}_k(\tau)
\end{equation}
hence defines a cylindrical $(\overline{\mathcal{F}}_{\tau},Q)$-Wiener
process. We can now apply the transformation rule \cite[Lem. 6.2]{Hamster2017}
for individual Brownian motions to compute
the desired transformation
\begin{equation}
\begin{array}{lcl}
    \int_0^{t_i(\tau)} \mathcal{S}_\s\big(V(s) \big) \, d W^Q_s
    & = & \lim_{m \to \infty} \sum_{k=1}^m \int_0^{t_i(\tau)}
      \mathcal{S}_\s\big(V(s)\big)[\sqrt{Q} e_k] \, d \beta_{k}(s)
    \\[0.2cm]
    & = &
      \lim_{m \to \infty} \sum_{k=1}^m
        \int_0^\tau \kappa^{-1/2}(\overline{V}_{(i)}(\tau')\big)
          \mathcal{S}_{\s}\big(\overline{V}_{(i)}(s) \big) [\sqrt{Q} e_k] \,
           d  \overline{\beta}_{k}(\tau')
\\[0.2cm]
  & = & \int_0^\tau \kappa^{-1/2}(\overline{V}_{(i)}(\tau')\big)
    \mathcal{S}_{\s}\big(\overline{V}_{(i)}(s) \big)
    \, d \overline{W}^Q_{\tau'}.
\end{array}
\end{equation}
The remaining statements can now be established as
in the proof of \cite[Prop. 6.3]{Hamster2017}.
\end{proof}

We remark that the diffusion
coefficient for the $i$-th component
of $\overline{V}_{(i)}$ is now again
equal to $\rho_i$. This allows
this component to be appropriately
estimated by analysing
the mild formulation of \sref{eq:stb:spde:for:ovl:v}.
The key here is that the off-diagonal 
elements of the semigroup $S(t)$ have
better smoothening properties than
the diagonal elements. Since all
the relevant estimates carry
over on account of {\S}\ref{sec:app:est},
the computations in \cite{Hamster2018}
can be used to 
establish Proposition \ref{prp:nls:general}.
and Corollary \ref{cor:nls:general:res}.


\appendix

\section{Estimates}
\label{sec:app:est}

In this section we set out to derive
certain key estimates that will 
build a bridge between
our setting here and the 
extensive computations in 
\cite{Hamster2017,Hamster2018}. The main issues
are that the functions $g$ and $b$
now need to be bounded in an appropriate
Hilbert-Schmidt norm and that
the term $\mathcal{K}_C$ has a more delicate
structure than its counterpart
in \cite{Hamster2017}.

Throughout this section, we will often use
a general pair $(\Phi, c)$ for our estimates,
since a-priori the wave $(\Phi_\s,c_\s)$ has not been constructed yet.
This pair is assumed to satisfy the following conditions.
\begin{itemize}
\item[(hPar)]{
 The condition (HTw) holds
 and the pair $(\Phi,c) \in \mathcal{U}_{H^1} \times \Real$
 satisfies the bounds
 \begin{equation}
    \norm{\Phi - \Phi_{0} }_{H^1(\R,\R^n)} \le
      \min\{ 1 ,[4 \norm{\psi_{\mathrm{tw}}}_{L^2(\R,\R^n)}]^{-1} \},
    \qquad
    \abs{c - c_0} \le 1.
  \end{equation}
}
\end{itemize}
We start in {\S}\ref{sec:est:prlm} by deriving
some preliminary estimates.
This will help us in {\S}\ref{sec:est:bridge} to formulate
the `bridge' estimates on the
three functions discussed above,
which concern both their size
and their Lipschitz properties.

\subsection{Preliminaries}
\label{sec:est:prlm}

On account of (Hq), the function 
$k \mapsto \sqrt{\widehat{q}(k)}$
is well-defined. It is hence tempting to construct
a convolution kernel $p$
for $\sqrt{Q}$ by
taking the inverse Fourier transform
of this map,

since then one formally has $q*v  = p*p*v$. Our first
result shows that this is indeed possible.

\begin{lem}\label{lem:app:sqrtQ}
Suppose that (Hq) is satisfied. Then the map
$k \mapsto \sqrt{\widehat{q}(k)}$ is contained in $L^2(\R,\R^{m \times m})$.
\end{lem}
\begin{proof}
It suffices to show that $\widehat{q} \in L^1(\R,\R^{m \times m})$, which
follows from the bound
\begin{align}\begin{split}
\nrm{\widehat{q}}_{L^1(\R,\R^{m \times m})}
&=\int_\R\frac{1}{(1+|k|^2)^{\frac{1}{2}}}(1+|k|^2)^{\frac{1}{2}}|\hat q(k)|dk
\leq K \norm{q}_{H^1(\R^{m \times m})}.
\end{split}
\end{align}
\end{proof}

Using this $L^2$-bound on $p$, one can now
show that any $z \in L^2(\R^{n \times m})$
can be interpreted as a Hilbert-Schmidt
operator from $L^2_Q$ into $L^2(\R,\R^n)$.
As usual, this proceeds via the pointwise
multiplication $z[w](x) = z(x) w(x)$.

\begin{lem}
\label{lem:est:bnd:hs:nrm:a}
Suppose that (Hq) is satisfied.
There exists $K > 0$ so that
for any $z \in L^2( \R,\R^{n \times m} )$,
we have $z \in HS\big( L^2_Q , L^2(\R,\R^n)\big)$ with
\begin{equation}
    \nrm{z}_{HS\big( L^2_Q , L^2(\R,\R^n)\big) }
     \le K \norm{z}_{L^2( \R,\R^{n \times m} )} .
\end{equation}
\end{lem}
\begin{proof}
Writing out the various matrix multiplications
in a component-wise fashion,
we obtain
\begin{equation}
    \begin{array}{lcl}
      \nrm{z}^2_{HS}
       &= &\sum_{k=0}^\infty \nrm{z[\sqrt Q e_k]}_{L^2(\R,\R^n)}^2
       \\
       & = & \sum_{k=0}^\infty \sum_{i=1}^n \sum_{j,j'=1}^m \int z_{ij}(x) \langle p_{j \cdot}(x - \cdot), e_k \rangle_{L^2(\R,\R^m)}
       z_{ij'}\langle p_{j' \cdot}(x - \cdot) ,e_k \rangle_{L^2(\R,\R^m)} \, dx
       \\[0.2cm]
       & = &  \sum_{i=1}^n \sum_{j,j'=1}^m \int z_{ij}(x) z_{ij'}(x)
        \langle p_{j \cdot}(x - \cdot), p_{j' \cdot}(x - \cdot)
          \rangle_{L^2(\R,\R^m)} \, dx
      \\[0.2cm]
        & = &
         \sum_{i=1}^n \sum_{j,j',l=1}^m\langle p_{j l}, p_{j' l} \rangle_{L^2( \Real)}
          \int z_{ij}(x) z_{ij'}(x)  \, dx
    \end{array}
\end{equation}
The result now follows by appealing to Cauchy-Schwarz.
\end{proof}

Our final two results concern
a bound on the cut-off functions \sref{eq:sps:def:cutoffs}
and a bound on the $L^2$-norm of $g$
that we borrow from  \cite{Hamster2017}.
This is especially useful when combined
with the bound in Lemma \ref{lem:est:bnd:hs:nrm:a}.

\begin{lem}
\label{lem:est:bnds:cutoff}
Suppose that (HEq), (Hg) and (hPar) are satisfied.
Then there exists a constant $K > 0$, which does not
depend on the pair $(\Phi, c)$, so that the following holds true.
For any $v \in H^1(\R,\R^n)$ and $\gamma \in \R$
we have the bound
\begin{equation}
 \label{eq:est:chi:l2:bnds}
    \begin{array}{lcl}
      \abs{\overline{\chi}_l(\Phi + v , \gamma) }
      + \abs{\overline{\chi}_h(\Phi + v , \gamma) }
      & \le &
        K ,
     \\[0.2cm]
    \end{array}
\end{equation}
while for any pair $(v_A, v_B ) \in H^1(\R,\R^n) \times H^1(\R,\R^n)$
and $(\gamma_A, \gamma_B) \in \Real^2$
we have the estimates
\begin{equation}
  \label{eq:est:chi:lip:bnds}
    \begin{array}{lcl}
       \abs{\overline{\chi}_l(\Phi + v_A , \gamma_A)
         - \overline{\chi}_l(\Phi + v_B, \gamma_B)}
         & \le &
        K \big[ \norm{v_A - v_B}_{L^2(\R,\R^n)} +
          (1 + \norm{v_A}_{L^2} )\abs{\gamma_1 - \gamma_2} \big] ,
       \\[0.2cm]
       \abs{\overline{\chi}_h(\Phi + v_A , \gamma_A)
         - \overline{\chi}_h(\Phi + v_B, \gamma_B)} & \le &
          K \big[ \norm{v_A - v_B}_{L^2(\R,\R^n)} + \abs{\gamma_A - \gamma_B}
          \big] .
    \end{array}
\end{equation}
\end{lem}
\begin{proof}
The bound \sref{eq:est:chi:l2:bnds} follows directly
from the definition of the cut-off functions.
The first Lipschitz bound in \sref{eq:est:chi:lip:bnds}
can be found in  \cite[Lem. 3.3]{Hamster2017},
while the second bound follows from the observation
\begin{equation}
    \abs{\norm{ \Phi + v_A - T_{\gamma_A} \Phi_{\mathrm{ref}} }_{L^2(\R,\R^n)}
    - \norm{ \Phi + v_B - T_{\gamma_B} \Phi_{\mathrm{ref}} }_{L^2(\R,\R^n)} }
     \le K\big[ \norm{ v_A - v_B}_{L^2(\R,\R^n)} + \abs{\gamma_A - \gamma_B} \big].
\end{equation}
\end{proof}

\begin{lem}
\label{lem:est:bnds:g:l2}
Suppose that (HEq), (Hg) and (hPar) are satisfied.
Then there exists a constant $K > 0$, which does not
depend on the pair $(\Phi, c)$, so that the following holds true.
For any $v \in H^1(\R,\R^n)$
we have the bounds
\begin{equation}
 \label{eq:prlm:g:l2:bnds}
    \begin{array}{lcl}
      \norm{g(\Phi + v)}_{L^2(\R^{n\times m})}
      & \le &
        K [1 + \norm{v}_{L^2(\R,\R^n)}] ,
     \\[0.2cm]
     \norm{\partial_\xi g(\Phi + v)}_{L^2(\R^{n\times m})}  & \le &
        K [1 + \norm{v}_{H^1(\R,\R^n)}] ,
    \end{array}
\end{equation}
while for any pair $(v_A, v_B ) \in H^1(\R,\R^n) \times H^1(\R,\R^n)$
we have the estimates
\begin{equation}
  \label{eq:lem:prlm:g:lipschitz}
    \begin{array}{lcl}
      \norm{g(\Phi + v_A) - g(\Phi + v_B) }_{L^2(\R^{n\times m})} & \le &
        K \norm{v_A - v_B}_{L^2(\R,\R^n)} ,
       \\[0.2cm]
       \norm{\partial_\xi [ g(\Phi + v_A) - g(\Phi + v_B) ]}_{L^2(\R^{n\times m})}
       & \le &
        K \big[ 1 + \norm{v_A}_{H^1(\R,\R^n)}  \big] \norm{v_A - v_B}_{H^1(\R,\R^n)} .
    \end{array}
\end{equation}
\end{lem}

\begin{proof}
This follows from lemma 3.2 in \cite{Hamster2017}.
\end{proof}

\subsection{Estimates for $g$, $b_\s$ and $\mathcal{K}_C$}
\label{sec:est:bridge}

By combining the estimates in Lemma's \ref{lem:est:bnd:hs:nrm:a} 
and \ref{lem:est:bnds:g:l2} above, we immediately
obtain bounds on $g(U)$ viewed as a pointwise
multiplication operator from  $L^2_Q$ into $L^2(\R,\R^n)$.
These correspond precisely with the $L^2$-bounds for the function
$g(U)$ itself, allowing the follow-up estimates
to be readily transferred from \cite{Hamster2017} to the current setting.

\begin{cor}
\label{cor:prlm:ests:g}
Suppose that (Hq), (HEq), (HSt) and (hPar) are satisfied.
Then there exists a constant $K > 0$,
which does not depend on $(\Phi,c)$ so that the following holds true.
For any $v \in H^1(\R,\R^n)$
we have the bounds
\begin{equation}
 \label{eq:prlm:g:l2:bnds:hs}
    \begin{array}{lcl}
      \norm{ g(\Phi + v)}_{HS\big(L^2_Q,L^2(\R,\R^n)\big)}  & \le &
        K [1 + \norm{v}_{L^2(\R,\R^n)}] ,
    \end{array}
\end{equation}
while for any pair $(v_A, v_B ) \in H^1(\R,\R^n) \times H^1(\R,\R^n)$
we have the estimates
\begin{equation}
  \label{eq:lem:prlm:g:lipschitz:hs}
    \begin{array}{lcl}
      \norm{ g(\Phi + v_A) -  g(\Phi + v_B) }_{HS\big(L^2_Q,L^2(\R,\R^n)\big)} & \le &
        K \norm{v_A - v_B}_{L^2(\R,\R^n)} .
    \end{array}
\end{equation}
\end{cor}

Turning to the nonlinearity $\mathcal{K}_C$ 
defined in \sref{eq:sps:defs:T:ABC},
our goal here is to derive estimates for
$\partial_\xi \mathcal{K}_C(U,\gamma)$
that are comparable to those obtained for the product 
$b(U,\gamma) \p_\xi g(U) $ in the context of \cite{Hamster2017},
where $b$ evaluates to a scalar.
To this end,
we introduce the auxiliary function
\begin{equation}
    \widetilde{\mathcal{K}}_C(U, \Gamma) =
       \overline{\chi}_l(U,\Gamma)
         \overline{\chi}_h(U,\Gamma) Q g^T(U)T_{\Gamma} \psit ,
\end{equation}
which in view of the identification \sref{eq:sps:def:g:adj:spec}
allows us to write
\begin{equation}
    \mathcal{K}_C(U,\gamma) = - \overline{\chi}_h(U,\Gamma) g(U) \widetilde{\mathcal{K}}_C(U, \Gamma).
\end{equation}
The strategy is to use the splitting
\begin{equation}
\label{eq:est:splitting:for:tc}
\begin{array}{lcl}
    \nrm{ \partial_\xi \mathcal{K}_C(U,\Gamma) }_{L^2(\R,\R^n)}
    & \le & \norm{ \overline{\chi}_h(U,\Gamma) \partial_\xi  g(U) }_{HS\big(L^2_Q , L^2(\R,\R^n)\big)}
      \norm{\widetilde{\mathcal{K}}_C(U, \Gamma) }_{L^2_Q}
\\[0.2cm]
  & & \qquad
      + \norm{\overline{\chi}_h(U,\Gamma) g(U) }_{HS\big(L^2_Q , L^2(\R,\R^n)\big)}
      \norm{\p_\xi \widetilde{\mathcal{K}}_C(U, \Gamma) }_{L^2_Q}
\end{array}
\end{equation}
together with its natural analogue for 
$\p_\xi [\mathcal{K}_C(U_A,\G_A) - \mathcal{K}_C(U_B,\G_B)]$.
The following two results provide bounds for the factors
in \sref{eq:est:splitting:for:tc}
that show that both products on the right hand side 
lead to similar expressions as those obtained
in \cite{Hamster2017}. In fact, we obtain slightly better
estimates as a consequence of a more refined use 
of the cutoff functions.

\begin{cor}
\label{cor:est:chig}
Suppose that (Hg), (HEq) and (hPar) are satisfied.
Then there exists a constant $K > 0$, which does not
depend on the pair $(\Phi, c)$, so that the following holds true.
For any $v \in H^1(\R,\R^n)$ and $\gamma \in \R$
we have the bounds
\begin{equation}
 \label{eq:est:chig:l2:bnds}
    \begin{array}{lcl}
      \norm{\overline{\chi}_h(\Phi + v , \gamma) g(\Phi + v)}_{L^2(\R^{n \times m})}
      & \le &
        K ,
     \\[0.2cm]
     \norm{\overline{\chi}_h(\Phi + v , \gamma) \partial_\xi g(\Phi + v)}_{L^2(\R^{n\times m})}
      & \le &
        K \big[ 1 + \norm{v}_{H^1(\R,\R^n)} \big] .
     \\[0.2cm]
    \end{array}
\end{equation}
In addition, for any pair $(v_A, v_B ) \in H^1(\R,\R^n) \times H^1(\R,\R^n)$
and $(\gamma_A, \gamma_B) \in \Real^2$,
the expression
\begin{equation}
    \Delta_{AB} \overline{\chi}_h g
     = \overline{\chi}_h(\Phi + v_A , \gamma_A)g(\Phi + v_A)
      - \overline{\chi}_h(\Phi + v_B , \gamma_B)g(\Phi + v_B)
\end{equation}
satisfies the estimates
\begin{equation}
  \label{eq:est:chig:lip:bnds}
    \begin{array}{lcl}
      \norm{\Delta_{AB} \overline{\chi}_h g}_{L^2(\R^{n \times m})} & \le &
        K \big[ \norm{v_A - v_B}_{L^2(\R,\R^n)} + \abs{\gamma_1 - \gamma_2} \big] ,
       \\[0.2cm]
       \norm{\partial_\xi \Delta_{AB} }_{L^2(\R^{n\times m})} & \le &
          K \big[ \norm{v_A - v_B}_{H^1(\R,\R^n)} + \abs{\gamma_A - \gamma_B}
          \big] \big[ 1 + \norm{v_A}_{H^1(\R,\R^n)} \big] .
    \end{array}
\end{equation}
\end{cor}
\begin{proof}
The estimates \sref{eq:est:chig:l2:bnds} follow
directly from Lemma \ref{lem:est:bnds:g:l2},
using the fact that the cut-off allows us to assume
an a-priori bound for $\nrm{v}_{L^2(\R,\R^n)}$.
The Lipschitz bounds \sref{eq:est:chig:lip:bnds}
can be obtained by writing
\begin{equation}
\begin{array}{lcl}
    \Delta_{AB} \overline{\chi}_h g & = & \big[\overline{\chi}_h( \Phi + v_A, \gamma_A)
      - \overline{\chi}_h( \Phi + v_B, \gamma_B)\big]
        g(\Phi + v_A)
    \\[0.2cm]
    & & \qquad
     +\overline{\chi}_h( \Phi + v_B, \gamma_B)
       \big[ g(\Phi + v_A) - g(\Phi + v_B) \big]
\end{array}
\end{equation}
and applying the results
from Lemma's \ref{lem:est:bnds:cutoff} and \ref{lem:est:bnds:g:l2}.
\end{proof}

\begin{lem}
\label{lem:est:bnd:wtc}
Suppose that (Hq), (HEq), (HSt) and (hPar) are satisfied.
Then there exists a constant $K > 0$,
which does not
depend on the pair $(\Phi, c)$, so that the following holds true.
For any $v \in H^1(\R,\R^n)$ and $\gamma \in \R$
we have the bounds
\begin{equation}
 \label{eq:est:wtc:l2:bnds}
    \begin{array}{lcl}
      \norm{\widetilde{\mathcal{K}}_C(\Phi + v,\gamma)}_{L^2_Q}
      & \le &
        K ,
     \\[0.2cm]
     \norm{\partial_\xi \widetilde{\mathcal{K}}_C(\Phi + v,\gamma)}_{L^2_Q}  & \le &
        K [1 + \norm{v}_{H^1(\R,\R^n)}] .
    \end{array}
\end{equation}
In addition, for any pair $(v_A, v_B ) \in H^1(\R,\R^n) \times H^1(\R,\R^n)$
and $(\gamma_A, \gamma_B) \in \Real^2$,
the expression
\begin{equation}
    \Delta_{AB}\widetilde{T}_C  
     = \widetilde{\mathcal{K}}_C(\Phi + v_A,\g_A)
        - \widetilde{\mathcal{K}}_C(\Phi + v_B,\g_B)
\end{equation}
satisfies the estimates
\begin{equation}
  \label{eq:est:wtc:lip:bnds}
    \begin{array}{lcl}
      \norm{\Delta_{AB}\widetilde{T}_C }_{L^2_Q} & \le &
        K \big[ \norm{v_A - v_B}_{L^2(\R,\R^n)} + \abs{\gamma_A -\gamma_B} \big],
       \\[0.2cm]
       \norm{\partial_\xi \Delta_{AB}\widetilde{T}_C}_{L^2_Q}
       & \le &
        K \big[ 1 + \norm{v_A}_{H^1(\R,\R^n)}  \big]
        \big[\norm{v_A - v_B}_{H^1(\R,\R^n)} +  \abs{\gamma_A -\gamma_B} \big].
    \end{array}
\end{equation}
\end{lem}
\begin{proof}
Note first
that for any $z \in H^1(\R,\R^{m \times n})$
and any $\psi \in W^{1, \infty}(\R,\R^n)$,
we have
\begin{equation}
   \norm{ Q z \psi}_{L^2_Q}^2 = \langle Q z\psi, z\psi \rangle_{L^2(\R,\R^m)}
    \le \norm{q}_{L^1} \norm{z}^2_{L^2( \R,\R^{n \times m} )}
    \norm{\psi}_\infty^2
\end{equation}
together with
\begin{equation}
\begin{array}{lcl}
   \norm{ \partial_\xi Q z\psi}_{L^2_Q}^2
   = \norm{  Q \partial_\xi [z\psi]}_{L^2_Q}^2
   & \le & \norm{q}_{L^1} \norm{\partial_\xi [z\psi]}_{L^2(\R,\R^m)}^2
   \\[0.2cm]
    & \le & \norm{q}_{L^1} \norm{z}^2_{H^1(\R,\R^{m \times n})}
    [\norm{\psi}_\infty + \norm{\psi'}_\infty]^2 .
\end{array}    
\end{equation}
The bounds \sref{eq:est:wtc:l2:bnds} hence follow
directly from Lemma \ref{lem:est:bnds:g:l2},
using the cut-off function again to eliminate the
dependence on $\norm{v}_{L^2(\R,\R^n)}$.

Turning to the Lipschitz estimates 
\sref{eq:est:wtc:lip:bnds},
we first compute
\begin{equation}
\begin{array}{lcl}
    \Delta_{AB} \widetilde{\mathcal{K}}_C & = &
    \big[\overline{\chi}_l( \Phi + v_A, \gamma_A)
      - \overline{\chi}_l( \Phi + v_B, \gamma_B)\big]
         Q \overline{\chi}_h( \Phi + v_A, \gamma_A)  g^T(\Phi + v_A)
         T_{\g_A} \psit
    \\[0.2cm]
    & & \qquad
      +\overline{\chi}_l( \Phi + v_B, \gamma_B)
                   Q \overline{\chi}_h( \Phi + v_A, \gamma_A)  g^T(\Phi + v_A)
                 \big[ T_{\g_A}\psit - T_{\g_B} \psit \big]
    \\[0.2cm]
    & & \qquad
     +\overline{\chi}_l( \Phi + v_B, \gamma_B)
        Q \big[\Delta_{AB} \overline{\chi}_h g \big]^T
         T_{\g_B} \psit .
\end{array}
\end{equation}
If $\overline{\chi}_h(\Phi + v_A, \gamma_A) \neq 0$,
then we can use an a-priori bound on $\norm{v_A}_{L^2(\R,\R^n)}$
to obtain the result directly
from Lemma \ref{lem:est:bnds:cutoff}
and Corollary \ref{cor:est:chig}.
On the other hand, if we have an a-priori bound on 
$\norm{v_B}_{L^2(\R,\R^n)}$, we can exploit symmetry
to replace the $\norm{v_A}_{L^2(\R,\R^n)}$  term in  
\sref{eq:lem:prlm:g:lipschitz} by $\norm{v_B}_{L^2(\R,\R^n)}$
and obtain the same result.
\end{proof}

We are now ready to consider the final
nonlinearity $\overline{b}$
that was defined in \sref{eq:sps:cutoffb}.
Fortunately, our estimates
for $\widetilde{\mathcal{K}_C}$ can also
be used to establish the following bounds,
which correspond precisely
to those obtained in 
\cite{Hamster2017}.

\begin{lem}
\label{lem:prlm:bnds:b}
Suppose that (Hq), (HEq), (Hg), (HSt) and (hPar) are satisfied.
Then there exist constants $K_b > 0$
and $K > 0$,
which do not depend on the pair 
$(\Phi,c)$ so that the following holds true. For any $v \in H^1(\R,\R^n)$ and $\gamma\in\Real$
we have the bound
\begin{equation}
  \label{eq:res:b:glb:bnd}
  \begin{array}{lcl}
      \nrm{\overline{b}(\Phi + \g, \psi)}_{HS(L_Q^2,\R)}  & \le &
        K_b,
     \\[0.2cm]
  \end{array}
\end{equation}
while for any set of pairs
$(v_A , v_B) \in H^1(\R,\R^n) \times H^1(\R,\R^n)$
and $(\g_A, \g_B) \in \Real^2$
we have the estimate
\begin{equation}
\label{eq:res:b:lip:bnd}
    \begin{array}{lcl}
      \nrm{\overline{b}(\Phi + v_A , \g_A) - \overline{b}(\Phi + v_B, \g_B) }_{HS(L^2_Q,\R)}
        & \le &
           K \norm{v_A-v_B}_{L^2(\R,\R^n)}
   \\[0.2cm]
   & & \qquad
           + K \big[ 1 +  \norm{v_B}_{L^2(\R,\R^n)} \big]
              \abs{\g_A - \g_B}.
    \end{array}
\end{equation}
\end{lem}
\begin{proof}
The computation \sref{eq:sps:hs:norm:ovl:b}
shows that
\begin{equation}
    \nrm{\overline{b}(\Phi + v, \g)}^2_{HS(L^2_Q,\R)}
    = \overline{\chi}_h(\Phi+V,\gamma)^2
    \nrm{\widetilde{\mathcal{K}}_C(\Phi + v_A , \gamma_A)}_{L^2_Q}^2,
\end{equation}
which on account
of \sref{eq:est:wtc:l2:bnds} 
immediately implies 
\sref{eq:res:b:glb:bnd}. 

Turning to the Lipschitz bound 
\sref{eq:res:b:lip:bnd},
we introduce
the notation
\begin{align}\begin{split}
\mathcal{I}_k &= \overline{b}(v_A+\Phi,\g_A)[\sqrt Qe_k]
-\overline{b}(v_B+\Phi,\g_B)[\sqrt Qe_k]\\
\end{split}
\end{align}
and note that
\begin{align}\begin{split}
\nrm{\overline{b}(\Phi+v_A,\g_A) - \overline{b}(\Phi+v_B, \g_B) }_{HS(L^2_Q,\R)}^2
=\sum_{k=0}^\infty \mathcal{I}_k^2.
\end{split}
\end{align}
We now compute
\begin{align}\begin{split}
\mathcal{I}_k 
&=\overline{\chi}_h(v_A + \Phi, \g_A)^2
\overline{\chi}_l(v_A + \Phi, \g_A)
\langle g(\Phi + v_A) \sqrt{Q} e_k, T_{\gamma_A} \psit \rangle
\\
& \qquad
- \overline{\chi}_h(V_B + \Phi, \g_B)^2
\overline{\chi}_l(V_B + \Phi, \g_B)
\langle g(\Phi + V_A) \sqrt{Q} e_k, T_{\gamma_B} \psit \rangle
\\
&=
\langle \sqrt{Q} e_k ,
  \overline{\chi}_h(v_A + \Phi, \g_A)\widetilde{\mathcal{K}}_C(\Phi + v_A , \gamma_A)
  - \overline{\chi}_h(v_B + \Phi, \g_B)\widetilde{\mathcal{K}}_C(\Phi + v_B , \gamma_B)
  \rangle_{L^2_Q}.
\\
\end{split}
\end{align}
In particular, we see that
\begin{equation}
    \sum_{k=0}^\infty \mathcal{I}_k^2 =
      \norm{\overline{\chi}_h(v_A + \Phi, \g_A)\widetilde{\mathcal{K}}_C(\Phi + v_A , \gamma_A)
  - \overline{\chi}_h(v_B + \Phi, \g_B)\widetilde{\mathcal{K}}_C(\Phi + v_B , \gamma_B)
      }^2_{L^2_Q}.
\end{equation}
The desired bound now follows
by combining
Lemma's \ref{lem:est:bnds:cutoff}
and \ref{lem:est:bnd:wtc}.
\end{proof}

\bibliographystyle{klunumHJ}
\bibliography{ref}

\end{document}

%% file: ricoS.pdf_tex
\begingroup%
  \makeatletter%
  \providecommand\color[2][]{%
    \errmessage{(Inkscape) Color is used for the text in Inkscape, but the package 'color.sty' is not loaded}%
    \renewcommand\color[2][]{}%
  }%
  \providecommand\transparent[1]{%
    \errmessage{(Inkscape) Transparency is used (non-zero) for the text in Inkscape, but the package 'transparent.sty' is not loaded}%
    \renewcommand\transparent[1]{}%
  }%
  \providecommand\rotatebox[2]{#2}%
  \newcommand*\fsize{\dimexpr\f@size pt\relax}%
  \newcommand*\lineheight[1]{\fontsize{\fsize}{#1\fsize}\selectfont}%
  \ifx\svgwidth\undefined%
    \setlength{\unitlength}{1391.25bp}%
    \ifx\svgscale\undefined%
      \relax%
    \else%
      \setlength{\unitlength}{\unitlength * \real{\svgscale}}%
    \fi%
  \else%
    \setlength{\unitlength}{\svgwidth}%
  \fi%
  \global\let\svgwidth\undefined%
  \global\let\svgscale\undefined%
  \makeatother%
  \begin{picture}(1,0.53962264)%
    \lineheight{1}%
    \setlength\tabcolsep{0pt}%
    \put(0,0){\includegraphics[width=\unitlength,page=1]{ricoS.pdf}}%
  \end{picture}%
\endgroup%

%% file: EV.pdf_tex
\begingroup%
  \makeatletter%
  \providecommand\color[2][]{%
    \errmessage{(Inkscape) Color is used for the text in Inkscape, but the package 'color.sty' is not loaded}%
    \renewcommand\color[2][]{}%
  }%
  \providecommand\transparent[1]{%
    \errmessage{(Inkscape) Transparency is used (non-zero) for the text in Inkscape, but the package 'transparent.sty' is not loaded}%
    \renewcommand\transparent[1]{}%
  }%
  \providecommand\rotatebox[2]{#2}%
  \newcommand*\fsize{\dimexpr\f@size pt\relax}%
  \newcommand*\lineheight[1]{\fontsize{\fsize}{#1\fsize}\selectfont}%
  \ifx\svgwidth\undefined%
    \setlength{\unitlength}{1391.25bp}%
    \ifx\svgscale\undefined%
      \relax%
    \else%
      \setlength{\unitlength}{\unitlength * \real{\svgscale}}%
    \fi%
  \else%
    \setlength{\unitlength}{\svgwidth}%
  \fi%
  \global\let\svgwidth\undefined%
  \global\let\svgscale\undefined%
  \makeatother%
  \begin{picture}(1,0.53962264)%
    \lineheight{1}%
    \setlength\tabcolsep{0pt}%
    \put(0,0){\includegraphics[width=\unitlength,page=1]{EV.pdf}}%
  \end{picture}%
\endgroup%

%% file: EV2.pdf_tex
\begingroup%
  \makeatletter%
  \providecommand\color[2][]{%
    \errmessage{(Inkscape) Color is used for the text in Inkscape, but the package 'color.sty' is not loaded}%
    \renewcommand\color[2][]{}%
  }%
  \providecommand\transparent[1]{%
    \errmessage{(Inkscape) Transparency is used (non-zero) for the text in Inkscape, but the package 'transparent.sty' is not loaded}%
    \renewcommand\transparent[1]{}%
  }%
  \providecommand\rotatebox[2]{#2}%
  \newcommand*\fsize{\dimexpr\f@size pt\relax}%
  \newcommand*\lineheight[1]{\fontsize{\fsize}{#1\fsize}\selectfont}%
  \ifx\svgwidth\undefined%
    \setlength{\unitlength}{1391.25bp}%
    \ifx\svgscale\undefined%
      \relax%
    \else%
      \setlength{\unitlength}{\unitlength * \real{\svgscale}}%
    \fi%
  \else%
    \setlength{\unitlength}{\svgwidth}%
  \fi%
  \global\let\svgwidth\undefined%
  \global\let\svgscale\undefined%
  \makeatother%
  \begin{picture}(1,0.53962264)%
    \lineheight{1}%
    \setlength\tabcolsep{0pt}%
    \put(0,0){\includegraphics[width=\unitlength,page=1]{EV2.pdf}}%
  \end{picture}%
\endgroup%

%% file: 2ndOrderCorNag.pdf_tex
\begingroup%
  \makeatletter%
  \providecommand\color[2][]{%
    \errmessage{(Inkscape) Color is used for the text in Inkscape, but the package 'color.sty' is not loaded}%
    \renewcommand\color[2][]{}%
  }%
  \providecommand\transparent[1]{%
    \errmessage{(Inkscape) Transparency is used (non-zero) for the text in Inkscape, but the package 'transparent.sty' is not loaded}%
    \renewcommand\transparent[1]{}%
  }%
  \providecommand\rotatebox[2]{#2}%
  \newcommand*\fsize{\dimexpr\f@size pt\relax}%
  \newcommand*\lineheight[1]{\fontsize{\fsize}{#1\fsize}\selectfont}%
  \ifx\svgwidth\undefined%
    \setlength{\unitlength}{1391.25bp}%
    \ifx\svgscale\undefined%
      \relax%
    \else%
      \setlength{\unitlength}{\unitlength * \real{\svgscale}}%
    \fi%
  \else%
    \setlength{\unitlength}{\svgwidth}%
  \fi%
  \global\let\svgwidth\undefined%
  \global\let\svgscale\undefined%
  \makeatother%
  \begin{picture}(1,0.53962264)%
    \lineheight{1}%
    \setlength\tabcolsep{0pt}%
    \put(0,0){\includegraphics[width=\unitlength,page=1]{2ndOrderCorNag.pdf}}%
  \end{picture}%
\endgroup%

%% file: EVfhn.pdf_tex
\begingroup%
  \makeatletter%
  \providecommand\color[2][]{%
    \errmessage{(Inkscape) Color is used for the text in Inkscape, but the package 'color.sty' is not loaded}%
    \renewcommand\color[2][]{}%
  }%
  \providecommand\transparent[1]{%
    \errmessage{(Inkscape) Transparency is used (non-zero) for the text in Inkscape, but the package 'transparent.sty' is not loaded}%
    \renewcommand\transparent[1]{}%
  }%
  \providecommand\rotatebox[2]{#2}%
  \newcommand*\fsize{\dimexpr\f@size pt\relax}%
  \newcommand*\lineheight[1]{\fontsize{\fsize}{#1\fsize}\selectfont}%
  \ifx\svgwidth\undefined%
    \setlength{\unitlength}{1391.25bp}%
    \ifx\svgscale\undefined%
      \relax%
    \else%
      \setlength{\unitlength}{\unitlength * \real{\svgscale}}%
    \fi%
  \else%
    \setlength{\unitlength}{\svgwidth}%
  \fi%
  \global\let\svgwidth\undefined%
  \global\let\svgscale\undefined%
  \makeatother%
  \begin{picture}(1,0.53962264)%
    \lineheight{1}%
    \setlength\tabcolsep{0pt}%
    \put(0,0){\includegraphics[width=\unitlength,page=1]{EVfhn.pdf}}%
  \end{picture}%
\endgroup%

%% file: EV2fhn.pdf_tex
\begingroup%
  \makeatletter%
  \providecommand\color[2][]{%
    \errmessage{(Inkscape) Color is used for the text in Inkscape, but the package 'color.sty' is not loaded}%
    \renewcommand\color[2][]{}%
  }%
  \providecommand\transparent[1]{%
    \errmessage{(Inkscape) Transparency is used (non-zero) for the text in Inkscape, but the package 'transparent.sty' is not loaded}%
    \renewcommand\transparent[1]{}%
  }%
  \providecommand\rotatebox[2]{#2}%
  \newcommand*\fsize{\dimexpr\f@size pt\relax}%
  \newcommand*\lineheight[1]{\fontsize{\fsize}{#1\fsize}\selectfont}%
  \ifx\svgwidth\undefined%
    \setlength{\unitlength}{1391.25bp}%
    \ifx\svgscale\undefined%
      \relax%
    \else%
      \setlength{\unitlength}{\unitlength * \real{\svgscale}}%
    \fi%
  \else%
    \setlength{\unitlength}{\svgwidth}%
  \fi%
  \global\let\svgwidth\undefined%
  \global\let\svgscale\undefined%
  \makeatother%
  \begin{picture}(1,0.53962264)%
    \lineheight{1}%
    \setlength\tabcolsep{0pt}%
    \put(0,0){\includegraphics[width=\unitlength,page=1]{EV2fhn.pdf}}%
  \end{picture}%
\endgroup%

%% file: GeneralQ.bbl
\begin{thebibliography}{000}

\bibitem{alexander1990topological}
J. Alexander, R. Gardner and C. Jones (1990), A topological invariant arising
  in the stability analysis of travelling waves.
\newblock {\em J. reine angew. Math} {\bf 410}(167-212), 143.

\bibitem{alili2005representations}
L. Alili, P. Patie and J.~L. Pedersen (2005), Representations of the first
  hitting time density of an Ornstein-Uhlenbeck process.
\newblock {\em Stochastic Models} {\bf 21}(4), 967--980.

\bibitem{Armero1996}
J. Armero, J. Sancho, J. Casademunt, A. Lacasta, L. Ramirez-Piscina, and F.
  Sagu{\'e}s (1996), External fluctuations in front propagation.
\newblock {\em Physical review letters} {\bf 76}(17), 3045.

\bibitem{Bernitt2017}
E. Bernitt and H.-G. D{\"o}bereiner (2017), Spatiotemporal Patterns of
  Noise-Driven Confined Actin Waves in Living Cells.
\newblock {\em Physical review letters} {\bf 118}(4), 048102.

\bibitem{Bressloff}
P.~C. Bressloff and M.~A. Webber (2012), Front propagation in stochastic neural
  fields.
\newblock {\em SIAM Journal on Applied Dynamical Systems} {\bf 11}(2),
  708--740.

\bibitem{carpenter1977geometric}
G.~A. Carpenter (1977), A geometric approach to singular perturbation problems
  with applications to nerve impulse equations.
\newblock {\em Journal of Differential Equations} {\bf 23}(3), 335--367.

\bibitem{Cartwright2018}
M. Cartwright and G.~A. Gottwald (2019), A collective coordinate framework to
  study the dynamics of travelling waves in stochastic partial differential
  equations.
\newblock {\em Physica D: Nonlinear Phenomena}.

\bibitem{Cerrai2005stabilization}
S. Cerrai (2005), Stabilization by noise for a class of stochastic
  reaction-diffusion equations.
\newblock {\em Probability theory and related fields} {\bf 133}(2), 190--214.

\bibitem{cornwell2017opening}
P. Cornwell (2017), Opening the Maslov Box for Traveling Waves in Skew-Gradient
  Systems.
\newblock {\em arXiv preprint arXiv:1709.01908}.

\bibitem{cornwell2017existence}
P. Cornwell and C.~K. Jones (2018), On the Existence and Stability of Fast
  Traveling Waves in a Doubly Diffusive FitzHugh--Nagumo System.
\newblock {\em SIAM Journal on Applied Dynamical Systems} {\bf 17}(1),
  754--787.

\bibitem{DaPratomild}
G. Da~Prato, A. Jentzen and M. R{\"o}ckner (2019), A mild It{\^o} formula for
  SPDEs.
\newblock {\em Transactions of the American Mathematical Society}.

\bibitem{DaPratomildArx}
G. Da~Prato, A. Jentzen and M. Roeckner (2010), A mild Ito formula for SPDEs.
\newblock {\em arXiv preprint arXiv:1009.3526}.

\bibitem{Dalang}
R.~C. Dalang and L. Quer-Sardanyons (2011), Stochastic integrals for
  spdeâ€™s: a comparison.
\newblock {\em Expositiones Mathematicae} {\bf 29}(1), 67--109.

\bibitem{evans2012introduction}
L.~C. Evans (2012), {\em An introduction to stochastic differential equations},
  Vol.~82.
\newblock American Mathematical Soc.

\bibitem{Garcia2001}
J. Garc{\'\i}a-Ojalvo, F. Sagu{\'e}s, J.~M. Sancho and L. Schimansky-Geier
  (2001), Noise-enhanced excitability in bistable activator-inhibitor media.
\newblock {\em Physical Review E} {\bf 65}(1), 011105.

\bibitem{GarciaSpatiallyExtended}
J. Garc{\'\i}a-Ojalvo and J. Sancho (2012), {\em Noise in spatially extended
  systems}.
\newblock Springer Science \& Business Media.

\bibitem{Hairernotes}
M. Hairer (2009), An Introduction to Stochastic PDEs.
\newblock \url{http://www.hairer.org/notes/SPDEs.pdf}.

\bibitem{Hairer2015}
M. Hairer and {\'E}. Pardoux (2015), A Wong-Zakai theorem for stochastic PDEs.
\newblock {\em Journal of the Mathematical Society of Japan} {\bf 67}(4),
  1551--1604.

\bibitem{Hamster2018}
C.~H.~S. Hamster and H.~J. Hupkes (2018), Stability of Traveling Waves for
  Systems of Reaction-Diffusion Equations with Multiplicative Noise.
\newblock {\em arXiv preprint arXiv:1808.04283}.

\bibitem{Hamster2017}
C.~H.~S. Hamster and H.~J. Hupkes (2019), Stability of Traveling Waves for
  Reaction-Diffusion Equations with Multiplicative Noise.
\newblock {\em SIAM Journal on Applied Dynamical Systems} {\bf 18}(1),
  205--278.

\bibitem{hastings1976travelling}
S. Hastings (1976), On travelling wave solutions of the Hodgkin-Huxley
  equations.
\newblock {\em Archive for Rational Mechanics and Analysis} {\bf 60}(3),
  229--257.

\bibitem{Inglis}
J. Inglis and J. MacLaurin (2016), A general framework for stochastic traveling
  waves and patterns, with application to neural field equations.
\newblock {\em SIAM Journal on Applied Dynamical Systems} {\bf 15}(1),
  195--234.

\bibitem{jones1991construction}
C. Jones, N. Kopell and R. Langer (1991), Construction of the FitzHugh-Nagumo
  pulse using differential forms.
\newblock In: {\em Patterns and dynamics in reactive media}.
\newblock Springer, pp. 101--115.

\bibitem{jones1984stability}
C.~K. Jones (1984), Stability of the travelling wave solution of the
  FitzHugh-Nagumo system.
\newblock {\em Transactions of the American Mathematical Society} {\bf 286}(2),
  431--469.

\bibitem{jones1995geometric}
C.~K. Jones (1995), Geometric singular perturbation theory.
\newblock In: {\em Dynamical systems}.
\newblock Springer, pp. 44--118.

\bibitem{Kadar1998}
S. Kadar, J. Wang and K. Showalter (1998), Noise-supported travelling waves in
  sub-excitable media.
\newblock {\em Nature} {\bf 391}(6669), 770.

\bibitem{Kapitula}
T. Kapitula and K. Promislow (2013), {\em Spectral and dynamical stability of
  nonlinear waves}.
\newblock Springer.

\bibitem{Karczewska2005}
A. Karczewska (2005), Stochastic integral with respect to cylindrical Wiener
  process.
\newblock {\em arXiv preprint math/0511512}.

\bibitem{StannatKruger}
J. Kr{\"u}ger and W. Stannat (2017), A multiscale-analysis of stochastic
  bistable reaction--diffusion equations.
\newblock {\em Nonlinear Analysis} {\bf 162}, 197--223.

\bibitem{KuehnReview}
C. Kuehn (2019), Travelling Waves in Monostable and Bistable Stochastic Partial
  Differential Equations.
\newblock {\em arXiv preprint arXiv:1904.03037}.

\bibitem{Lang}
E. Lang (2016), A multiscale analysis of traveling waves in stochastic neural
  fields.
\newblock {\em SIAM Journal on Applied Dynamical Systems} {\bf 15}(3),
  1581--1614.

\bibitem{LangThesis}
E. Lang (2016), Traveling waves in stochastic neural fields.

\bibitem{LiuRockner}
W. Liu and M. R{\"o}ckner (2010), {SPDE} in {H}ilbert space with locally
  monotone coefficients.
\newblock {\em Journal of Functional Analysis} {\bf 259}(11), 2902--2922.

\bibitem{Lord2012}
G. Lord and V. Th{\"u}mmler (2012), Computing stochastic traveling waves.
\newblock {\em SIAM Journal on Scientific Computing} {\bf 34}(1), B24--B43.

\bibitem{Lord2014Book}
G.~J. Lord, C.~E. Powell and T. Shardlow (2014), {\em An introduction to
  computational stochastic PDEs}.
\newblock Cambridge University Press.

\bibitem{lorenzi2004analytic}
L. Lorenzi, A. Lunardi, G. Metafune and D. Pallara (2004), Analytic semigroups
  and reaction-diffusion problems.
\newblock In: {\em Internet Seminar}, Vol. 2005.
\newblock p. 127.

\bibitem{Mueller1995}
C. Mueller and R.~B. Sowers (1995), Random travelling waves for the KPP
  equation with noise.
\newblock {\em Journal of Functional Analysis} {\bf 128}(2), 439--498.

\bibitem{nobile1985exponential}
A. Nobile, L. Ricciardi and L. Sacerdote (1985), Exponential trends of
  Ornstein--Uhlenbeck first-passage-time densities.
\newblock {\em Journal of Applied Probability} {\bf 22}(2), 360--369.

\bibitem{Peszat1997}
S. Peszat and J. Zabczyk (1997), Stochastic evolution equations with a
  spatially homogeneous Wiener process.
\newblock {\em Stochastic Processes and their Applications} {\bf 72}(2),
  187--204.

\bibitem{pickands1969}
J. Pickands (1969), Asymptotic properties of the maximum in a stationary
  Gaussian process.
\newblock {\em Transactions of the American Mathematical Society} {\bf 145},
  75--86.

\bibitem{DaPratoZab}
G. Prato and J. Zabczyk (1992), {\em Stochastic equations in infinite
  dimensions}.
\newblock Cambridge University Press, Cambridge New York.

\bibitem{Concise}
C. Pr{\'e}v{\^o}t and M. R{\"o}ckner (2007), {\em A concise course on
  stochastic partial differential equations}, Vol. 1905.
\newblock Springer.

\bibitem{revuz2013continuous}
D. Revuz and M. Yor (2013), {\em Continuous martingales and Brownian motion},
  Vol. 293.
\newblock Springer Science \& Business Media.

\bibitem{ricciardi1988first}
L.~M. Ricciardi and S. Sato (1988), First-passage-time density and moments of
  the Ornstein-Uhlenbeck process.
\newblock {\em Journal of Applied Probability} {\bf 25}(1), 43--57.

\bibitem{Shardlow}
T. Shardlow (2005), Numerical simulation of stochastic PDEs for excitable
  media.
\newblock {\em Journal of computational and applied mathematics} {\bf 175}(2),
  429--446.

\bibitem{StannatNag}
W. Stannat (2013), Stability of travelling waves in stochastic {N}agumo
  equations.
\newblock {\em arXiv preprint arXiv:1301.6378}.

\bibitem{Twardowska1996}
K. Twardowska (1996), Wong-Zakai approximations for stochastic differential
  equations.
\newblock {\em Acta Applicandae Mathematica} {\bf 43}(3), 317--359.

\bibitem{vanKampen}
N. Van~Kampen (1981), It{\^o} versus Stratonovich.
\newblock {\em Journal of Statistical Physics} {\bf 24}(1), 175--187.

\bibitem{Zumbrun2009}
K. Zumbrun (2011), Instantaneous {S}hock {L}ocation and {O}ne-{D}imensional
  {N}onlinear {S}tability of {V}iscous {S}hock {W}aves.
\newblock {\em Quarterly of applied mathematics} {\bf 69}(1), 177--202.

\end{thebibliography}
